\documentclass{article}

\usepackage[english]{babel}

\usepackage[letterpaper,top=2cm,bottom=2cm,left=3cm,right=3cm,marginparwidth=1.75cm]{geometry}

\usepackage{bm}
\usepackage{amsmath}
\usepackage{amssymb}
\usepackage{amsthm}
\usepackage{graphicx}
\usepackage{parallel}
\usepackage{ stmaryrd }
\usepackage{tikz}
\usetikzlibrary{calc}
\usepackage{multicol}
\usepackage{ stmaryrd }
\usepackage{ latexsym }
\usepackage[shortlabels]{enumitem}
\usepackage{hyperref}
\hypersetup{colorlinks,allcolors=blue,breaklinks=true}
\usepackage{lineno}

\newcommand{\forkindep}[1][]{
  \mathrel{
    \mathop{
      \vcenter{
        \hbox{\oalign{\noalign{\kern-.3ex}\hfil$\vert$\hfil\cr
              \noalign{\kern-.7ex}
              $\smile$\cr\noalign{\kern-.3ex}}}
      }
    }\displaylimits_{#1}
  }
}

\newcommand{\mipo}{\operatorname{MiPo}}
\newcommand{\Char}{\operatorname{Char}}

\newcommand{\MM}{\mathbb{M}}
\newcommand{\NN}{\mathbb{N}}

\newcommand{\QQ}{\mathbb{Q}}

\newcommand{\FF}{\mathbb{F}}

\newcommand{\VV}{\mathbb{V}}

\newcommand{\Kp}[1]{K[X]_{\operatorname{irr}}^{#1}}
\newcommand{\Cc}{{\mathcal{C}}}
\newcommand{\Ccalg}{{\mathcal{C}^{\operatorname{alg}}}}
\newcommand{\Cctrans}{{\mathcal{C}^{\operatorname{tr}}}}

\newcommand{\mm}{\mathcal{M}}

\newcommand{\jj}{\mathcal{J}}
\newcommand{\kk}{\mathcal{K}}
\newcommand{\rr}{\mathcal{R}}

\newcommand{\set}[1]{{\{#1\}}}
\newcommand{\Set}[1]{{\left\{#1\right\}}}
\newcommand{\spanA}[2]{{{\langle#1\rangle_{#2}}}}

\newcommand{\Fac}{{\operatorname{Fac}}}
\newcommand{\Ker}{{\operatorname{Ker}}}
\newcommand{\ACF}{{\operatorname{ACF}}}
\newcommand{\Image}{{\operatorname{Im}}}

\newcommand{\rk}{{\operatorname{rk}}}

\newcommand{\Diag}{{\operatorname{Diag}}}
\newcommand{\Id}{{\operatorname{Id}}}

\newcommand{\Lex}{{\operatorname{Lex}}}
\newcommand{\acl}{{\operatorname{acl}}}

\newcommand{\ua}{{\underline{a}}{}}
\newcommand{\ub}{{\underline{b}}{}}

\newcommand{\ud}{{\underline{d}}{}}

\newcommand{\ur}{{\underline{r}}{}}

\newcommand{\ut}{{\underline{t}}{}}
\newcommand{\uu}{{\underline{u}}{}}
\newcommand{\uv}{{\underline{v}}{}}
\newcommand{\uw}{{\underline{w}}{}}
\newcommand{\ux}{{\underline{x}}{}}
\newcommand{\uy}{{\underline{y}}{}}

\newcommand{\uz}{{\underline{z}}{}}

\newcommand{\uzero}{{\underline{0}}}
\newcommand{\ulambda}{{\underline{\lambda}}}

\newcommand{\umu}{{\underline{\mu}}}

\newcommand{\utau}{{\underline{\tau}}}

\newcommand{\li}{{\scalebox{0.5}{{$\operatorname{li}$}}}}
\newcommand{\ld}{{\scalebox{0.5}{{$\operatorname{ld}$}}}}

\newcommand{\sh}{{\scalebox{0.5}{{$\operatorname{sh}$}}}}
\newcommand{\va}{{\scalebox{0.5}{{$\operatorname{va}$}}}}

\newcommand{\lii}{{\scalebox{0.5}{$\operatorname{li}$}}}
\newcommand{\ldd}{{\scalebox{0.5}{$\operatorname{ld}$}}}

\newcommand{\tiluw}{{\Tilde{\underline{w}}}}

\newcommand{\uxvec}{{\underline{\Vec{x}}}}

\newcommand{\xvec}{{\Vec{x}}}

\newcommand{\tiluy}{{\underline{\Tilde{y}}}{}}

\newcommand{\Hfour}{{$(\operatorname{H4})$}}

\newcommand{\TKvs}{{T_{K\operatorname{-vs}}}}
\newcommand{\TKvsThe}{{T_{K\operatorname{-vs},\theta}}}
\newcommand{\TKvsTheC}{{T^C_{K\operatorname{-vs},\theta}}}

\newcommand{\RCF}{{\operatorname{RCF}}}

\newcommand{\tp}{{\operatorname{tp}}}

\newcommand{\LK}{{L_K}}

\newcommand{\TPtwo}{$\operatorname{TP}_2$}

\newcommand{\NATP}{$\operatorname{NATP}$}
\newcommand{\ATP}{$\operatorname{ATP}$}

\newcommand{\SOPone}{$\operatorname{SOP}_1$}
\newcommand{\SOPtwo}{$\operatorname{SOP}_2$}
\newcommand{\NSOPone}{$\operatorname{NSOP}_1$}
\newcommand{\NSOPtwo}{$\operatorname{NSOP}_2$}

\newcommand{\LKThe}{{L_{K,\theta}}}
\newcommand{\LRC}{{L_{R_C}}}

\newcommand{\Lr}{L_{\operatorname{r}}}

\newcommand{\emptyseq}{{\scalebox{0.5}{$\langle\rangle$}}}
\newcommand{\seq}[1]{{\scalebox{0.5}{$\langle #1 \rangle$}}}
\newcommand{\bigemptyseq}{{\langle\rangle}}
\newcommand{\bigseq}[1]{{\langle #1 \rangle}}

\newcommand{\TrafoLemma}{\hyperref[lemma_trafoo]{Transformation Lemma}}

\newcommand{\placeholderNotation}{\hyperref[def_placeholder_notation]{Placeholder-Notation}}

\newcommand{\commonBaseTheorem}{\hyperref[theorem_r_c_element]{Common Base Theorem}}

\newtheorem*{notation}{Notation}

\newtheorem{theorem}{Theorem}[section]
\newtheorem*{theorem*}{Theorem}
\newtheorem{theoremi}{Theorem}

\newtheorem{definition}[theorem]{Definition}
\newtheorem{fact}[theorem]{Fact}
\newtheorem{remark}[theorem]{Remark}
\newtheorem{lemma}[theorem]{Lemma}
\newtheorem{corollary}[theorem]{Corollary}
\newtheorem{example}[theorem]{Example}
\newtheorem{observation}[theorem]{Observation}

\newtheorem{notationnumber}[theorem]{Notation}

\newtheorem{subclaim}{Claim}[theorem]
\newtheorem{subdefinition}[subclaim]{Definition}
\newtheorem{subobservation}[subclaim]{Observation}
\newtheorem*{subdefinition*}{Definition}
\newtheorem*{subclaim*}{Claim}

\newenvironment{innerproof}[1][Proof]{\begin{proof}[#1]}{\end{proof}}

\usepackage{titlesec}
\titleformat{\section}
  {\normalfont\Large\bfseries\boldmath}{\thesection}{1em}{}
\titleformat{\subsection}
  {\normalfont\large\bfseries\boldmath}{\thesubsection}{1em}{}
\titleformat{\subsubsection}
  {\normalfont\normalsize\bfseries\boldmath}{\thesubsubsection}{1em}{}

\title{Model Theory of Generic Vector Space Endomorphisms IV: Preservation of NATP}
\author{Leon Chini}
\date{July 21, 2026}

\newcommand{\Addresses}{{
  \bigskip
  \footnotesize

  \textsc{Mathematisches Institut, Universität Bonn, Endenicher Allee 60, D-53115 Bonn, Germany}\par\nopagebreak
  \textit{E-mail address}: 
  \href{mailto:lchini@uni-bonn.de}{\tt lchini@uni-bonn.de}
}}

\begin{document}
\maketitle

\begin{abstract}
\noindent \noindent This paper further studies the model companion of an endomorphism acting on a vector space, possibly with extra structure. Let $T$ be a model-complete theory that $\varnothing$-defines an infinite $K$-vector space $\mathbb{V}$. In previous work, we introduced a family $\{T^C_\theta : C \in \mathcal{C}\}$ of extensions of the theory $T_\theta := T \cup \{\text{``$\theta$ is an endomorphism of $\mathbb{V}$''}\}$ that parameterizes all consistent extensions of the form  
$$  
    T_\theta \cup \left\{\sum\nolimits_{k}\bigcap\nolimits_{l}\operatorname{Ker}(\rho_{j, k, l}[\theta]) = \sum\nolimits_{k}\bigcap\nolimits_{l} \operatorname{Ker}(\eta_{j, k, l}[\theta]) : j \in \mathcal{J}\right\},  
$$  
where all sums and intersections are finite, all the $\rho[\theta]$'s and $\eta[\theta]$'s are polynomials over $K$ with $\theta$ plugged in, and $\mathcal{J}$ is some possibly infinite index set. We also presented a sufficient condition that implies that every $T^C_\theta$ has a model companion $T\theta^C$. In this paper, we show that, under this sufficient condition, the model companion $T\theta^C$ has $\operatorname{NATP}$, a neostability property recently introduced by Ahn and Kim, whenever $T$ does.
\end{abstract}

\tableofcontents
\section{Introduction}
\noindent This paper is a continuation of \cite{Chi25}, \cite{Chi25b}, and \cite{Chi26}, which deal with the model companion of an endomorphism acting on a vector space, possibly with extra structure. For the relevance of this line of inquiry and a description of earlier work, see the introduction of \cite{Chi25}. The goal of this paper is to show that this construction preserves the neostability property \NATP{} under the same assumption that implied companionability in \cite{Chi25}.

Recently, there has been a lot of interest in mutual generalizations of \NSOPone{} and $\operatorname{NTP}_2$.
Perhaps the most well-studied one is \NATP{}.
Here, \ATP{} stands for \textit{antichain tree property}, so \NATP{} stands for not having the antichain tree property.
This notion was first introduced by Ahn and Kim in \cite{AK24}.
We say that a formula $\varphi(\ux; \uy)$ has the antichain tree property if there is a binary tree of parameters $(\ub_\nu : \nu \in 2^{<\omega})$ such that $\set{\varphi(\ux; \ub_\nu) : \nu \in X}$ is consistent if and only if $X \subseteq 2^{<\omega}$ is an antichain, i.e., if $X$ contains no path of length $2$.
A theory has \ATP{} if some formula has \ATP{}.
It was shown, for example, in \cite{AKL23}, that the antichain tree property is always witnessed by a formula in a single free variable.
It was also shown there that, in order to show that a theory has \ATP{}, it is enough to find a formula $\varphi(\ux; \uy)$ and a tree $(\ub_\nu : \nu \in 2^{<\omega})$ such that antichains are consistent and paths of some fixed length $k \geq 2$ are inconsistent.
In the same paper, it was shown that \NATP{} is preserved by Mekler's construction for groups, that both Skolem arithmetic and the theory of atomless Boolean algebras have the antichain tree property, and that there are \NATP{} theories with both \TPtwo{} and \SOPone{}.
In \cite{AKLL25}, it was shown that several constructions preserve \NATP{}, among them lovely pairs for geometric theories and certain dense-codense expansions of vector spaces.
It was also shown in \cite{Han25} that \NATP{} theories satisfy a ``strongly bi-invariant Kim's Lemma''.

Such generalizations of Kim's Lemma play an important role in the study of mutual generalizations of \NSOPone{} and $\operatorname{NTP}_2$.
The original version of Kim's Lemma states that, in a simple theory, a formula $\varphi(\ux; \ub)$ divides over a set $A$ if and only if it divides along every Morley sequence over $A$ in $\tp(\ub/A)$.
Kim showed that Kim's Lemma characterizes simple theories \cite{Kim01}.
Similarly, variants of Kim's Lemma characterize both \NSOPone{} and $\operatorname{NTP}_2$ theories; see \cite{KR20} and \cite{Che14}.
In \cite{KR24b}, Kruckman and Ramsey generalized these two versions of Kim's Lemma to what is now often called \textit{New Kim's Lemma}.
In an attempt to find a syntactic characterization of theories satisfying New Kim's Lemma, they introduced the property $\operatorname{NBTP}$, which stands for not the \textit{bizarre tree property}.
Another syntactic property, called $\operatorname{NCTP}$, stands for not the \textit{comb tree property}.
It was introduced by Mutchnik in \cite{Mut26} under a slightly different name; here, we say that a theory has $\operatorname{NCTP}$ if it does not have $k$-$\operatorname{DCTP}_2$, as in Definition 4.8 in \cite{Mut26}, for any $k \geq 2$.
This property was shown to be equivalent to yet another generalization of Kim's Lemma by Hanson in \cite{Han25}.
Both the bizarre tree property ($\operatorname{BTP}$) and the comb tree property ($\operatorname{CTP}$) are, similarly to \ATP{}, defined by the consistency or inconsistency of certain partial types $\set{\varphi(\ux; \ub_\nu) : \nu \in X}$, where the parameters are indexed by tree structures.
It was also shown that $\operatorname{NBTP}$ and $\operatorname{NCTP}$ are preserved by Mekler's construction \cite{AK26}.
Currently, the following implications are known:
$$
\begin{array}{ccccccccc}
    \operatorname{NTP}_2 & \rightarrow & \operatorname{NBTP} & \rightarrow& \text{New Kim's Lemma} & \rightarrow& \operatorname{NCTP}&\rightarrow& \operatorname{NATP}\\
     & & \uparrow \\
     & & \operatorname{NSOP}_1,
\end{array}
$$
see \cite{KR24b} and \cite{Han25}.
To our knowledge, there is currently no example of a theory with $\operatorname{BTP}$ that is \NATP{}.

We now discuss the results from two of our previous papers, namely \cite{Chi25} and \cite{Chi25b}. Let $L$ be a language and $T$ a model-complete $L$-theory with an infinite $\varnothing$-definable $K$-vector space $\VV$ in every model.
Given a consistent set $C$ of constraints on an endomorphism, which encodes conditions of the form  
$$
    \sum\nolimits_{k}\bigcap\nolimits_{l}\Ker(\rho_{k, l}[\theta]) = \sum\nolimits_{k}\bigcap\nolimits_{l} \Ker(\eta_{k, l}[\theta])
$$  
(where all sums and intersections are finite,  
and all the $\rho[\theta]$'s and $\eta[\theta]$'s are polynomials over $K$ with the endomorphism $\theta$ plugged in), we define the following theory in the language $L_\theta := L \cup \set{\theta}$:  
$$
T^C_\theta := T \cup \set{\text{``$\theta$ is an endomorphism of $\VV$''}} \cup \set{\text{``$\theta$ satisfies the constraints in $C$''}}.
$$
Note that these sets of constraints are a simplification for the sake of this introduction and that we will actually use so-called kernel configurations instead (see Definition \ref{def_kernel_conf}).
In \cite{Chi25}, we showed that $T^C_\theta$ has a model companion $T\theta^C$ if $T$ satisfies a certain condition \Hfour{}, which corresponds to Definition 1.11 in \cite{dEl21b} and basically states that ``$\psi(\ux; \uy)$ implies no finite disjunction of non-trivial linear dependencies in $\ux$ over $\VV$'' can be expressed as an $L$-formula $\sigma_\psi(\uy)$ for every $L$-formula $\psi(\ux; \uy)$. These results, as well as all others relevant to this paper, will be recalled in Section \ref{sec_prelim_res}. In \cite{Chi25b}, we studied definable sets in $T\theta^C$, the completions of this theory, as well as the algebraic closure in its models.

The goal of the entire paper is to prove the following theorem:
\begin{theoremi}[Theorem \ref{theorem_preserve_natp}] \label{theorem_AAA}
    Suppose that $T$ satisfies \Hfour{}. Then $T\theta^C$, the model companion of $T^C_\theta$, is \NATP{} if and only if $T$ is \NATP{}.
\end{theoremi}
\noindent The idea behind the proof of Theorem \ref{theorem_AAA} consists of essentially two steps:
\begin{enumerate}[(I)]
    \item Show that if $T\theta^{C}$ has the antichain tree property and $T$ satisfies \Hfour{}, then there is an $L$-formula $\psi(\ux_\sh\ux_\va; \uz; \uw)$ and a binary tree $(\ud_\nu : \nu \in 2^{<\omega})$ of parameters in some $\mm \models T$ such that:
    \begin{enumerate}[(i)]
        \item The type $\set{\psi(\ux_\sh\ux_\nu; \uz; \ud_\nu) : \nu \in X}$ implies no finite disjunction of non-trivial linear dependencies in $\ux_\sh(\ux_\nu : \nu \in X)$ over $\VV$ whenever $X \subseteq 2^{<\omega}$ is an antichain.
        \item The type $\set{\psi(\ux_\sh\ux_\nu; \uz; \ud_\nu), \psi(\ux_\sh\ux_{\mu}; \uz; \ud_{\mu})}$ implies some finite disjunction of non-trivial linear dependencies in $\ux_\sh\ux_\nu\ux_\mu$ over $\VV$ whenever $\nu$ and $\mu$ form a path.
    \end{enumerate}
    This can essentially be seen as $T$ having some kind of ``linearly dependent antichain tree property'', especially if $\ux_\sh$ is empty.
    \item Show that this ``linearly dependent antichain tree property'' implies the actual antichain tree property.
\end{enumerate}
In Section \ref{sec_step_red_ez}, we show Step (I) for the ``easiest-to-work-with'' set of constraints $C_0$, which requires that every $\rho[\theta]$ is injective unless $\rho = 0$.
In Section \ref{sec_step_ii_always_ez}, we prove Step (II), which is completely independent of the constraints $C$ considered. Finally, we prove Step (I) for every (reasonable) set of constraints $C$ in Section \ref{sec_proof_general}. Note that the proof idea in Section \ref{sec_proof_general} is the same as in Section \ref{sec_step_red_ez}; however, working with general sets of constraints makes the proof significantly more technical.

Mutchnik showed in \cite{Mut26} that a theory is \NSOPone{} (not the \textit{$1$-strong order property}) if and only if it is \NSOPtwo{} (not the \textit{$2$-strong order property}). We say that a formula $\varphi(\ux; \uy)$ has the $2$-strong order property if there is a binary tree of parameters $(\ub_\nu : \nu \in 2^{<\omega})$ such that $\set{\varphi(\ux; \ub_\nu) : \nu \in X}$ is consistent if and only if $X \subseteq 2^{<\omega}$ is a path. Note that this is essentially the definition of the antichain tree property with paths and antichains swapped. One could now try to prove that our construction preserves \NSOPone{} by simply replacing antichains with paths and vice versa in the proof presented. Although some non-trivial problems arise when doing so, it seems extremely likely to the author that these can be dealt with and that our construction therefore preserves \NSOPone{}. Note that \NSOPone{} theories have been studied more extensively than \NATP{} theories.
For example, there are Kim-Pillay-style characterizations of \NSOPone{} theories, which state that a theory is \NSOPone{} if and only if there is a ternary relation on small subsets of a monster model satisfying certain axioms, and that Kim-independence over models is this relation in that case.
Several variants have been proved, beginning with the pioneering work of Chernikov and Ramsey (Proposition 5.3 in \cite{CR16}). A major subsequent development is due to Kaplan and Ramsey (Theorem 9.1 in \cite{KR20}), and an analogue in the setting of positive logic was later obtained by Dobrowolski and Kamsma (Theorem 9.1 in \cite{DK22}).
Such a Kim-Pillay-style characterization has been used by d'Elbée in \cite{dEl25} to show that the theory $\operatorname{ACFH}$ is \NSOPone{} and that Kim-independence over models is essentially forking-independence in $\ACF{}$ of the algebraic closures in $\operatorname{ACFH}$ of all sets involved. Hence, another conceivable approach to proving that our construction preserves \NSOPone{} is to show that a similarly defined ternary relation satisfies all axioms described by such a Kim-Pillay-style characterization. This would come with the advantage that we would also obtain a description of Kim-forking over models, but in full generality it might require additional assumptions besides \Hfour{}, such as some variant of condition (A) given in Theorem 4.1 in \cite{dEl21b}.

One might also try to generalize the proof strategy of Theorem \ref{theorem_AAA} to the neostability properties $\operatorname{NBTP}$ or $\operatorname{NCTP}$, but even if our construction actually preserves these properties, it seems likely to the author that other strategies are needed to prove this. \\

\noindent \textbf{Acknowledgment.} The author would like to thank Christian d'Elbée for reading an earlier draft of this paper and for providing valuable feedback and suggestions.
The author would also like to thank Philipp Hieronymi for general support.

\section{Preliminary Results} \label{sec_prelim_res}

In this section, we provide a recap of all relevant results from \cite{Chi25} and \cite{Chi25b}.

\subsection{The setting}  \label{sec_setting}

Let $L$ be a first-order language, let $T$ be a model-complete $L$-theory, and let $K$ be a field.
Furthermore, assume that the theory of $K$-vector spaces is definable in $T$.
By this we mean that there are $L$-formulas $\Omega_{\VV}(\ux)$, $\Omega_{0}(\ux)$, $\Omega_{+}(\ux_1, \ux_2, \uy)$, and $\Omega_{\lambda\cdot}(\ux, \uy)$ for each $\lambda \in K$ such that, in every model $\mm \models T$, they define an infinite $K$-vector space $(\VV, 0, +, (\lambda \cdot)_{\lambda \in K})^{\mm}$.
In the case $K = \QQ$, one may also view $(\VV, 0, +)^{\mm}$ as a torsion-free divisible abelian group, since these are precisely the $\QQ$-vector spaces.
If a model $\mm \models T$ is given, then $\VV$ denotes $\VV^\mm$.
If the model is denoted with $\mm'$ instead of $\mm$, then we will write $\VV'$ instead of $\VV$.
If no model of $T$ is clear from the context, then we will still use the letter $\VV$ to denote a $K$-vector space.
We may say ``vector space'' instead of ``$K$-vector space'', ``polynomial'' instead of ``$K$-polynomial'', ``linearly independent'' instead of ``$K$-linearly independent'', and so on. Definable will always mean $\varnothing$-definable.

\begin{example}
    \label{example_main}
    Two of our main examples are as follows:
    \begin{enumerate}[(i)]
        \item Let $\LK = \set{0, +, (\lambda \cdot)_{\lambda\in K}}$ and let $\TKvs$ be the theory of $K$-vector spaces, with the obvious formulas.
        Then, for any $\mm \models \TKvs$, we choose $(\VV, 0, +, (\lambda \cdot)_{\lambda\in K}) = \mm$.
        Similarly, we can also work with ordered divisible abelian groups, since these are precisely ordered $\QQ$-vector spaces.
        \item Let $K = \QQ$, $\Lr = \set{0, 1, +, \cdot, <}$, and $T = \RCF$, the theory of real closed fields.
        Then, for any $\rr \models \RCF{}$, we choose $(\VV, 0, +, (q \cdot)_{q\in \QQ})^{\rr}$ to be $(\rr_{>0}, 1, \cdot, (x \mapsto x^q)_{q\in \QQ})$.
    \end{enumerate}
\end{example}

\noindent Notationally, we will treat $\VV$ as a unary set, or even as a separate sort.
Any $\LK$-term or formula can then be viewed as an $L$-definable function or an $L$-formula, respectively.
Note that $0$, $+$, and $(\lambda\cdot)_{\lambda\in K}$ need not belong to $L$, so they are not necessarily $L$-terms.
For a given theory, there can be multiple definable vector spaces, as in the case of $\RCF{}$.
Thus, whenever a theory $T$ is given, we actually mean the tuple $(T, \Omega_{\VV}, \Omega_{0}, \Omega_{+}, (\Omega_{\lambda\cdot})_{\lambda \in K})$.

We now define $L_\theta := L(\theta) := L \cup \set{\theta}$, where $\theta$ is a function symbol not contained in $L$.
This has the drawback that, for example, if $\rr \models \RCF$ and the positive elements are viewed as a $\QQ$-vector space, then $\theta$ must also be defined outside $\VV = \rr_{>0}$.
In this case, one might try to set $\theta(0) = 0$ and $\theta(-1) \in \set{1, -1}$ so as to extend $\theta$ to an endomorphism of $(\rr, 1, \cdot)$, but then the ambient structure is no longer a vector space.
For the sake of uniformity, we instead define $\theta(x)$ to be the neutral element of $\VV$ for all $x \not\in \VV$ and set
$$
T_\theta := T \cup \set{\text{``$\theta_{\restriction \VV}$ is a $(\VV, 0, +, (\lambda \cdot)_{\lambda\in K})$-endomorphism''}} \cup \set{\forall x \not\in \VV: \theta(x) = 0}.
$$
In particular, $\theta^n(x) = 0$ for all $x \not\in \VV$ and all $n > 0$.
For consistency, we also define $\theta^0(x) = 0$ whenever $x \not\in \VV$, and $x + y = 0$ whenever either $x \not\in \VV$ or $y \not\in \VV$.
Practically, we will ignore the behavior of $\theta$ outside of $\VV$ and simply treat $\theta$ as a function defined only on $\VV$.
For example, we set $\Ker(\theta) := \set{v \in \VV : \theta(v) = 0}$, and similarly define the kernel for any function that is an endomorphism of $\VV$.

Note that if $\VV$ is an $n$-ary set, then we actually need $n$ different $n$-ary function symbols $\theta_1, \dots, \theta_n$ rather than a single unary function symbol $\theta$.
However, as mentioned above, we will treat elements of $\VV$ as singletons in order to simplify notation and will therefore pretend that $\theta$ is unary.

\begin{definition}
    Given a polynomial $\rho \in K[X]$ and any $d \geq \deg(\rho)$, we let $\rho[\theta]$ denote the endomorphism of $\VV$ defined by
    $
    \rho[\theta](v) := \sum\nolimits_{i=0}^{d} (\rho)_i \cdot \theta^i(v)
    $
    for all $v \in \VV$. Here each $(\rho)_i$ is the respective coefficient of $X^i$ in $\rho$.
\end{definition}

\begin{notation}
    If $\theta$ is clear from the context we may write $\Ker(\rho)$ and $\Image(\rho)$ instead of $\Ker(\rho[\theta])$ and $\Image(\rho[\theta])$.
\end{notation}

\noindent It is easy to check that $\rho[\theta] + \eta[\theta] = (\rho + \eta)[\theta]$ and $\rho[\theta] \circ \eta[\theta] = (\rho \cdot \eta)[\theta]$. Many classical facts for polynomials, such as Bézout's identity or Euclidean division, can be translated to this setting. The following can be shown with some Euclidean divisions:

\begin{fact}[Lemma 2.7 in \cite{Chi25}]\label{lemma_bound_term_light}
    Let $\ux = (x_1, \dots, x_n)$ be a tuple of variables, and assume that a formula
    $$
        E(\ux; \uy) := \bigwedge\nolimits_{k=1}^n \xi_k[\theta](x_k) = \sum\nolimits_{l=1}^{k-1} Q_{k, l}[\theta](x_l) + \mu_k(\uy)
    $$
    is given, where the $\xi_k$ are non-zero polynomials, the $Q_{k, l}$ are polynomials, and the $\mu_k(\uy)$ are $\LKThe$-terms.
    Then any $\LKThe$-term $\tau(\ux; \uy)$ is equivalent modulo $\TKvsThe \cup \set{E(\ux; \uy)}$ to a term of the form
    $$
        \sum\nolimits_{k=1}^n \rho_k[\theta](x_k) + \mu'(\uy)
    $$
    with $\deg(\rho_k) < \deg(\xi_k)$ for all $k$, where $\mu'(\uy)$ is an $\LKThe$-term.
\end{fact}

\subsection{Kernel configurations and extensions of $T_\theta$}

\noindent As stated in the introduction, we consider a family $\set{T^C_\theta : C \in \Cc}$ of extensions of $T_\theta$.

\begin{notation}
    We let $\Kp{}$ denote the set of all monic irreducible polynomials over $K$.
\end{notation}

\noindent We start by defining our index set $\Cc$:

\begin{definition} \label{def_kernel_conf}
    We call a pair $(c, d)$ a \textbf{kernel configuration} if
    \begin{enumerate}[(i)]
        \item $c \colon \Kp{} \to \NN \cup \set{\infty}$ is a function; and
        \item $d \in \NN_{> 0} \cup \set{\infty}$ is either $\infty$ or satisfies $d = \sum_{f \in \Kp{}} \deg(f) \cdot c(f)$.
    \end{enumerate}
    We let $\Cc$ denote the set of all kernel configurations.
    Given such a kernel configuration $C = (c, d) \in \Cc$, we set $C(f) := c(f)$ for all $f \in \Kp{}$.
    We define the \textbf{degree} of $C$ by $\deg(C) := d$.
    We say that $C$ is \textbf{algebraic} if $\deg(C) < \infty$.
    In this case, we define the \textbf{minimal polynomial} of $C$ by $\mipo(C) := \prod_{f \in \Kp{}} f^{C(f)}$
    (since $\deg(C) < \infty$, only finitely many factors are different from $1$, so this product is well defined).
    We say that $C$ is \textbf{transcendental} if $\deg(C) = \infty$.
    We let $\Ccalg$ and $\Cctrans$ denote the sets of all algebraic and all transcendental kernel configurations, respectively.
\end{definition}

\noindent Note that the set of kernel configurations depends on the field $K$.
We are now ready to define the family $\set{T^C_\theta : C \in \Cc}$:

\begin{definition} \label{def_T_C_theta}
    Given $C \in \Cc$ and an endomorphism $\theta \colon \VV \to \VV$, we say that $\theta$ is a \textbf{$\mathbf{C}$-endomorphism} if one of the following holds:
    \begin{enumerate}[(i)]
        \item $C$ is algebraic and $\Ker(\mipo(C)) = \VV$, that is, $\mipo(C)[\theta] = 0$.
        \item $C$ is transcendental and $\Ker(f^{C(f)}) = \Ker(f^{C(f)+1})$ for all $f \in \Kp{}$ with $C(f) < \infty$.
    \end{enumerate}
    We define $T^C_\theta := T_\theta \cup \set{\text{``$\theta$ is a $C$-endomorphism''}}$.
\end{definition}

\noindent The family $\set{T^C_\theta : C \in \Cc}$ might seem a bit arbitrary at first; however, notice that any consistent extension of the form
$$
    T_\theta \cup \Set{\sum\nolimits_{k}\bigcap\nolimits_{l}\Ker(\rho_{j, k, l}[\theta]) = \sum\nolimits_{k}\bigcap\nolimits_{l} \Ker(\eta_{j, k, l}[\theta]) : j \in \jj},
$$
where all sums and intersections are finite,  all the $\rho_{j, k, l}$'s and $\eta_{j, k, l}$'s are polynomials over $K$, and $\jj$ is a potentially infinite index set, is equivalent to some $T^C_\theta$ (see Corollary 2.13 in \cite{Chi25} - the proof heavily uses consequences of Bézout's identity). Also note that every $T^C_\theta$ is consistent, and that $T^C_\theta \not\equiv T^{C'}_\theta$ whenever $C \neq C'$ (see Lemma 2.19 in \cite{Chi25}). So the set $\Cc$ parametrizes all consistent extensions as described above, in some sense. Some concrete examples:
\begin{enumerate}[(i)]
    \item Let $C_\infty$ be transcendental with $C_\infty(f) = \infty$ for all $f \in \Kp{}$. One can check that $T_\theta^{C_\infty}$ is $T_\theta$. 
    \item Let $C_0$ be transcendental with $C_0(f) = 0$ for all $f \in \Kp{}$. One can check that $T_\theta^{C_0}$ is $T_\theta \cup \set{\text{``$\rho[\theta]$ is injective''} : \rho \in K[X] \setminus \set{0}}$. In an existentially closed model of $T_\theta^{C_0}$, the maps $\rho[\theta]$ are isomorphisms, so one can solve systems of equations of the form $\bigwedge_{k=1}^m \sum_{l=1}^n \rho_{k, l}[\theta](x_l) = y_k$ just like in a $K(X)$-vector space.
    \item Let $C_f$ be algebraic with $\mipo(C_f) = f \in \Kp{}$. One can, similarly to (ii), solve such systems of equations as in a $K[X]/(f)$-vector space. Here, however, one has the potential advantage that the sequence $\set{\theta^i(v) : i \in \omega}$ is already determined by $\set{\theta^i(v) : 0 \leq i < \deg(f)}$.
\end{enumerate}
These are, in some cases, the ``easiest'' examples to work with, and it can often be helpful to first work with one of these kernel configurations and then turn to the general case. By contrast, the ``hardest'' kernel configurations to work with are those for which $\set{f \in \Kp{} : 0 < C(f) < \infty}$ is infinite.

In \cite{Chi25}, we also showed that every $T^C_\theta$ is inductive (see Lemma 3.2 there). Hence, the model companion of each $T^C_\theta$ exists if and only if the class of existentially closed models of $T^C_\theta$ is elementary. In this case, the model companion is exactly the axiomatization.

One can easily check that if $C$ and $C'$ are algebraic kernel configurations, then $C = C'$ if and only if $\mipo(C) = \mipo(C')$. Furthermore, we have $\deg(C) = \deg(\mipo(C))$.

\begin{fact} \label{fact_trivvivi}
    If $C$ is \textbf{trivial}, that is, if $\deg(C) = 1$, then the models of $T_\theta^C$ and $T$ are interdefinable.
    Hence, the model companion of $T_\theta^C$ is $T_\theta^C$ itself.
\end{fact}

\noindent We will often implicitly assume that $C$ is non-trivial.

\begin{notation}
    We introduce a few more notations for working with a kernel configuration $C \in \Cc$:
    \begin{enumerate}[(i)]
        \item Given $f \in \Kp{}$ with $C(f) < \infty$, we write $f^C$ instead of $f^{C(f)}$, $f^{C+k}$ instead of $f^{C(f) + k}$, and so on.
        \item Given a finite set $F \subseteq \Kp{}$ with $C(f) < \infty$ for all $f \in F$, we set \hbox{$F^C := \prod_{f \in F} f^C$}.
        \item We define the following subsets of $\Kp{}$:
        $$
        \Kp{C<\infty} := \set{f \in \Kp{} : C(f) < \infty}, \quad \Kp{0<C<\infty} := \set{f \in \Kp{} : 0 < C(f) < \infty},
        $$
        $$
        \Kp{C=0} := \set{f \in \Kp{} : C(f) = 0}, \quad \text{and} \quad \Kp{C=\infty} := \set{f \in \Kp{} : C(f) = \infty}.
        $$
    \end{enumerate}
\end{notation}

\noindent With the above, for any algebraic kernel configuration $C \in \Ccalg$, we have $\deg(C) = \deg(\mipo(C))$ and
    $$
    \mipo(C) = \prod\nolimits_{f \in \Kp{0<C<\infty}} f^C = (\Kp{0<C<\infty})^C.
    $$
\noindent One should note that any $\LKThe$-sentence that holds in $\TKvsTheC$ also holds in $T^C_\theta$ (here $\LKThe$ is the language of $K$-vector spaces with an endomorphism).
To be more precise, one has to modify the sentence accordingly, e.g., quantifiers of the form $\exists x$ must be replaced with $\exists x \in \VV$ and the formulas that define addition/scalar multiplication must be used instead of the function symbols in $\LKThe$.
In general, it also turns out that whenever a model $(\mm, \theta)$ of $T^C_\theta$ is existentially closed, $(\VV, \theta)$ is an existentially closed model of $\TKvsTheC$ (see Remark 3.4 in \cite{Chi25}).

\subsection{$C$-image-completeness}

\noindent Note that the condition of $\theta$ being a $C$-endomorphism does not (at least in the transcendental case) imply any kind of equations that involve the image of $\rho[\theta]$ for some $\rho \in K[X]$.
In Remark 2.21 in \cite{Chi25}, we discussed that it is very unlikely that considering expansions of $T_\theta$ that also impose equations on the images (or mixed equations with sums and intersections of both kernels and images) will lead to new model companions.
However, the following condition holds in any existentially closed model of $T^C_\theta$ and is fundamental to understanding the structure of these models:

\begin{definition} \label{def_C_image_comple}
    We say that an endomorphism $\theta \colon \VV \to \VV$ is \textbf{$\mathbf{C}$-image-complete} if it is a $C$-endomorphism and $\Image(f^{C+1}) = \Image(f^C)$  holds for all $f \in \Kp{C<\infty}$ (recall $\Image(f^C) := \Image(f^{C(f)}[\theta])$).
    We may call a model $(\mm, \theta) \models T_\theta$ \textbf{$\mathbf{C}$-image-complete} if the endomorphism $\theta$ is $C$-image-complete.
\end{definition}

\begin{fact} \label{fact_c_image_comple}
    The following statements hold:
    \begin{enumerate}[(i)]
        \item If $C$ is algebraic, then every $C$-endomorphism is $C$-image-complete (see Lemma 3.11 in \cite{Chi25}).
        \item If $C$ is any kernel configuration and $(\mm, \theta) \models T_\theta^C$ is existentially closed, then $(\mm, \theta)$, or equivalently $\theta$, is also $C$-image-complete (see Corollary 3.13 in \cite{Chi25}).
    \end{enumerate}
\end{fact}

\noindent Note that the converse of (ii) in Fact \ref{fact_c_image_comple} above does not hold.
The most important consequence of $C$-image-completeness is that we can decompose $\VV$ into definable direct summands as follows.

\begin{fact}[Lemma 3.14 in \cite{Chi25}] \label{lemma_decomposition}
    If $\theta \colon \VV \to \VV$ is $C$-image-complete and $F \subseteq \Kp{0<C<\infty}$ is a finite set, then we have
    $$
    \VV = \Image(F^C) \oplus \Ker(F^C) = \Image(F^C) \oplus \bigoplus\nolimits_{f\in F} \Ker(f^C).
    $$
    If $C$ is algebraic and $F = \Kp{0<C<\infty}$, then the summand $\Image(F^C)$ is $\set{0}$ and can therefore be omitted.
\end{fact}

\noindent In \cite{Chi25}, we defined additional endomorphisms of $\VV$ using these decompositions.

\begin{notation}
    For any $\rho \in K[X] \setminus \set{0}$, we let $\Fac(\rho) := \set{f \in \Kp{} : f \mid \rho}$ denote the set of all irreducible factors of $\rho$.
    As a convention, we set $\Fac(0) = \varnothing$.
\end{notation}

\begin{fact}[Lemma 3.15 in \cite{Chi25}] \label{fact_endo_gen}
    In the theory \hbox{$\TKvsThe \cup \set{\text{``$\theta$ is $C$-image-complete''}}$}, the following endomorphisms are definable:
    \begin{enumerate}[(i)]
        \item For $F \subseteq \Kp{0<C<\infty}$ finite, we define the \textbf{projection to the image of $\bm{F^C[\theta]}$} by
        $$
        \pi_{\Image(F^C)}(x) := \text{``the unique $u \!\in\! \Image(F^C)$ for which there is $v \in \Ker(F^C)$ with $x \!=\! u \!+\! v$''.}
        $$
        \item For $F \subseteq \Kp{0<C<\infty}$ finite, we define the \textbf{projection to the kernel of $\bm{F^C[\theta]}$} by
        $$
        \pi_{\Ker(F^C)}(x) := \text{``the unique $v \!\in\! \Ker(F^C)$ for which there is $u \in \Image(F^C)$ with $x \!=\! u \!+\! v$''.}
        $$
        We clearly have $\pi_{\Ker(F^C)} = 1[\theta] - \pi_{\Image(F^C)}$.
        \item For every monic polynomial $\eta \in K[X]$ with $\Fac(\eta) \subseteq \Kp{C<\infty}$, we define the \textbf{pseudo-inverse of $\bm{\eta[\theta]}$} by
        $$
        \eta[\theta]^{-1}(x) := \text{``the unique $u \in \Image(\Fac(\eta)^C)$ with $\eta[\theta](u) = \pi_{\Image(\Fac(\eta)^C)}(x)$''.}
        $$
        Notice that $\Fac(\eta)^C = (\Fac(\eta) \cap \Kp{0<C<\infty})^C$.
        In practice, we will also use $\eta[\theta]^{-1}$ for (non-zero) non-monic polynomials by setting $\eta[\theta]^{-1} := \lambda^{-1} \cdot (\eta/\lambda)[\theta]^{-1}$ for the leading coefficient $\lambda$ of $\eta$.
    \end{enumerate}
\end{fact}

\noindent Notice that whenever $\Fac(\eta) \cap \Kp{0<C<\infty} \neq \varnothing$, we obtain $\eta[\theta] \circ \eta[\theta]^{-1} = \pi_{\Image(\Fac(\eta)^C)} \neq \Id = 1[\theta]$, so the notation might be a bit misleading. Here $\Id$ is the identity on $\VV$.
It is also easy to see that $\Id = 1[\theta] = \pi_{\Image(f^C)} + \pi_{\Ker(f^C)}$. We consider the ring generated by all $\LKThe$-definable endomorphisms we have collected so far:

\begin{fact}[Theorem 3.18 in \cite{Chi25}] \label{theorem_r_c_def}
    We let $R_C$ be the set of all endomorphisms definable in the theory $\TKvsThe \cup \set{\text{``$\theta$ is $C$-image-complete''}}$ that are $\set{+, \circ}$-generated by
    $$
    \set{\rho[\theta] : \rho \in K[X]} \cup \set{\pi_{\Image(F^C)} : F \subseteq \Kp{0<C<\infty} \text{ finite}} \cup \set{\eta[\theta]^{-1} : \eta \text{ monic with }\Fac(\eta) \subseteq \Kp{C<\infty}}.
    $$
    The structure $(R_C, 0[\theta], 1[\theta], +, \circ)$ is a unitary commutative ring with $\Char(R_C) = \Char(K)$. In fact, it can also be seen as a $K$-algebra, as $K \subseteq R_C$ (identifying $\lambda \in K$ with the endomorphism $x \mapsto \lambda \cdot x$ which is $\lambda[\theta]$).
\end{fact}

\noindent We may sometimes write $0$ instead of $0[\theta]$ and $\Id$ or $1$ instead of $1[\theta]$. In particular, if we regard $R_C$ purely as a ring, we may write $(R_C, 0, 1, +, \cdot)$. The ring $R_C$ may again look complicated at first, but for the ``easiest to work with'' kernel configurations, we obtain the following:
\begin{enumerate}[(i)]
    \item $(R_{C_\infty}, 0, 1, +, \cdot) \simeq (K[X], 0, 1, +, \cdot)$ for the unique kernel configuration $C_\infty \in \Cctrans$, which satisfies \hbox{$C_\infty(f) = \infty$} for all $f \in \Kp{}$.
    \item $(R_{C_0}, 0, 1, +, \cdot) \simeq (K(X), 0, 1, + , \cdot)$ for the unique kernel configuration $C_0 \in \Cctrans$ with $C_0(f) = 0$ for all $f \in \Kp{}$. 
    Notice that this is a field.
    \item $(R_{C}, 0, 1, +, \cdot) \simeq (K[X]/(\mipo(C)), 0, 1, +, \cdot)$ for all $C \in \Ccalg$.
    This also implies that our ring $(R_{C}, 0, 1, +, \cdot)$ is a field for all algebraic kernel configurations $C$ with $\mipo(C)$ being irreducible.
\end{enumerate}
Other examples of $R_C$ can be found in Corollary 3.25 in \cite{Chi25}. 
There, the case where $C$ is transcendental and $\Kp{0<C<\infty}$ is infinite again turns out to be the most complicated. Multiplication rules, such as $\rho[\theta] \circ \pi_{\Image(F^C)} = \rho[\theta]$ if $F^C \mid \rho$, can be found in Lemma 3.21 in \cite{Chi25}.  

\begin{fact}[see Remark 3.26 in \cite{Chi25}] \label{fact_when_field}
    $R_C$ is a field if and only if $C = C_0$, as in (ii) above, or if $C$ is algebraic with $\mipo(C)$ irreducible. 
\end{fact}

\noindent Also note that the elements of $R_C$ are, as defined in Fact \ref{theorem_r_c_def}, definable functions in the theory $\TKvsThe \cup \set{\text{``$\theta$ is $C$-image-complete''}}$.
By this, we mean that the elements of $R_C$ are equivalence classes of $\LKThe$-formulas modulo the theory \hbox{$\TKvsThe \cup \set{\text{``$\theta$ is $C$-image-complete''}}$} that define an endomorphism in every model of $\TKvsThe \cup \set{\text{``$\theta$ is $C$-image-complete''}}$.
So, in order to prove $r = r'$, we need to show
$$
    r^{(\VV, \theta)} = r'^{(\VV, \theta)} \quad \text{for all $(\VV, \theta) \models \TKvsThe \cup \set{\text{``$\theta$ is $C$-image-complete''}}$},
$$
and, in order to prove $r \neq r'$, we need to find $(\VV, \theta) \models \TKvsThe \cup \set{\text{``$\theta$ is $C$-image-complete''}}$ with
$
    r^{(\VV, \theta)} \neq r'^{(\VV, \theta)}.
$
In Remark 3.20 in \cite{Chi25}, we showed that $r \neq r'$ implies $r^{(\VV, \theta)} \neq r'^{(\VV, \theta)}$ if $(\VV, \theta)$ is an existentially closed model of $\TKvsTheC$.
In Theorem 4.10 in \cite{Chi26} we proved that in an existentially closed model of $T_\theta^C$, any $\LKThe$-definable endomorphism of $\VV$ is, in fact, in $R_C$.

\begin{definition}
    We define $\LRC$ as the language of (left) $R_C$-modules (with $R_C$ as in Fact \ref{theorem_r_c_def}), i.e., $\LRC = (0, +, (r)_{r \in R_C})$, where each $r \in R_C$ is treated as a unary function symbol.
\end{definition}

\noindent We will write $r(x)$ instead of $r \cdot x$, since $r$ will usually be a function such as $\pi_{\Image(F^C)}$.
Given a model $(\VV, \theta) \models \TKvsThe \cup \set{\text{``$\theta$ is $C$-image-complete''}}$, we define an $\LRC$-structure on $\VV$ using the definable functions of which $R_C$ consists. This structure is obviously interdefinable with $(\VV, \theta)$.

\begin{fact}[Common Base Theorem - Theorem 3.24 in \cite{Chi25}]\label{theorem_r_c_element}
    Given $r_1, \dots, r_q \in R_C$ and any finite subset $F_0 \subseteq \Kp{0<C<\infty}$, we can write
    $$
    r_i = \rho_{F,i}[\theta] \circ \eta[\theta]^{-1} \circ \pi_{\Image(F^C)} + \sum\nolimits_{f\in F} \rho_{f,i}[\theta] \circ \pi_{\Ker(f^C)}
    $$
    for all $i \in \set{1, \dots, q}$, where
    \begin{enumerate}[(i)]
        \item $F \subseteq \Kp{0<C<\infty}$ is finite with $F_0 \subseteq F$,
        \item $\eta$ is a monic polynomial with $\Fac(\eta) \subseteq F \cup \Kp{C=0}$,
        \item the $\rho_{F,i}$ satisfy both $\Fac(\rho_{F,i}) \subseteq F \cup \Kp{C=0}\cup\Kp{C=\infty}$ and $\gcd(\rho_{F,1}, \dots, \rho_{F,q}, \eta) = 1$,
        \item the $\rho_{f,i}$ are polynomials with $\deg(\rho_{f,i}) < \deg(f^C)$.
    \end{enumerate}
    If $C$ is algebraic, we can furthermore choose $F = \Kp{0<C<\infty}$, resulting in
    $
    r_i = \sum\nolimits_{f\in F} \rho_{f,i}[\theta] \circ \pi_{\Ker(f^C)}
    $.
\end{fact}

\subsection{Existentially closed models and first-order axiomatization}

We now state the characterization of the existentially closed models of $T^C_\theta$ from \cite{Chi25}. We start with the remaining ingredients:

\begin{definition}
    \label{def_param_c_sequence_system} A \textbf{parametrized $\mathbf{C}$-sequence-system} is an $\LKThe$-formula of the form
    $$
    S(\ux; \uy) = \bigwedge\nolimits_{k=1}^n f_k^{q_k}[\theta](x_{\ldd, k}) = y_k
    $$
    with $\ux := \ux_\li\ux_\ld := (x_{\lii, k} : 1 \leq k \leq m)(x_{\ldd, k} : 1 \leq k \leq n)$ and $\uy = (y_1, \dots, y_n)$ that satisfies the following conditions:
    \begin{enumerate}[(i)]
        \item If $C$ is algebraic, then $m = 0$.
        \item $f_k \in \Kp{0<C}$ and $q_k \in \set{q \in \NN : 0 < q \leq C(f_k)}$ hold for all $k \in \set{1, \dots, n}$.
    \end{enumerate}
    
\end{definition}
\noindent We will always denote parametrized $C$-sequence-systems by the letter $S$.
Given such a parametrized $C$-sequence-system $S(\ux; \uy)$, we assume that everything is as above, that is, $m$, $n$, and the $f_k$'s and $q_k$'s are defined implicitly, and we set $\ux = \ux_\li\ux_\ld$ as above.
When we partition $\ux = \ux_\li\ux_\ld$ as above, we think of:
    \begin{enumerate}[(i)]
        \item $\ux_\li$ as the linearly independent part of $\ux$, since $S(\ux; \uy)$ does not imply any linear dependencies for the sequence $(\theta^i(x_{\lii, k}) : 1 \leq k \leq m, i \in \omega)$;
        \item $\ux_\ld$ as the linearly dependent part of $\ux$, since the formula $S(\ux; \uy)$ implies that the sequence $(\theta^i(x_{\ldd, k}) : i \in \omega)$ is linearly dependent over $\spanA{y_k}{K}$ for each $k \in \set{1, \dots, n}$.
    \end{enumerate}
The names $\ux_\li$ and $\ux_\ld$ for these tuples are abbreviations for linearly independent and linearly dependent. Notice that in the algebraic case, we require $\ux_\li$ to be empty, which makes sense, as we have $\sum_{i=0}^{\deg(\mipo(C))} (\mipo(C))_i \cdot \theta^i(v) = 0$ for any $v \in \VV$ in that case.

\begin{definition}
    \label{def_c_sequence_system} \label{def_compatible} Let $S(\ux; \uy) = \bigwedge_{k=1}^n f_k^{q_k}[\theta](x_{\ldd, k}) = y_k$ be a parametrized $C$-sequence-system as in Definition \ref{def_param_c_sequence_system}, and let $(\VV, \theta) \models \TKvsThe$ be given.
\begin{enumerate}[(i)]
    \item We say that a tuple $\uu = (u_1, \dots, u_n) \in \VV$ is \textbf{compatible} with $S$ if $u_k \in \Ker(f_k^{C-q_k})$ for every $k \in \set{1, \dots, n}$ with $f_k \in \Kp{0<C<\infty}$.
    \item A \textbf{$\mathbf{C}$-sequence-system} over $(\VV, \theta)$ is an $\LKThe(\VV)$-formula of the form
    \hbox{$
    S(\ux) = S'(\ux; \uu)
    $}
    where $S'$ is a parametrized $C$-sequence-system and $\uu \in \VV$ is compatible with $S'$.
\end{enumerate}
\end{definition}

\noindent We will also denote $C$-sequence-systems over some $(\VV, \theta) \models \TKvsTheC$ by the letter $S$.
Notice that a $C$-sequence-system over $(\VV, \theta)$ is also a $C$-sequence-system over any extension $(\VV', \theta')$ that is also a model of $\TKvsThe$.

\begin{definition}[Placeholder notation] \label{def_placeholder_notation}
    Let $\ux = (x_k : k \in \kk)$ be a tuple of variables.
    We define the \textbf{placeholder sequence} \hbox{$\uxvec := (x^i_k : k \in \kk, i \in \omega)$} to be a new tuple of variables.
    We call each $x^i_k$ a \textbf{placeholder variable} or a \textbf{placeholder} for $\theta^i(x_k)$.
    We furthermore define:
    \begin{enumerate}[(i)]
        \item $\ux^i := (x_k^i : k \in \kk)$ for each $i \in \omega$, and
        \item $\xvec_k := (x^i_k : i \in \omega)$ for each $k \in \kk$.
    \end{enumerate}
    We may sometimes write $(\ux^i : i \in \omega)$ or $(\xvec_k : k \in \kk)$ instead of $\uxvec$.
    For a singleton $x$, we similarly define $\xvec := (x^i : i \in \omega)$.
    If a formula $\psi(\uxvec; \uw)$ or a term $\lambda(\uxvec{}; \uw)$ is given, we define:
    \begin{enumerate}[(i)]
        \setcounter{enumi}{2}
        \item $\psi_\theta(\ux; \uw) := \psi((\theta^i(x_k) : k \in \kk, i \in \omega); \uw)$.
        \item $\lambda_\theta(\ux; \uw) := \lambda((\theta^i(x_k) : k \in \kk, i \in \omega); \uw)$.
    \end{enumerate}
\end{definition}

\begin{definition} \label{def_formual_bounded}
    Let $S(\ux; \uy)= \bigwedge\nolimits_{k=1}^n f_k^{q_k}[\theta](x_{\ldd, k}) = y_k$ be a parametrized $C$-sequence-system as in Definition \ref{def_param_c_sequence_system}.
    If $\psi(\uxvec; \uw)$ is a formula, then we say $\psi(\uxvec; \uw)$ is \textbf{bounded} by $S$ if one of the following equivalent conditions holds:
        \begin{enumerate}[(i)]
            \item For all $k \in \set{1, \dots, n}$, the variable $x^i_{\ldd, k}$ does not appear in $\psi(\uxvec; \uw)$ for \hbox{$i \geq \deg(f_k^{q_k})$}.
            \item For all $k \in \set{1, \dots, n}$, the term $\theta^i(x_{\ldd, k})$ does not appear in $\psi_\theta(\ux; \uw)$ for \hbox{$i \geq \deg(f_k^{q_k})$}.
        \end{enumerate}
    We also say that $\psi_\theta(\ux; \uw)$ is \textbf{bounded} by $S$, if $\psi(\uxvec; \uw)$ is bounded by $S$. Furthermore, we say that a formula is \textbf{bounded} by a $C$-sequence-system $S$ (i.e., a parametrized $C$-sequence-system with some compatible parameters plugged in) if it is bounded by the underlying parametrized $C$-sequence-system.
\end{definition}

\noindent In practice, for a formula $\psi(\uxvec)$ to be bounded by $S$ means that no subterm of the form $\theta^i(x_{\ldd, k})$ appearing in $\psi_\theta(\ux)$ can be replaced by applying a Euclidean division with the equation $f_k^{q_k}[\theta](x_{\ldd, k}) = y_k$ from $S(\ux; \uy)$. Indeed, if $i \geq \deg(f^{q_k}_k)$, then we could replace $\theta^i(x_{\ldd, k})$ with $\chi[\theta](y_k) + r[\theta](x_{\ldd, k})$, where $\chi, r \in K[X]$ are the unique polynomials  satisfying $\chi \cdot f^{q_k}_k + r = X^i$ and $\deg(r) < \deg(f^{q_k}_k) \leq i$.

Now that we have all ingredients, we can state the characterization of existentially closed models of $T^C_\theta$:

\begin{theorem}[Theorem 3.32 in \cite{Chi25}] \label{theorem_big_characterization}
    $(\mm, \theta) \models T^C_\theta$ is existentially closed if and only if it is $C$-image-complete and 
    $$
        (\mm, \theta) \models \exists \ux \in \VV : \psi_\theta(\ux) \wedge S(\ux)
    $$
    holds for any $C$-sequence-system $S(\ux)$ over $(\VV, \theta)$ and $L(M)$-formula $\psi(\uxvec)$ that is bounded by $S$ and does not imply any finite disjunction of non-trivial linear dependencies in $\uxvec$ over $\VV$.
\end{theorem}

\noindent We give two examples where the characterization simplifies quite a lot:
\begin{enumerate}[(i)]
    \item Fix $f \in \Kp{}$ and let $C_f$ be the unique algebraic kernel configuration with $\mipo(C_f) = f$.
    A model $(\mm, \theta) \models T^{C_f}_\theta$ is existentially closed if and only if
    $$
        (\mm, \theta) \models \exists \ux \in \VV : \psi(\theta^0(\ux), \dots, \theta^{\deg(f)-1}(\ux))
    $$
    holds for every $L(M)$-formula $\psi(\ux^0, \dots, \ux^{\deg(f)-1})$ that does not imply any finite disjunction of non-trivial linear dependencies in $\ux^0, \dots, \ux^{\deg(f)-1}$ over $\VV$.
    This is Theorem 3.33 in \cite{Chi25}.
    
    \item Let $C_0$ be the unique transcendental kernel configuration with $C_0(f) = 0$ for all $f \in \Kp{}$.
    A model $(\mm, \theta) \models T^{C_0}_\theta$ is existentially closed if and only if $\rho[\theta]$ is invertible for every $\rho \in K[X] \setminus \set{0}$ and
    $$
        (\mm, \theta) \models \exists \ux \in \VV : \psi_\theta(\ux)
    $$
    holds for every $L(M)$-formula $\psi(\uxvec)$ that does not imply any finite disjunction of non-trivial linear dependencies in $\uxvec$ over $\VV$.
    This is Theorem 3.34 in \cite{Chi25}.
\end{enumerate}

\noindent The next step is to first-order axiomatize this characterization when possible.
For this, we need the following two families of formulas:

\begin{fact}[Lemma 3.36 in \cite{Chi25}]
    Given a parametrized $C$-sequence-system $S(\ux; \uy)$, there is a $\LKThe$-formula $\delta_S(\uy)$ such that $(\mm, \theta) \models \delta_S(\uu)$ holds if and only if $\uu$ is compatible with $S$.
\end{fact}

\begin{definition}[see Definition 1.11 in \cite{dEl21b}] \label{def_hfour}
    We say that $T$ (with the specific choice of $\VV$) satisfies $(\operatorname{H4})$ if, for every $L$-formula $\psi(\ux; \uw)$, there is an $L$-formula $\sigma_\psi(\uw)$ such that, for all $\mm \models T$ and $\ud \in M$, we have $\mm \models \sigma_\psi(\ud)$ if and only if one of the following two equivalent conditions holds:
    \begin{enumerate}[(i)]
        \item The formula $\psi(\ux; \ud)$ implies no finite disjunction of non-trivial linear dependencies in $\ux$ over $\VV$.
        \item There are an elementary extension $\mm' \succ \mm$ and a tuple $\uv' \in \VV'$ linearly independent over $\VV$ such that $\mm' \models \psi(\uv'; \ud)$.
    \end{enumerate}
\end{definition}

\begin{theorem}[Theorem 3.39 in \cite{Chi25}] \label{theorem_first_oder}
If $T$ satisfies \Hfour{}, then $T_\theta^C$ has a model companion $T\theta^C$, i.e., a first-order axiomatization of the class of existentially closed models.
It is axiomatized by the theory \hbox{$T_\theta \cup \set{\text{``$\theta$ is $C$-image-complete''}}$} together with the sentence
$$
\forall \uw:\forall \uy \in \VV : (\sigma_\psi(\uw) \wedge \delta_S(\uy)) \rightarrow \exists \ux \in \VV : \psi_\theta(\ux; \uw) \wedge S(\ux; \uy)
$$
for every parametrized $C$-sequence-system $S(\ux; \uy)$ and every $L$-formula $\psi(\uxvec; \uw)$ that is boun\-ded by $S$.
\end{theorem}

\begin{example} \label{examples_hfour}
    \Hfour{} holds in the following settings:
    \begin{enumerate}[(i)]
        \item The theory $\TKvs$ with the vector space $(\VV, +, 0, (\lambda \cdot)_{\lambda \in K})$ being the entire structure satisfies \Hfour{}.
        This follows easily from quantifier elimination.
        \item Any complete and model-complete o-minimal theory $T$ extending the theory of divisible ordered abelian groups, with $(\VV, +, 0, (q \cdot)_{q \in \QQ})$ being a subinterval of the line (but not necessarily a subgroup) and continuous operations, satisfies \Hfour{} if and only if there is no infinite definable family of germs of $(\VV, +, 0, (q \cdot)_{q \in \QQ})$-endomorphisms at $0_\VV$; combine Theorem 2.4 and Lemma 2.9 in \cite{Blo23}.
        \item Any complete and model-complete o-minimal expansion of $\RCF{}$ with $(\VV, +, 0, (q \cdot)_{q \in \QQ})$ given by $(R_{>0}, \cdot, 1, (x \mapsto x^q)_{q\in \QQ})$ satisfies \Hfour{} if and only if no partial exponential function is definable.
        This is a special case of (ii); see the proof of Theorem A in \cite{Blo23}.
        \item Let $\FF_q$ be a finite field with $q = p^r$.
        If $T$ expands the theory of an $\FF_q$-vector space and $\VV$ is that vector space, then $T$ satisfies \Hfour{} if and only if it eliminates $\exists^\infty$; see the proof of Theorem 5.2 in \cite{dEl21b}.

        For expansions of the theories $\ACF{}_p$, $\operatorname{SCF}_{p, e}$ ($e$ either finite or infinite), $\operatorname{Psf}_p$, $\operatorname{ACFA}_p$, and $\operatorname{DCF}_p$, this implies that whenever $\FF_q$ is contained in every model as constants, \Hfour{} holds with the $\FF_q$-vector space given by addition.
        For more details, see Example 5.10 in \cite{dEl21b}.
        \item If $T$ eliminates $\exists^\infty x \in \VV$ (see Definition 3.3 in \cite{Chi26}; recall that $\VV$ might not be the entire domain, or even an $n$-ary set), and the following condition holds:
        \begin{itemize}
            \item[(V)] For any $\mm' \succ \mm \models T$ and any tuple $\uv'$ in $\VV'$, $\uv'$ is $\acl_L$-independent over $M$ if and only if it is linearly independent over $\VV$.
        \end{itemize}
        then $T$ satisfies \Hfour{} (see Lemma 3.4 in \cite{Chi26}).
    \end{enumerate}
    In Theorem 3.6 in \cite{Chi26}, we proved that the elimination of $\exists^\infty x \in \VV$ is a necessary condition for the existence of the model companion of $T^C_\theta$ (for non-trivial $C$), and consequently a necessary condition for \Hfour{}.
\end{example}

\noindent One important observation is that the formulas $\sigma_\psi(\uw)$ obtained from \Hfour{} are always equivalent to formulas stating that $\psi(\ux; \uw)$ has a realization in $\VV$ that satisfies a fixed finite set of non-trivial linear inequations, which are parametrized by $\uw$:

\begin{fact}[Lemma 3.7 in \cite{Chi25b}]
    \label{lemma_hfour_fml}
    If $T$ satisfies \Hfour{}, then one can choose the formula $\sigma_\psi(\uw)$ (see Definition \ref{def_hfour}) for each $L$-formula $\psi(\ux; \uw)$ with $\ux = (x_1, \dots, x_n)$ to be of the form
    $$
    \exists \ux \in \VV : \psi(\ux; \uw) \wedge \bigwedge\nolimits_{k=1}^m \neg\varphi_k\Big(\sum\nolimits_{l=1}^n \lambda_{k, l} \cdot x_l; \uw\Big),
    $$
    with $(\lambda_{k, 1}, \dots, \lambda_{k, n}) \in K^n \setminus \set{\uzero}$ and $\varphi_k(y; \uw)$ algebraic in $y$ for each $k \in \set{1, \dots, m}$.
\end{fact}

\noindent In \cite{Chi25b}, we used the fact above to prove that all formulas are  equivalent to a finite disjunction of certain nice formulas modulo $T\theta^C$ (the model companion of $T^C_\theta$). In the following, we present a partial result tailored to our needs in this paper:

\begin{definition}[Definition 3.1 and 3.4 in \cite{Chi25b}] \label{def_alg_pattern}\label{def_alg_pat_comp}
    An \textbf{algebraic pattern} in $\uw$ is an $L_\theta$-formula of the form
    $$
    \psi(y_1, \dots, y_m; \uw) := \bigwedge\nolimits_{k=1}^m \exists x \!\in\! \VV : y_k = r_k(x) \wedge \varphi_k(x; y_1, \dots, y_{k-1}; \uw),
    $$
    where each $r_k$ belongs to the ring $R_C$ (see Fact \ref{theorem_r_c_def}), and each $\varphi_k(x; y_1, \dots, y_{k-1}; \uw)$ is an $L$-formula algebraic in $x$.
    In practice, we identify an algebraic pattern with the set $Y$ it defines.
    In this case, we let $\psi_Y(y_1, \dots, y_m; \uw)$ denote the corresponding formula.
    Given a $|\uw|$-sized tuple $\ud$ in a model $(\mm, \theta)$ of $T_\theta \cup \set{\text{``$\theta$ is $C$-image-complete''}}$, we let $Y_\ud$ denote the set defined by $\psi_Y(y_1, \dots, y_m; \ud)$.

    Let $Y$ be an algebraic pattern in $\uw$ defined by $\psi_Y(\uy; \uw)$, let $\pi$ be a coordinate projection, and let $S(\ux; \tiluy)$ be a parametrized $C$-sequence-system.
    We say that the pair $(Y, \pi)$ is \textbf{compatible} with $S$ if
    $$
    T_\theta \cup \set{\text{``$\theta$ is $C$-image-complete''}} \models \forall\uw : \forall \uy \in Y_\uw : \text{``$\pi(\uy)$ is compatible with $S$''}.
    $$
\end{definition}

\begin{fact}[Theorem 3.5 and Lemma 3.8 in \cite{Chi25b}] \label{theorem_big_fml_preceise}
    Suppose that $T$ satisfies \Hfour{}. Any existential $L_\theta$-formula $\phi(\uz;\uw)$ is, modulo $T_\theta \cup \set{\text{``$\theta$ is $C$-image-complete''}}$, equivalent to a finite disjunction of formulas of the form
    $$
    \exists \uy \in Y_\uw : \exists \ux \in \VV : \psi_\theta(\ux; \uz; \uy\uw) \wedge S(\ux; \pi(\uy)), \quad \text{where}
    $$
    \begin{enumerate}[(i)]
        \item $S(\ux; \tiluy)$ is a parametrized $C$-sequence-system;
        \item $\psi(\uxvec; \uz; \uy\uw)$ is an $L$-formula bounded by $S$;
        \item $Y$ is an algebraic pattern in $\uw$, $\pi$ is a projection, and $(Y, \pi)$ is compatible with $S$.
    \end{enumerate}
\end{fact}

\subsection{The antichain tree property}

We now introduce the notation, definitions, and facts around the antichain tree property that are needed for the proof of our main result, i.e., that, assuming \Hfour{}, $T\theta^C$ has \NATP{} if and only if $T$ has \NATP{}.
Most of these results, as well as most of the notation, are taken from \cite{AKL23}.
Throughout this section, $T$ can be any first-order theory, not necessarily one in our usual setting from Section \ref{sec_setting}; i.e., we do not assume that there is a definable vector space.

\begin{notation}\label{language_of_trees}
    Let $\lambda$ be an ordinal and $\kappa$ a cardinal.
    \begin{enumerate}[(i)]
        \item By $\kappa^\lambda$, we mean the set of all functions from $\lambda$ to $\kappa$, i.e., the set of sequences in $\kappa$ indexed by $\lambda$.
        \item By $\kappa^{<\lambda}$, we mean the tree $\bigcup_{\alpha\in \lambda}{\kappa^\alpha}$.
        \item By $\bigemptyseq$, we mean the empty string in $\kappa^{<\lambda}$, i.e., the empty function.
        We may write elements $\nu \in \kappa^n$ as strings of elements in $\kappa$, such as $\nu := \bigseq{011 \dots 001}$ (in this example, $\nu(0) = 0$ and $\nu(n-1) = 1$).
        We may also write $\bigseq{0}^n$ to denote the string $\bigseq{0 \dots 0}$ of length $n$.
    \end{enumerate}
    Let $\nu, \mu \in \kappa^{<\lambda}$ be given.
    \begin{enumerate}[(i)]
        \setcounter{enumi}{3}
        \item By $\nu \triangleleft \mu$, we mean $\nu \subsetneq \mu$.
        If $\nu \trianglelefteq \mu$ or $\mu \trianglelefteq \nu$, then we say $\nu$ and $\mu$ are \textbf{comparable}.
        \item By $\nu \perp \mu$, we mean that $\nu \not\trianglelefteq \mu$ and $\mu \not\trianglelefteq \nu$.
        We say $\nu$ and $\mu$ are \textbf{incomparable} if $\nu \perp \mu$.
        \item By $\nu \wedge \mu$, we mean the maximal $\eta \in \kappa^{<\lambda}$ such that $\eta \trianglelefteq \nu$ and $\eta \trianglelefteq \mu$.
        \item By $l(\nu)$, we mean the domain of $\nu$.
        \item By $\nu <_\Lex \mu$, we mean that either $\nu \triangleleft \mu$, or $\nu \perp \mu$ and $\nu(l(\nu \wedge \mu)) < \mu(l(\nu \wedge \mu))$.
        \item By $\nu^\frown \mu$, we mean $\nu \cup \set{(l(\nu) + i, \mu(i)) : i < l(\mu)}$, i.e., the concatenation of the two strings.
    \end{enumerate}
    Let $X \subseteq \kappa^{<\lambda}$.
    \begin{enumerate}[(i)]
        \setcounter{enumi}{9}
        \item By $\nu^\frown X$, we mean $\set{\nu^\frown \eta : \eta \in X}$.
        \item We say that $X$ is a \textbf{path} if the elements of $X$ are pairwise comparable.
        \item We say that $X$ is an \textbf{antichain} if the elements of $X$ are pairwise incomparable, i.e., $\nu \perp \mu$ for all distinct $\nu, \mu \in X$.
    \end{enumerate}
\end{notation}

\noindent
We provide a slightly more general definition of the antichain tree property than given in the introduction.

\begin{definition}[see Definition 3.19 in \cite{AKL23}] \label{def_kATP}
    We say that a formula $\varphi(\ux; \uy)$ has the $\mathbf{k}$\textbf{-antichain tree property} ($k$-\ATP{}) in $T$ if there is a model of $T$ containing a tree-indexed set of parameters $(\ub_\nu : \nu \in \kappa^{<\omega})$ (where $\kappa \geq 2$) such that for any $X \subseteq \kappa^{<\omega}$:
    \begin{enumerate}[(i)]
        \item The type $\set{\varphi(\ux; \ub_\nu) : \nu \in X}$ is consistent if $X$ is an antichain.
        \item The type $\set{\varphi(\ux; \ub_\nu) : \nu \in X}$ is $k$-inconsistent if $X$ is a path.
    \end{enumerate}
    We say that $T$ has the $\mathbf{k}$\textbf{-antichain tree property} if there is a formula $\varphi(\ux; \uy)$ that has $k$-\ATP{} in $T$.
    In this case, we say that $\varphi(\ux; \uy)$ witnesses $k$-\ATP{} in $\ux$ together with $(\ub_\nu : \nu \in \kappa^{<\omega})$.

    We say that $\varphi(\ux; \uy)$ or $T$ has the \textbf{antichain tree property} (\ATP{}) if it has $2$-\ATP{}.
    We say that $T$ is \NATP{} if it does not have \ATP{}.
\end{definition}

\noindent We will almost always set $\lambda = \omega$ and either $\kappa = 2$ or $\kappa = \omega$ in the above.
Also, we may refer to tree-indexed sets of parameters $(\ub_\nu : \nu \in \kappa^{<\omega})$ simply as trees.

\begin{figure}[tbp]
    \centering
    \begin{tikzpicture}[
        node/.style={circle, draw, fill=white, inner sep=1pt, minimum size=6mm, font=\scriptsize},
        tree edge/.style={line width=0.35pt},
        path edge/.style={draw=blue!60!black, very thick},
        path node/.style={node, draw=blue!60!black, fill=blue!8, very thick},
        antichain node/.style={node, draw=red!65!black, fill=red!8, very thick}
    ]
        \coordinate (root) at (0,0);
        \coordinate (c0) at (-2,1);
        \coordinate (c1) at (2,1);
        \coordinate (c00) at (-3,2);
        \coordinate (c01) at (-1,2);
        \coordinate (c10) at (1,2);
        \coordinate (c11) at (3,2);
        \coordinate (c000) at (-3.5,3);
        \coordinate (c001) at (-2.5,3);
        \coordinate (c010) at (-1.5,3);
        \coordinate (c011) at (-0.5,3);
        \coordinate (c100) at (0.5,3);
        \coordinate (c101) at (1.5,3);
        \coordinate (c110) at (2.5,3);
        \coordinate (c111) at (3.5,3);

        \draw[tree edge] (root) -- (c0) -- (c00) -- (c000);
        \draw[tree edge] (c00) -- (c001);
        \draw[tree edge] (c0) -- (c01) -- (c010);
        \draw[tree edge] (c01) -- (c011);
        \draw[tree edge] (root) -- (c1) -- (c10) -- (c100);
        \draw[tree edge] (c10) -- (c101);
        \draw[tree edge] (c1) -- (c11) -- (c110);
        \draw[tree edge] (c11) -- (c111);
        \draw[path edge] (root) -- (c1) -- (c10) -- (c101);

        \node[path node] at (root) {$\bigemptyseq$};
        \node[node] at (c0) {$\bigseq{0}$};
        \node[path node] at (c1) {$\bigseq{1}$};
        \node[node] at (c00) {$\bigseq{00}$};
        \node[antichain node] at (c01) {$\bigseq{01}$};
        \node[node] at (c10) {$\bigseq{10}$};
        \node[antichain node] at (c11) {$\bigseq{11}$};
        \node[antichain node] at (c000) {\scalebox{0.72}{$\bigseq{000}$}};
        \node[node] at (c001) {\scalebox{0.72}{$\bigseq{001}$}};
        \node[node] at (c010) {\scalebox{0.72}{$\bigseq{010}$}};
        \node[node] at (c011) {\scalebox{0.72}{$\bigseq{011}$}};
        \node[node] at (c100) {\scalebox{0.72}{$\bigseq{100}$}};
        \node[path node] at (c101) {\scalebox{0.72}{$\bigseq{101}$}};
        \node[node] at (c110) {\scalebox{0.72}{$\bigseq{110}$}};
        \node[node] at (c111) {\scalebox{0.72}{$\bigseq{111}$}};
    \end{tikzpicture}
    \caption{A binary tree of depth $3$ with the root at the bottom. The set $\set{\bigseq{000}, \bigseq{01}, \bigseq{11}}$ in red forms an antichain, while the set $\set{\bigemptyseq, \bigseq{1}, \bigseq{101}}$ in blue is a path. The lexicographic order $<_\Lex$ can be read recursively: each node comes before all nodes in its upper-left subtree, and these come before all nodes in its upper-right subtree. For example, all nodes extending $\bigseq{0}$ come before all nodes extending $\bigseq{1}$.}
    \label{figure_binary_tree_path_antichain}
\end{figure}
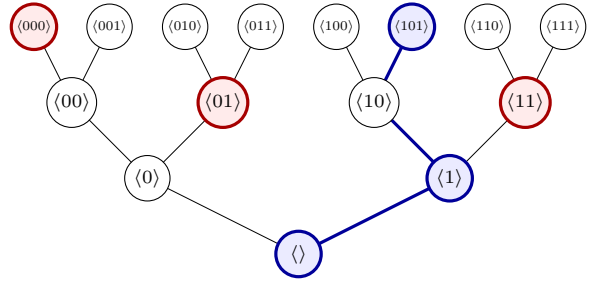
\begin{fact}[Lemma 3.20 in \cite{AKL23}] \label{fact_k_atp_equiv_atp}
    For any $k \geq 2$, a theory $T$ has $k$-\ATP{} if and only if it has \ATP{}.
\end{fact}

\begin{fact}[Lemma 3.18 in \cite{AKL23}] \label{fact_atp_disj}
    If a disjunction $\bigvee_{i=1}^q \varphi_i(\ux; \uy)$ has \ATP{} in $T$, then there is $i \in \set{1, \dots, q}$ such that $\varphi_i(\ux; \uy)$ has \ATP{} in $T$.
\end{fact}

\begin{definition} \label{def_language_for_trees}
    We consider $\kappa^{<\lambda}$ as a structure in the following languages:
    \begin{enumerate}[(i)]
        \item $L_0 := \set{\trianglelefteq, <_\Lex, \wedge}$, with everything interpreted as outlined above.
        \item $L_\delta := \set{\Delta, <_\Lex}$, where $\Delta$ is a $4$-ary relation symbol, which we interpret as follows:
        $$
        \Delta(\nu_1, \nu_2, \nu_3, \nu_4) \quad :\Leftrightarrow \quad (\nu_1 \wedge \nu_2) \trianglelefteq (\nu_3 \wedge \nu_4).
        $$
    \end{enumerate}
    We also consider $\kappa^\lambda$ as an $L_\delta$-structure (interpreted as an $L_\delta$-substructure of $\kappa^{<\lambda+1}$).
\end{definition}

\begin{definition}
    Let $\underline{\nu} = (\nu_1, \dots, \nu_n)$ and $\umu{} = (\mu_1, \dots, \mu_n)$ be finite tuples from $\kappa^{<\lambda}$.
    \begin{enumerate}[(i)]
        \item By $\operatorname{qftp}_0(\underline{\nu})$, we mean the set of quantifier-free $L_0$-formulas $\varphi(\ux)$ such that $\kappa^{<\lambda} \models \varphi(\underline{\nu})$.
        We say that $\underline{\nu}$ and $\umu{}$ are \textbf{strongly isomorphic} if $\operatorname{qftp}_0(\underline{\nu}) = \operatorname{qftp}_0(\umu{})$.
        We say that $X, Y \subseteq \kappa^{<\lambda}$ are \textbf{strongly isomorphic} if the tuples $\underline{\nu}$ and $\umu{}$ that consist of the elements of $X$ and $Y$ ordered by $<_\Lex$ are strongly isomorphic.
        \item Similarly, we say that $\underline{\nu} \simeq_\delta \umu{}$ if $\underline{\nu}$ and $\umu{}$ have the same quantifier-free $L_\delta$-type.
        We also define $X \simeq_\delta Y$ as $\underline{\nu} \simeq_\delta \umu{}$ for the tuples as in (i).
    \end{enumerate}
\end{definition}

\begin{fact}[Lemma 3.5 in \cite{AKL23}] \label{L_zero_type_eq_L_delta_type}
    Let $\underline{\nu}$ and $\umu{}$ be two tuples of the same length in $\kappa^{<\lambda}$.
    Then $\underline{\nu}$ and $\umu{}$ have the same quantifier-free $L_\delta$-type if and only if $\underline{\nu}$ and $\umu{}$ have the same quantifier-free $L_0$-type.
    Given two subsets $X, Y \subseteq \kappa^{<\lambda}$, $X$ is strongly isomorphic to $Y$ if and only if $X \simeq_\delta Y$.
\end{fact}

\noindent For our proofs, it is crucial to work with trees that are indiscernible in the sense we now define.

\begin{definition}
    Let $L$ be a language, let $\mm$ be an $L$-structure, and let $(\ua_\nu : \nu \in \kappa^{<\lambda})$ and $(\ub_\nu : \nu \in \kappa^{<\lambda})$ be parameters from $\mm$.
    Write $\ua_{\underline{\nu}} := (\ua_{\nu_1}, \dots, \ua_{\nu_n})$.
    \begin{enumerate}[(i)]
        \item We say that $(\ua_\nu : \nu \in \kappa^{<\lambda})$ is \textbf{strongly ($\mathbf{L}$-)indiscernible} if $\tp(\ua_{\underline{\nu}})=\tp(\ua_\umu{})$ holds for all $\underline{\nu}$ and $\umu{}$ with $\operatorname{qftp}_0(\underline{\nu})=\operatorname{qftp}_0(\umu{})$.
        \item We say that $(\ub_\nu : \nu \in \kappa^{<\lambda})$ is \textbf{strongly locally based} on $(\ua_\nu : \nu \in \kappa^{<\lambda})$ if for every tuple $\underline{\nu}$ and every finite set $\Gamma$ of $L$-formulas in $|\underline{\nu}|$ variables, there is a tuple $\umu{}$ such that $\operatorname{qftp}_0(\underline{\nu})=\operatorname{qftp}_0(\umu{})$ and $\mm \models \varphi(\ub_{\underline{\nu}}) \leftrightarrow \varphi(\ua_{\umu{}})$ holds for all $\varphi(\ux) \in \Gamma$.
    \end{enumerate}
\end{definition}

\noindent Next, we introduce the analog of ``Ramsey and compactness'' for $\omega^{<\omega}$-indexed trees.

\begin{fact}[Modeling property of strong indiscernibility]\label{fact_modeling_property} \label{fact_model_prop}
    Let a tree $(\ua_\nu : \nu \in \omega^{<\omega})$ be given in some $L$-structure.
    Then there is a strongly indiscernible tree $(\ub_\nu : \nu \in \omega^{<\omega})$ in some elementary extension that is strongly locally based on $(\ua_\nu : \nu \in \omega^{<\omega})$.
\begin{proof}
    Use Theorem 16 in \cite{TT12} with $\Gamma((\ux_\nu : \nu \in \omega^{<\omega}))$ being the $\operatorname{EM}$-type of $(\ua_\nu : \nu \in \omega^{<\omega})$, as defined in Definition 2.6 in \cite{Sco15}.
\end{proof}
\end{fact}

\noindent Note that the above is for trees indexed by $\omega^{<\omega}$.
We now give a weak variant for trees indexed by $2^{<\omega}$ that at least preserves properties such as ``$\set{\phi(\uz; \ua_\nu) : \nu \in X}$ is consistent for all antichains $X$'' or ``$\set{\phi(\uz; \ua_\nu) : \nu \in X}$ is $k$-inconsistent for all paths $X$''.

\begin{lemma} \label{lemma_weak_modeling}
    Fix $\kappa \geq 2$.
    Let $(\ua_\nu : \nu \in 2^{<\omega})$ be a tree of parameters in some $L$-structure $\mm$.
    Then there is a strongly indiscernible tree $(\ub_\nu : \nu \in \kappa^{<\omega})$, in some elementary extension $\mm' \succ \mm$, with the following property:
    \begin{itemize}
        \item[] Let $\varphi(\ux_1, \dots, \ux_n)$ be an $L$-formula and let $\Sigma_0$ be a partial quantifier-free $\set{\triangleleft, <_\Lex}$-type in $n$ variables such that $\mm \models \varphi(\ua_{\underline{\nu}})$ for all $\underline{\nu}$ with $\operatorname{qftp}_0(\underline{\nu}) \supseteq \Sigma_0$.
        Then $\mm' \models \varphi(\ub_{\underline{\nu}})$ for all $\underline{\nu} \in \kappa^{<\omega}$ with $\operatorname{qftp}_0(\underline{\nu}) \supseteq \Sigma_0$.
    \end{itemize}
\begin{proof}
    Fix $m \geq 1$. Let $f$ be the bijection that maps elements of $\set{0, \dots, 2^m-1}$ to their binary representations of length $m$, i.e., $f(0) = \bigseq{0}^m$, $f(1) = \bigseq{0}^{m-1}{}^\frown\bigseq{1}$, $f(2) = \bigseq{0}^{m-2}{}^\frown\bigseq{10}$, and so on.
    We define $\iota \colon (2^m)^{\leq\omega} \to 2^{\leq\omega}$ as follows: Set $\iota(\bigemptyseq) = \bigemptyseq$.
    Assuming that $\iota(\nu)$ is already defined, set $\iota(\nu^\frown \bigseq{k}) := \iota(\nu)^\frown f(k)$ for $k \in \set{0, \dots, 2^m-1}$ (see Figure \ref{figure_iota_recursive_construction} for the case $m = 2$).
    Finally, set $\iota(p) := \bigcup_{k \in \omega} \iota(p_{\restriction k})$ for any $p \in (2^m)^\omega$.
    It is clear that $\iota$ is an $\set{\triangleleft, <_\Lex}$-embedding.

    Let $\iota$ be as above.
    Define the tree $(\ub'_{\nu} : \nu \in (2^m)^{<\omega})$ by setting $\ub'_\nu := \ua_{\iota(\nu)}$.
    One can easily check that $(\ub'_{\nu} : \nu \in (2^m)^{<\omega})$ satisfies the property from the statement, since $\iota$ is an $\set{\triangleleft, <_\Lex}$-embedding.
    By compactness, letting $m$ go to infinity, we find a tree $(\ub_{\nu} : \nu \in \omega^{<\omega})$ that satisfies the property from the statement.
    Using the modeling property of strong indiscernibility, Fact \ref{fact_modeling_property}, we can assume that $(\ub_{\nu} : \nu \in \omega^{<\omega})$ is strongly indiscernible and still satisfies the property from the statement.
    If $\kappa \leq \omega$, then the subtree $(\ub_{\nu} : \nu \in \kappa^{<\omega})$ is as desired.
    If $\kappa > \omega$, then we can extend the tree to be indexed by $\kappa^{<\omega}$.
\end{proof}
\end{lemma}

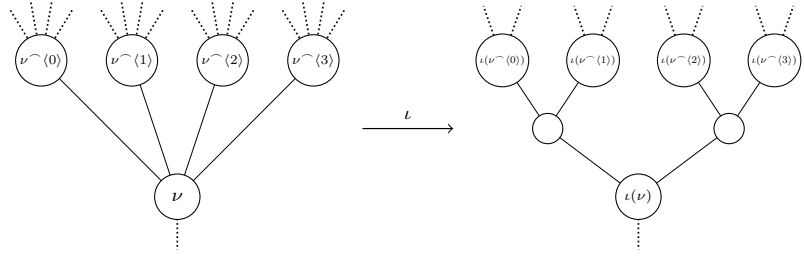
\begin{figure}[tbp]
    \centering
    \begin{tikzpicture}[
        label node/.style={circle, draw, fill=white, inner sep=1pt, minimum size=6mm, font=\scriptsize},
        long label node/.style={circle, draw, fill=white, inner sep=1pt, minimum size=6mm, font=\scriptsize},
        small node/.style={circle, draw, fill=white, inner sep=1pt, minimum size=4mm},
        tree edge/.style={line width=0.35pt},
        continuation edge/.style={densely dotted, line width=0.65pt},
        map arrow/.style={->, line width=0.45pt}
    ]
        \coordinate (leftroot) at (0,0);
        \coordinate (left0) at (-1.8,1.8);
        \coordinate (left1) at (-0.6,1.8);
        \coordinate (left2) at (0.6,1.8);
        \coordinate (left3) at (1.8,1.8);

        \draw[continuation edge] (leftroot) -- ++(0,-0.7);
        \draw[tree edge] (leftroot) -- (left0);
        \draw[tree edge] (leftroot) -- (left1);
        \draw[tree edge] (leftroot) -- (left2);
        \draw[tree edge] (leftroot) -- (left3);
        \foreach \leaf in {left0,left1,left2,left3} {
            \draw[continuation edge] (\leaf) -- ++(-0.42,0.68);
            \draw[continuation edge] (\leaf) -- ++(-0.14,0.78);
            \draw[continuation edge] (\leaf) -- ++(0.14,0.78);
            \draw[continuation edge] (\leaf) -- ++(0.42,0.68);
        }

        \node[label node] at (leftroot) {$\nu$};
        \node[long label node] at (left0) {\scalebox{0.72}{$\nu^\frown\bigseq{0}$}};
        \node[long label node] at (left1) {\scalebox{0.72}{$\nu^\frown\bigseq{1}$}};
        \node[long label node] at (left2) {\scalebox{0.72}{$\nu^\frown\bigseq{2}$}};
        \node[long label node] at (left3) {\scalebox{0.72}{$\nu^\frown\bigseq{3}$}};

        \draw[map arrow] (2.45,0.9) -- (3.65,0.9) node[midway, above, font=\scriptsize] {$\iota$};

        \coordinate (rightroot) at (6.1,0);
        \coordinate (right0) at (4.9,0.9);
        \coordinate (right1) at (7.3,0.9);
        \coordinate (right00) at (4.3,1.8);
        \coordinate (right01) at (5.5,1.8);
        \coordinate (right10) at (6.7,1.8);
        \coordinate (right11) at (7.9,1.8);

        \draw[continuation edge] (rightroot) -- ++(0,-0.7);
        \draw[tree edge] (rightroot) -- (right0) -- (right00);
        \draw[tree edge] (right0) -- (right01);
        \draw[tree edge] (rightroot) -- (right1) -- (right10);
        \draw[tree edge] (right1) -- (right11);
        \foreach \leaf in {right00,right01,right10,right11} {
            \draw[continuation edge] (\leaf) -- ++(-0.25,0.72);
            \draw[continuation edge] (\leaf) -- ++(0.25,0.72);
        }

        \node[long label node] at (rightroot) {\scalebox{0.76}{$\iota(\nu)$}};
        \node[small node] at (right0) {};
        \node[small node] at (right1) {};
        \node[long label node] at (right00) {\scalebox{0.54}{$\iota(\nu^\frown\bigseq{0})$}};
        \node[long label node] at (right01) {\scalebox{0.54}{$\iota(\nu^\frown\bigseq{1})$}};
        \node[long label node] at (right10) {\scalebox{0.54}{$\iota(\nu^\frown\bigseq{2})$}};
        \node[long label node] at (right11) {\scalebox{0.54}{$\iota(\nu^\frown\bigseq{3})$}};
    \end{tikzpicture}
    \caption{Visualization of the recursive construction of $\iota$ in the proof of Lemma \ref{lemma_weak_modeling} in the case $m=2$. One step in a $2^2$-branching tree is replaced by two binary steps.}
    \label{figure_iota_recursive_construction}
\end{figure}

\noindent The above shows the following.

\begin{fact}[see also Remark 2.8 in \cite{AKLL25}] \label{fact_infinite_branching_tree} \label{fact_ic_witness}
    If $\varphi(\ux; \uy)$ has ($k$-)\ATP{} in $T$, then for any $\kappa \geq 2$, we can find a strongly indiscernible tree $(\ub_\nu : \nu \in \kappa^{<\omega})$ in some model of $T$ such that $\varphi(\ux; \uy)$ witnesses ($k$-)\ATP{} with this tree.
\end{fact}

\noindent This also shows that our original definition of the antichain tree property, given in the introduction, coincides with the one given in this section, Definition \ref{def_kATP}.

\begin{lemma} \label{lemma_remove_alg_set}
    Fix $\kappa \geq 2$.
    Assume that $\exists \uy : \varphi(\uy; \uw) \wedge \psi(\uz; \uy\uw)$ witnesses \ATP{} and that $|\varphi(\mm; \ud)| \leq q$ holds for all $\ud$ from any model $\mm \models T$.
    Then the formula $\psi(\uz;\uy\uw)$ witnesses \ATP{} together with a strongly indiscernible tree $(\ub_\nu\ud_\nu : \nu \in \kappa^{<\omega})$ in some model $\mm \models T$ where $\mm \models \varphi(\ub_\nu; \ud_\nu)$ holds for all $\nu \in \kappa^{<\omega}$.
\begin{proof}
    Let $\exists \uy : \varphi(\uy; \uw) \wedge \psi(\uz; \uy\uw)$ witness \ATP{} together with $(\ud_\nu : \nu \in 2^{<\omega})$.
    For each $\nu \in 2^{<\omega}$, let $\ub_{\nu, 1}, \dots, \ub_{\nu, q}$ be the realizations of $\varphi(\uy; \ud_\nu)$.
    If there are fewer than $q$ realizations, then use any realization to fill the remaining entries $\ub_{\nu, i}$; there must be at least one realization for each $\nu$, since the formula $\exists \uy : \varphi(\uy; \uw) \wedge \psi(\uz; \uy\uw)$ witnesses \ATP{}.
    Then
    $$
    \bigvee\nolimits_{i=1}^q \varphi(\uy_i; \uw) \wedge \psi(\uz; \uy_i\uw) \quad \text{together with}\quad (\ub_{\nu, 1}, \dots, \ub_{\nu, q}, \ud_\nu : \nu \in 2^{<\omega})
    $$
    witnesses \ATP{}.
    Now use Fact \ref{fact_atp_disj} and Fact \ref{fact_ic_witness} to see that $\varphi(\uy; \uw) \wedge \psi(\uz; \uy\uw)$ witnesses \ATP{} together with some strongly indiscernible tree $(\ub'_\nu\ud'_\nu : \nu \in \kappa^{<\omega})$ in some model $\mm \models T$.
    Clearly $\mm \models \varphi(\ub'_\nu; \ud'_\nu)$ holds for all $\nu \in \kappa^{<\omega}$, so $\psi(\uz; \uy\uw)$ also witnesses \ATP{} together with $(\ub'_\nu\ud'_\nu : \nu \in \kappa^{<\omega})$.
\end{proof}
\end{lemma}

\noindent Notice that paths are $L_0$-substructures of $\kappa^{<\lambda}$; however, antichains with more than one element are not $L_0$-substructures, as they are not closed under $\wedge$.
On the other hand, $L_\delta$ is a purely relational language, so any subset of an $L_\delta$-structure is an $L_\delta$-substructure.
The language $L_\delta$ was introduced in \cite{AKL23} to generalize the so-called ``Path-Collapse Lemma'' (Lemma 4.3 of \cite{CR16}) to antichains.
This lemma is used to show that if a theory has \SOPtwo{}, then this is witnessed by a formula $\varphi(x; \uy)$ with $x$ being a singleton (Corollary 4.1 in \cite{CR16}).
Using the language $L_\delta$, it was proved that the antichain tree property is also always witnessed by a formula $\varphi(x; \uy)$ with $x$ being a singleton (Theorem 3.17 in \cite{AKL23}).
The language $L_\delta$ can also be used to generalize Fact \ref{fact_modeling_property} to antichains; see Lemma 3.7 in \cite{AKL23}.
We only need Fact \ref{L_zero_type_eq_L_delta_type}, i.e., that the languages $L_0$ and $L_\delta$ express essentially the same information on trees, and the fact that the class of all finite substructures of $\omega^\omega$ is a Ramsey class; see Fact \ref{fact_is_ramsey} below.
We recall the definition of a Ramsey class.

\begin{definition}
    Let $L$ be a language.
    \begin{enumerate}[(i)]
        \item Let $M$ be an $L$-structure.
        By $\operatorname{age}(M)$, we mean the class of all finitely generated $L$-structures that are isomorphic to a substructure of $M$.
        \item Let $A$ and $C$ be $L$-structures.
        We write $\binom{C}{A}$ to denote the class of $L$-substructures $A' \subseteq C$ isomorphic to $A$.
        \item A \textbf{$\mathbf{k}$-coloring} of a set $A$ is a map from $A$ to $k$.
        \item Let $\kk$ be a class of $L$-structures, let $A, B, C \in \kk$, and let $k \in \omega\setminus\set{0}$.
        By $C \to (B)^A_k$, we mean that for all $k$-colorings $f$ of $\binom{C}{A}$, there exists $B' \subseteq C$ such that $B'$ is $L$-isomorphic to $B$ and $f{\restriction\!\binom{B'}{A}}$ is constant.
        \item A class $\kk$ of $L$-structures is called a \textbf{Ramsey class} if, for any $A, B \in \kk$ and $k \in \omega\setminus\set{0}$, there exists $C \in \kk$ such that $C \to (B)^A_k$.
    \end{enumerate}
\end{definition}

\begin{fact}[\cite{AKL23}] \label{fact_is_ramsey}
    The class of $L_\delta$-structures $\operatorname{age}(\omega^\omega)$ is a Ramsey class.
\begin{proof}
    Use Lemma 3.7 in \cite{AKL23} in combination with Theorem 3.12 in \cite{Sco15}.
    One can easily check that $\omega^\omega$ is, as an $L_\delta$-structure, locally finite, ordered by $<_\Lex$, and \texttt{qfi}.
    Indeed, local finiteness means that finitely generated substructures are finite; also, for each $n$, there are only finitely many quantifier-free $L_\delta$-types, so we can use Observation 2.1 and (1) of Proposition 1 in \cite{Sco15} to show that $\omega^\omega$ is \texttt{qfi}.
\end{proof}
\end{fact}

\noindent We end this section by presenting some lemmas that help us simplify intersections of algebraic sets definable using parameters in a strongly indiscernible tree.

\begin{observation}
\label{observation_terms_all_equal}
    Let $(\ub_\nu : \nu \in \kappa^{<\omega})$ be strongly indiscernible and let $t(\uy)$ be an $L$-term.
    If there are distinct $\mu, \eta \in \kappa^{<\omega}$ with $t(\ub_\mu) = t(\ub_\eta)$, then $t(\ub_\nu)$ does not depend on $\nu \in \kappa^{<\omega}$.
\begin{proof}
    If $\mu$ and $\eta$ are comparable, then, after swapping them, we may assume that $\mu \triangleleft \eta$.
    By strong indiscernibility, we then have $t(\ub_\nu) = t(\ub_{\emptyseq})$ for any $\nu \in \kappa^{<\omega}$.
    Thus, suppose that $\mu$ and $\eta$ form an antichain.
    Strong indiscernibility gives $t(\ub_\seq{0}) = t(\ub_\seq{1}) = t(\ub_\seq{00})$, since both $\set{\bigseq{0}, \bigseq{1}}$ and $\set{\bigseq{00}, \bigseq{1}}$ form antichains strongly isomorphic to $\set{\mu, \eta}$.
    Now conclude since $\bigseq{0} \triangleleft \bigseq{00}$.
\end{proof}
\end{observation}

\begin{lemma} \label{lemma_make_intersection_definable}
    Let $\varphi(\uy; \uw_1\uw_2)$ be an $L$-formula algebraic in $\uy$, and let $(\ud_\nu : \nu \in 2^{<\omega})$ be a strongly indiscernible tree in some model $\mm \models T$.
    Then we can find a strongly indiscernible tree $(\ud'_\nu U'_\nu : \nu \in 2^{<\omega})$ in some elementary extension $\mm' \succ \mm$, where each $U'_\nu$ is a finite set, such that:
    \begin{enumerate}[(i)]
        \item $\varphi(\mm'; \ud'_\emptyseq\ud'_\mu) \cap \varphi(\mm'; \ud'_\emptyseq\ud'_\eta)$ is contained in $U'_\emptyseq{}^{|\uy|}$ for any $\mu, \eta \in 2^{<\omega}$ with $\mu \perp \eta$.
        \item For every $L$-formula $\chi(\ux_1, \dots, \ux_n)$ and every quantifier-free $\set{\triangleleft, <_\Lex}$-type $\Sigma_0$ in $n$ variables, if $\mm \models \chi(\ud_{\underline{\nu}})$ for all $\underline{\nu}$ with $\operatorname{qftp}_0(\underline{\nu}) \supseteq \Sigma_0$, then $\mm' \models \chi(\ud'_{\underline{\nu}})$ for all $\underline{\nu}$ with $\operatorname{qftp}_0(\underline{\nu}) \supseteq \Sigma_0$.
    \end{enumerate}
\begin{proof}
    Let $\mu, \eta \in 2^{<\omega}$ be such that $\mu \perp \eta$ and $\mu <_\Lex \eta$.
    Set $\xi := \mu \wedge \eta$ and choose paths $p_0, p_1 \in 2^\omega$ such that $\mu = \xi^\frown p_{0\restriction n_0}$ and $\eta = \xi^\frown p_{1\restriction n_1}$ for some $n_0, n_1 > 0$.
    Note that $p_0(0) = 0$ and $p_1(0) = 1$.

    By our assumption, there is some $q \geq 0$ such that $|\varphi(\mm; \ud_0\ud_1)| \leq q$ for all parameter tuples $\ud_0\ud_1$ from $M$.
    For all $i > 0$, the set $\varphi(\mm; \ud_\emptyseq\ud_{\mu}) \cap \varphi(\mm; \ud_\emptyseq\ud_{\xi^\frown p_{1\restriction i}})$ is a subset of $\varphi(\mm; \ud_\emptyseq\ud_{\mu})$.
    Since $|\varphi(\mm; \ud_\emptyseq\ud_{\mu})| \leq q$, we can find $i, j$ with $0 < i < j$ such that
    $$
    \varphi(\mm; \ud_\emptyseq\ud_{\mu}) \cap \varphi(\mm; \ud_\emptyseq\ud_{\xi^\frown p_{1\restriction i}}) = \varphi(\mm; \ud_\emptyseq\ud_{\mu}) \cap \varphi(\mm; \ud_\emptyseq\ud_{\xi^\frown p_{1\restriction j}}).
    $$
    By strong indiscernibility of $(\ud_\nu : \nu \in 2^{<\omega})$, this gives
    $$
    \varphi(\mm; \ud_\emptyseq\ud_{\mu}) \cap \varphi(\mm; \ud_\emptyseq\ud_{\eta}) = \varphi(\mm; \ud_\emptyseq\ud_{\mu}) \cap \varphi(\mm; \ud_\emptyseq\ud_{\xi^\frown \seq{1}}).
    $$
    With the same argument and $\xi = \mu \wedge \eta$, we obtain
    \begin{align}
        \varphi(\mm; \ud_\emptyseq\ud_{\mu}) \cap \varphi(\mm; \ud_\emptyseq\ud_{\eta}) = \varphi(\mm; \ud_\emptyseq\ud_{(\mu \wedge \eta)^\frown \seq{0}}) \cap \varphi(\mm; \ud_\emptyseq\ud_{(\mu \wedge \eta)^\frown \seq{1}}). \label{tag_idk_what_this_is}
    \end{align}
    Now fix any $\xi \neq \bigemptyseq, \bigseq{0}, \bigseq{1}$ and set $N := 2^q + 1$.
    For now, assume that $\xi(0) = \xi(1) = 0$.
    For all $i$ with $0 \leq i < N$, we obtain
    \begin{align*}
        \varphi(\mm; \ud_\emptyseq\ud_{\xi^\frown \seq{0}^i{}^\frown \seq{0}}) \cap \varphi(\mm; \ud_\emptyseq\ud_{\xi^\frown \seq{0}^i{}^\frown \seq{1}}) = \varphi(\mm; \ud_\emptyseq\ud_{\xi^\frown \seq{0}^{N}}) \cap \varphi(\mm; \ud_\emptyseq\ud_{\xi^\frown \seq{0}^i{}^\frown \seq{1}})
    \end{align*}
    by (\ref{tag_idk_what_this_is}).
    Since $\varphi(\mm; \ud_\emptyseq\ud_{\xi^\frown \seq{0}^{N}})$ has at most $q$ elements, the family of its intersections with the sets $\varphi(\mm; \ud_\emptyseq\ud_{\xi^\frown \seq{0}^i{}^\frown \seq{1}})$ has at most $2^q$ members.
    Hence, there are $i, j$ with $0 \leq i < j < N$ such that
    \begin{align}
        \varphi(\mm; \ud_\emptyseq\ud_{\xi^\frown \seq{0}^{N}}) \cap \varphi(\mm; \ud_\emptyseq\ud_{\xi^\frown \seq{0}^i{}^\frown \seq{1}}) = \varphi(\mm; \ud_\emptyseq\ud_{\xi^\frown \seq{0}^{N}}) \cap \varphi(\mm; \ud_\emptyseq\ud_{\xi^\frown \seq{0}^j{}^\frown \seq{1}}). \label{tag_idk_2}
    \end{align}
    By strong indiscernibility, this equality holds for any $i, j$ with $0 \leq i < j < N$.
    Since we have $\xi(0) = \xi(1) = 0$, the nodes $\bigseq{01}$ and $\xi^\frown\bigseq{1}$ have the same quantifier-free $L_0$-type over $(\bigemptyseq, \xi^\frown\bigseq{0}^N, \xi^\frown\bigseq{0}^\frown\bigseq{1})$ (see Figure \ref{figure_same_qftp_over_base}). Together with the second equality below and strong indiscernibility of our tree, this gives the third equality below:
    \begin{align*}
        \varphi(\mm; \ud_\emptyseq\ud_{\xi^\frown \seq{0}}) \cap \varphi(\mm; \ud_\emptyseq\ud_{\xi^\frown \seq{1}}) &\underset{(\ref{tag_idk_what_this_is})}{=} \varphi(\mm; \ud_\emptyseq\ud_{\xi^\frown \seq{0}^N}) \cap \varphi(\mm; \ud_\emptyseq\ud_{\xi^\frown \seq{1}}) \\
        &\underset{(\ref{tag_idk_2})}{=} \varphi(\mm; \ud_\emptyseq\ud_{\xi^\frown \seq{0}^N}) \cap \varphi(\mm; \ud_\emptyseq\ud_{\xi^\frown\seq{0}^\frown \seq{1}}) \\
        &\underset{\hphantom{(\ref{tag_idk_what_this_is})}}{=} \varphi(\mm; \ud_\emptyseq\ud_{\xi^\frown \seq{0}^N}) \cap \varphi(\mm; \ud_\emptyseq\ud_{\seq{01}}) \\
        &\underset{(\ref{tag_idk_what_this_is})}{=} \varphi(\mm; \ud_\emptyseq\ud_{\seq{00}}) \cap \varphi(\mm; \ud_\emptyseq\ud_{\seq{01}})
    \end{align*}
    Generalizing this to the cases in which $\xi(0) \neq 0$ or $\xi(1) \neq 0$, we obtain:
    \begin{enumerate}[(a)]
        \item $\varphi(\mm; \ud_\emptyseq\ud_{\xi^\frown \seq{0}}) \cap \varphi(\mm; \ud_\emptyseq\ud_{\xi^\frown \seq{1}}) = \varphi(\mm; \ud_\emptyseq\ud_{\seq{00}}) \cap \varphi(\mm; \ud_\emptyseq\ud_{\seq{01}})$ if $\xi(0) = 0$.
        \item $\varphi(\mm; \ud_\emptyseq\ud_{\xi^\frown \seq{0}}) \cap \varphi(\mm; \ud_\emptyseq\ud_{\xi^\frown \seq{1}}) = \varphi(\mm; \ud_\emptyseq\ud_{\seq{10}}) \cap \varphi(\mm; \ud_\emptyseq\ud_{\seq{11}})$ if $\xi(0) = 1$.
    \end{enumerate}
    Note that the equations in (a) and (b) also follow from (\ref{tag_idk_what_this_is}) if $\xi = \bigseq{0}$ or $\xi = \bigseq{1}$, respectively.
    If $\xi = \bigemptyseq$, then we trivially have
    $$
    \varphi(\mm; \ud_\emptyseq\ud_{\xi^\frown \seq{0}}) \cap \varphi(\mm; \ud_\emptyseq\ud_{\xi^\frown \seq{1}}) = \varphi(\mm; \ud_\emptyseq\ud_{\seq{0}}) \cap \varphi(\mm; \ud_\emptyseq\ud_{\seq{1}}).
    $$
    For any $\nu \in 2^{<\omega}$, define $U_{\nu}$ as the set of entries of the tuples contained in the set
    \begin{align*}
        \big(\varphi(\mm; \ud_\nu\ud_{\nu^\frown\seq{00}}) \cap \varphi(\mm; \ud_\nu\ud_{\nu^\frown\seq{01}})\big) &\cup \big(\varphi(\mm; \ud_\nu\ud_{\nu^\frown\seq{10}}) \cap \varphi(\mm; \ud_\nu\ud_{\nu^\frown\seq{11}})\big) \\& \hspace{60pt} \cup \big(\varphi(\mm; \ud_\nu\ud_{\nu^\frown\seq{0}}) \cap \varphi(\mm; \ud_\nu\ud_{\nu^\frown\seq{1}})\big)
    \end{align*}
    By the above, (\ref{tag_idk_what_this_is}), and strong indiscernibility, we obtain
    $$
    U_\nu^{|\uy|} \supseteq \varphi(\mm; \ud_\nu\ud_{\mu}) \cap \varphi(\mm; \ud_\nu\ud_{\eta})
    $$
    for any $\nu, \mu, \eta \in 2^{<\omega}$ with $\nu \triangleleft \mu, \eta$ and $\mu \perp \eta$.
    Now apply Lemma \ref{lemma_weak_modeling}, the weak variant of the modeling property for strongly indiscernible trees indexed by $2^{<\omega}$, to the tree $(\ud_\nu U_\nu : \nu \in 2^{<\omega})$ to conclude.
\end{proof}
\end{lemma}

        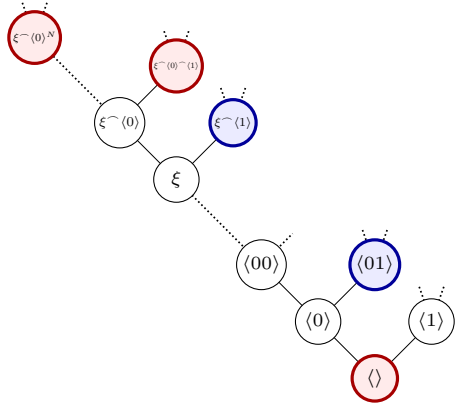
\begin{figure}[tbp]
        \centering
        \begin{tikzpicture}[
            x=0.75cm,
            y=0.75cm,
            node/.style={circle, draw, fill=white, inner sep=1pt, minimum size=6mm, font=\scriptsize},
            long label node/.style={circle, draw, fill=white, inner sep=1pt, minimum size=6mm, font=\scriptsize},
            tree edge/.style={line width=0.35pt},
            continuation edge/.style={densely dotted, line width=0.65pt}
        ]
            \coordinate (root) at (0,0);
            \coordinate (n0) at (-1,1);
            \coordinate (n1) at (1,1);
            \coordinate (n00) at (-2,2);
            \coordinate (n01) at (0,2);
            \coordinate (xi) at (-3.5,3.5);
            \coordinate (xi0) at (-4.5,4.5);
            \coordinate (xi1) at (-2.5,4.5);
            \coordinate (xi0N) at (-6,6);
            \coordinate (xi01) at (-3.5,5.5);

            \draw[tree edge] (root) -- (n0) -- (n00);
            \draw[tree edge] (root) -- (n1);
            \draw[tree edge] (n0) -- (n01);
            \draw[continuation edge] (n00) -- (xi);
            \draw[continuation edge] (n00) -- ++(0.55,0.55);
            \draw[tree edge] (xi) -- (xi0);
            \draw[tree edge] (xi) -- (xi1);
            \draw[continuation edge] (xi0) -- (xi0N);
            \draw[tree edge] (xi0) -- (xi01);
            \foreach \leaf in {n1,n01,xi1,xi0N,xi01} {
                \draw[continuation edge] (\leaf) -- ++(-0.22,0.62);
                \draw[continuation edge] (\leaf) -- ++(0.22,0.62);
            }

            \node[node, draw=red!65!black, fill=red!8, very thick] at (root) {$\bigemptyseq$};
            \node[node] at (n0) {$\bigseq{0}$};
            \node[node] at (n1) {$\bigseq{1}$};
            \node[node] at (n00) {$\bigseq{00}$};
            \node[node, draw=blue!60!black, fill=blue!8, very thick] at (n01) {$\bigseq{01}$};
            \node[node] at (xi) {$\xi$};
            \node[long label node] at (xi0) {\scalebox{0.72}{$\xi^\frown\bigseq{0}$}};
            \node[long label node, draw=blue!60!black, fill=blue!8, very thick] at (xi1) {\scalebox{0.65}{$\xi^\frown\bigseq{1}$}};
            \node[long label node, draw=red!65!black, fill=red!8, very thick] at (xi0N) {\scalebox{0.58}{$\xi^\frown\bigseq{0}^N$}};
            \node[long label node, draw=red!65!black, fill=red!8, very thick] at (xi01) {\scalebox{0.43}{$\xi^\frown\bigseq{0}^\frown\bigseq{1}$}};
        \end{tikzpicture}
        \caption{The blue nodes $\bigseq{01}$ and $\xi^\frown\bigseq{1}$ have the same quantifier-free $L_0$-type over the set consisting of all red nodes.}
        \label{figure_same_qftp_over_base}
    \end{figure}

\begin{lemma} \label{lemma_fint_set_eq_or}
    Let $\varphi(\uy; \uw_1\uw_2)$ be an $L$-formula algebraic in $\uy$, and let $(\ud_\nu : \nu \in 2^{<\omega})$ be a strongly indiscernible tree.
    Then we have
    $$
    \varphi(\mm; \ud_{\seq{0}^{i}}\ud_{\seq{0}^{k}}) \cap \varphi(\mm; \ud_{\seq{0}^{j}}\ud_{\seq{0}^{k}}) =  \varphi(\mm; \ud_{\emptyseq}\ud_{\seq{0}^{k}}) \cap \varphi(\mm; \ud_{\seq{0}}\ud_{\seq{0}^{k}})
    $$
    for all $i, j, k$ with $0 \leq i < j < k$.
\begin{proof}
    By strong indiscernibility of the tree, there is some $q \in \omega$ such that $\varphi(\mm; \ud_{\seq{0}^{i}}\ud_{\seq{0}^{j}})$ contains exactly $q$ elements for all $i, j$ with $i < j$.
    Define $N := 2^{q} + 2$.
    By the pigeonhole principle, there must be $j, k$ with $0 < j < k < N$ such that
    $$
    \varphi(\mm; \ud_{\seq{0}^{0}}\ud_{\seq{0}^{N}}) \cap \varphi(\mm; \ud_{\seq{0}^{j}}\ud_{\seq{0}^{N}}) = \varphi(\mm; \ud_{\seq{0}^{0}}\ud_{\seq{0}^{N}}) \cap \varphi(\mm; \ud_{\seq{0}^{k}}\ud_{\seq{0}^{N}}).
    $$
    By strong indiscernibility of $(\ud_\nu : \nu \in 2^{<\omega})$, it follows that, for all $i, j, k, l$ with $0 \leq i < j < k < l$,
    \begin{align}
        \varphi(\mm; \ud_{\seq{0}^{i}}\ud_{\seq{0}^{l}}) \cap \varphi(\mm; \ud_{\seq{0}^{j}}\ud_{\seq{0}^{l}}) = \varphi(\mm; \ud_{\seq{0}^{i}}\ud_{\seq{0}^{l}}) \cap \varphi(\mm; \ud_{\seq{0}^{k}}\ud_{\seq{0}^{l}}). \label{tag_eq_sets_lol}
    \end{align}
    Similarly, we can find $i, j$ with $0 \leq i < j < N - 1$ such that
    $$
    \varphi(\mm; \ud_{\seq{0}^{i}}\ud_{\seq{0}^{N}}) \cap \varphi(\mm; \ud_{\seq{0}^{N-1}}\ud_{\seq{0}^{N}}) = \varphi(\mm; \ud_{\seq{0}^{j}}\ud_{\seq{0}^{N}}) \cap \varphi(\mm; \ud_{\seq{0}^{N-1}}\ud_{\seq{0}^{N}}).
    $$
    Together with strong indiscernibility of $(\ud_\nu : \nu \in 2^{<\omega})$, this implies that, for all $i, j, k, l$ with $0 \leq i < j < k < l$,
    \begin{align}
        \varphi(\mm; \ud_{\seq{0}^{i}}\ud_{\seq{0}^{l}}) \cap \varphi(\mm; \ud_{\seq{0}^{k}}\ud_{\seq{0}^{l}}) = \varphi(\mm; \ud_{\seq{0}^{j}}\ud_{\seq{0}^{l}}) \cap \varphi(\mm; \ud_{\seq{0}^{k}}\ud_{\seq{0}^{l}}). \label{tag_eq_sets_lol2}
    \end{align}
    Now fix $i, j, k$ with $0 \leq i < j < k$.
    With the above, we immediately obtain
    \begin{align*}
        \varphi(\mm; \ud_{\seq{0}^{i}}\ud_{\seq{0}^{k}}) \cap \varphi(\mm; \ud_{\seq{0}^{j}}\ud_{\seq{0}^{k}}) &\underset{(\ref{tag_eq_sets_lol2})}{=} \varphi(\mm; \ud_{\seq{0}^{0}}\ud_{\seq{0}^{k}}) \cap \varphi(\mm; \ud_{\seq{0}^{j}}\ud_{\seq{0}^{k}}) \\
        &\underset{(\ref{tag_eq_sets_lol})}{=} \varphi(\mm; \ud_{\seq{0}^{0}}\ud_{\seq{0}^{k}}) \cap \varphi(\mm; \ud_{\seq{0}^{1}}\ud_{\seq{0}^{k}})
    \end{align*}
    Since $\bigseq{0}^0 = \bigemptyseq$, this completes the proof.
\end{proof}
\end{lemma}

\section{Preservation of NATP in the case $C = C_0$} \label{sec_natp_ez}

Recall that $C_0$ is the unique transcendental kernel configuration that asks for all $\rho[\theta]$'s to be injective (unless $\rho = 0$). In an existentially closed model of $T\theta^{C_0}$, every $\rho[\theta]$ is invertible (unless $\rho = 0$), and all $\LKThe$-definable endomorphisms of $\VV$ are of the form $\rho[\theta] \circ \eta[\theta]^{-1}$ with $\eta \neq 0$. This means that the ring $R_{C_0}$ from Fact \ref{theorem_r_c_def} is isomorphic to $K(X)$. Additionally, any $C_0$-sequence-system $S(\ux)$ is simply $\top$, which significantly simplifies our characterizations of existentially closed models and definable sets (see Theorem \ref{theorem_big_characterization} and Fact \ref{theorem_big_fml_preceise}). The goal of this section is to prove the following theorem:

\begin{theorem}\label{theorem_preserve_natp_ez}
    Suppose that $T$ satisfies \Hfour{}. If $T$ has \NATP{}, then $T\theta^{C_0}$ also has \NATP{}.
\end{theorem}

\noindent In Section \ref{sec_proof_general}, we show that the theorem above also holds for other kernel configurations $C$, but the proof is significantly more technical. Hence, we present a considerably less technical proof for the case $C = C_0$, which still captures the underlying idea. This idea consists of essentially two steps:
\begin{enumerate}[(I)]
    \item Show that if $T\theta^{C_0}$ has the antichain tree property, then there is an $L$-formula $\psi(\ux_\sh\ux_\va; \uz; \uw)$ and a strongly $L$-indiscernible tree $(\ud_\nu : \nu \in 2^{<\omega})$ in some $\mm \models T$ such that:
    \begin{enumerate}[(i)]
        \item The type $\set{\psi(\ux_\sh\ux_\nu; \uz; \ud_\nu) : \nu \in X}$ implies no finite disjunction of non-trivial linear dependencies in $\ux_\sh(\ux_\nu : \nu \in X)$ over $\VV$ whenever $X \subseteq 2^{<\omega}$ is an antichain.
        \item The type $\set{\psi(\ux_\sh\ux_\emptyseq; \uz; \ud_\emptyseq), \psi(\ux_\sh\ux_{\seq{0}}; \uz; \ud_{\seq{0}})}$ implies some finite disjunction of non-trivial linear dependencies in $\ux_\sh\ux_\emptyseq\ux_{\seq{0}}$ over $\VV$.
    \end{enumerate}
    This can essentially be seen as $T$ having some kind of ``linearly dependent antichain tree property''.
    \item Show that this ``linearly dependent antichain tree property'' implies the actual antichain tree property.
\end{enumerate}
Step (II) works entirely within the theory $T$ and is hence completely independent of the kernel configuration $C$; this means that we can reuse it in Section \ref{sec_proof_general}.

\subsection{Step (I): Reduction to a problem in $T$ in the case $C = C_0$} \label{sec_step_red_ez}

\noindent The goal of this section is to prove the following lemma.
It essentially states that if $T\theta^{C_0}$ has \ATP{}, then $T$ satisfies a property similar to \ATP{}, but with consistency in Definition \ref{def_kATP} replaced by ``implies no finite disjunction of non-trivial linear dependencies''.

\begin{lemma} \label{lemma_atp_step_11_ez}
    If the theory $T\theta^{C_0}$ has \ATP{}, then there is an $L$-formula $\psi(\ux_\sh\ux_\va; \uz; \uw)$ and a strongly $L$-indiscernible tree $(\ud_\nu : \nu \in 2^{<\omega})$ in some $\mm \models T$ such that:
    \begin{enumerate}[(i)]
        \item The partial type $\set{\psi(\ux_\sh\ux_\nu; \uz; \ud_\nu) : \nu \in X}$ implies no finite disjunction of non-trivial linear dependencies in $\ux_\sh(\ux_\nu : \nu \in X)$ over $\VV$ whenever $X \subset 2^{<\omega}$ is an antichain.
        \item The partial type $\set{\psi(\ux_\sh\ux_\emptyseq; \uz; \ud_\emptyseq), \psi(\ux_\sh\ux_{\seq{0}}; \uz; \ud_{\seq{0}})}$ implies some finite disjunction of non-trivial linear dependencies in $\ux_\sh\ux_\emptyseq\ux_{\seq{0}}$ over $\VV$.
    \end{enumerate}
\end{lemma}

\noindent The following notation is already used in the lemma above.

\begin{notation}
    If a tuple of variables $\ux_\va$ is written with ``\hspace{0.9pt}$\operatorname{va}$\hspace{-1pt}'' as a subscript, then, for any $\nu \in \omega^{<\omega}$, we let $\ux_\nu$ be a tuple of the same length.
\end{notation}

\begin{proof}[Proof of Lemma \ref{lemma_atp_step_11_ez}]
    Assume that $T\theta^{C_0}$ has \ATP{}.
    Using our characterization of definable sets, namely Fact \ref{theorem_big_fml_preceise}, and Lemma \ref{lemma_remove_alg_set}, we obtain a formula of the form
    $$
     \exists \ux \in \VV : \psi_\theta(\ux; \uz; \uw),
    $$
    where $\psi(\uxvec{}; \uz; \uw)$ is an $L$-formula, which witnesses \ATP{} in $\uz$ together with a strongly indiscernible tree $(\ud_\nu : \nu \in \omega^{<\omega})$.
    Here, since $C = C_0$, every parametrized $C$-sequence-system is just $\top$.
    Also recall our \placeholderNotation{}: The formula $\psi_\theta(\ux; \uz; \uw)$ is the $L_\theta$-formula obtained by replacing all placeholder variables $x^i_k$, which make up $\uxvec{}$, in $\psi(\uxvec{}; \uz; \uw)$ with $\theta^i(x_k)$.
    Let $\ux_\sh$ be the empty tuple and set $\ux_\va := \ux$.
    Now the formula
    $$
    \ux_\sh \in \VV \wedge \exists \ux_\va \in \VV : \psi_\theta(\ux_\sh\ux_\va; \uz; \uw)
    $$
    witnesses \ATP{} in $\ux_\sh\uz$ together with the tree $(\ud_\nu : \nu \in \omega^{<\omega})$. Throughout this entire proof, we work in a sufficiently large monster model $(\MM, \theta) \models T\theta^{C_0}$ that contains this tree. Note that we will still write $M$ to denote the underlying set of $\MM$.

    Since $(\NN^2, <_\Lex)$ is a well-ordered set, we can assume that there is no other $L$-formula $\psi'(\uxvec{}'_\sh\uxvec{}'_\va; \uz; \uw')$ such that the formula $\ux'_\sh \in \VV \wedge \exists \ux'_\va \in \VV : \psi'_\theta(\ux'_\sh\ux'_\va; \uz; \uw')$ witnesses \ATP{} and $(|\ux'_\va|, |\ux'_\sh|) <_\Lex (|\ux_\va|, |\ux_\sh|)$, where $<_\Lex$ treats the first entry as more significant.
    Note that $\ux_\sh$ may now be non-empty.

    Assume, toward a contradiction, that there is some antichain $X \subset 2^{<\omega}$ such that the partial type $\set{\psi(\uxvec{}_\sh\uxvec{}_\nu; \uz; \ud_\nu) : \nu \in X}$ implies some finite disjunction of non-trivial linear dependencies in $\uxvec{}_\sh(\uxvec{}_\nu : \nu \in X)$ over $\VV$.
    By compactness, we may assume that $X$ is finite.
    Notice that, even though $X \subset 2^{<\omega}$, we need $(\ud_\nu : \nu \in \omega^{<\omega})$ to be indexed by $\omega^{<\omega}$ in order to apply Ramsey theory in the following claim.

\begin{subclaim} \label{lemma_remove_disj_ez}
    There is an $L$-formula $\varphi(y; (\uw_\nu : \nu \in X))$ algebraic in $y$ and a non-trivial $\LKThe$-term $\lambda(\ux_\sh(\ux_{\nu} : \nu \in X))$ such that the following holds:
\begin{itemize}
    \item[$(\star)$] Given any antichain $Y \subset 2^{<\omega}$, we can find a realization $\uv_\sh(\uv_\nu : \nu \in Y)\ua$ of the partial type $\set{\psi_\theta(\ux_\sh\ux_\nu; \uz; \ud_\nu) : \nu \in Y}$ such that $\uv_\sh(\uv_{\nu} : \nu \in Y) \in \VV$, and for any antichain $X' \subseteq Y$ that is strongly isomorphic to $X$, we have
$$
(\MM, \theta) \models \varphi(\lambda(\uv_\sh(\uv_{\nu} : \nu \in X')); (\ud_\nu : \nu \in X')).
$$
\end{itemize}
\begin{innerproof}
    Using Fact \ref{lemma_hfour_fml}, and recalling that the theory $T$ satisfies \Hfour{}, we see that the partial type
    $
    \set{\psi(\uxvec{}_{\sh}\uxvec{}_{\nu}; \uz; \ud_{\nu}) : \nu \in X}
    $
    must imply a disjunction of the form
    $$
    \bigvee\nolimits_{i=1}^q \varphi_i\big( \lambda_i(\uxvec{}_{\sh}(\uxvec{}_{\nu} : \nu \in X)) ; (\ud_{\nu} : \nu \in X)\big),
    $$
    where each $\lambda_i(\uxvec{}_{\sh}(\uxvec{}_{\nu} : \nu \in X))$ is a non-zero $\LK$-term and each $\varphi_i(y; (\uw_{\nu} : \nu \in X))$ is algebraic in $y$.
    By the indiscernibility of $(\ud_{\nu} : \nu \in \omega^{<\omega})$, we see that
    $
    \set{\psi_{\theta}(\ux_{\sh}\ux_{\nu}; \uz; \ud_{\nu}) : \nu \in X'}
    $
    implies
    $$
    \bigvee\nolimits_{i=1}^q \varphi_i\big( \lambda_{i, \theta}(\ux_{\sh}(\ux_{\nu} : \nu \in X')) ; (\ud_{\nu} : \nu \in X')\big)
    $$
    for any $X' \subset \omega^{<\omega}$ strongly isomorphic to $X$.

    We now show that there is some $i_0 \in \set{1, \dots, q}$ such that
    $$
    \varphi(y; (\uw_\nu : \nu \in X)) := \varphi_{i_0}(y; (\uw_\nu : \nu \in X))\quad \text{and}\quad\lambda(\ux_\sh(\ux_\nu : \nu \in X)) := \lambda_{i_0, \theta}(\ux_\sh(\ux_\nu : \nu \in X))
    $$
    satisfy the conclusion of this claim.
    Since the partial type $\set{\psi(\uxvec{}_{\sh}\uxvec{}_{\nu}; \uz; \ud_{\nu}) : \nu \in X}$ is consistent, we have $q > 0$.
    First, note that it is enough to show the existence of such an $i_0$ for any fixed finite antichain $Y \subset \omega^{<\omega}$.
    To see this, use compactness, the strong indiscernibility of $(\ud_{\nu} : \nu \in \omega^{<\omega})$, and the fact that given two finite antichains $Y_1, Y_2 \subset \omega^{<\omega}$, there is another finite antichain $Y \subset \omega^{<\omega}$ such that both $Y_1$ and $Y_2$ can be strongly embedded into $Y$.

    From now on, we treat all antichains and subsets of trees as $L_\delta$-structures, where $L_\delta$ is the language introduced in (ii) of Definition \ref{def_language_for_trees}.
    Recall that we write $A \simeq_\delta B$ if $A \subseteq \omega^{\leq \omega}$ and $B \subseteq \omega^{\leq \omega}$, written as lexicographically ordered tuples, have the same quantifier-free $L_\delta$-type, i.e., if they are isomorphic as $L_\delta$-structures.
    Recall Fact \ref{L_zero_type_eq_L_delta_type}, which states that if $A, B \subseteq \kappa^{<\omega}$, then $A \simeq_\delta B$ if and only if $A$ and $B$ are strongly isomorphic.

    Now let $Y \subset \omega^{<\omega}$ be a finite antichain.
    Since $\operatorname{age}(\omega^{\omega})$ is a Ramsey class, by Fact \ref{fact_is_ramsey}, there is a finite substructure $Z \subset \omega^{\omega}$ such that for any $\set{1, \dots, q}$-coloring of $\set{X' \subseteq Z : X' \simeq_\delta X}$ there is $Y' \subseteq Z$ with $Y' \simeq_\delta Y$ such that $\set{X' \subseteq Y' : X' \simeq_\delta X}$ is monochrome.
    Identify $Z$ via some $L_\delta$-embedding as a substructure of $\omega^{<\omega}$, and let $\uv'_\sh(\uv'_{\nu} : \nu \in Z)\ua'$ with \hbox{$\uv'_\sh(\uv'_{\nu} : \nu \in Z) \!\in\! \VV$} realize
    $
        \set{\psi_{\theta}(\ux_{\sh}\ux_{\nu}; \uz; \ud_{\nu}) : \nu \in Z}.
    $
    Such a realization exists, since $Z$ is an antichain.
    Color $\set{X' \subseteq Z : X' \simeq_\delta X}$ by
    $$
        X' \mapsto \text{``the least $i$ with $\MM \models \varphi_i\big( \lambda_{i, \theta}(\uv'_\sh(\uv'_\nu : \nu \in X')) ; (\ud_{\nu} : \nu \in X')\big)$''}.
    $$
    Then there is $Y' \subseteq Z$ with $Y' \simeq_\delta Y$ such that $\set{X' \subseteq Y' : X' \simeq_\delta X}$ is monochrome of color $i_0$.
    As stated above, for subsets of $\omega^{<\omega}$, being $L_\delta$-isomorphic is the same as being strongly isomorphic.
    Thus, since $(\ud_\nu : \nu \in \omega^{<\omega})$ is strongly indiscernible, we can find a realization $\uv_\sh(\uv_{\nu} : \nu \in Y)\ua$ of
    $
        \set{\psi_{\theta}(\ux_{\sh}\ux_{\nu}; \uz; \ud_{\nu}) : \nu \in Y}
    $
    such that $\uv_\sh(\uv_{\nu} : \nu \in Y) \in \VV$, and
    $$
    \MM \models \varphi_{i_0}\big( \lambda_{i_0, \theta}(\uv_\sh(\uv_\nu : \nu \in X')) ; (\ud_{\nu} : \nu \in X')\big)
    $$
    holds for all $X' \subseteq Y$ with $X'$ strongly isomorphic to $X$.
    Therefore, $i_0$ is as desired for $Y$.
\end{innerproof}
\end{subclaim}

\noindent Write $X = \set{\mu_1, \dots, \mu_r}$, where the $\mu_k$'s are distinct.
Define the $L$-formula $\varphi(y; \uw_1, \dots, \uw_r)$ and the $\LKThe$-term $\lambda(\ux_\sh\ux_{1}\dots \ux_{r})$ such that $\varphi(y; \uw_{\mu_1}, \dots, \uw_{\mu_r})$ and $\lambda(\ux_\sh\ux_{\mu_1}\dots \ux_{\mu_r})$ are the formula and term obtained from Claim \ref{lemma_remove_disj_ez}.
Assume that the tuple $\ux_\sh = (x_{\sh, 1}, \dots, x_{\sh, m})$ has length $m$, and that $\ux_\va = (x_{\va, 1}, \dots, x_{\va, n})$ has length $n$.
Since we do not assume that the $\mu_k$'s are ordered in any specific way, we can permute them.
Using this and the fact that $\lambda(\ux_\sh\ux_{1}\dots \ux_{r})$ is non-trivial, we can assume that $\lambda(\ux_\sh\ux_{1}\dots \ux_{r})$ can be written in one of the following forms:
\begin{enumerate}[(i)]
    \item $-\eta[\theta](x_{r, n}) + \sum_{l=1}^{n-1} \rho_{r, l}[\theta](x_{r, l}) + \sum_{k = 1}^{r-1}\sum_{l=1}^{n} \rho_{k, l}[\theta](x_{k, l}) + \sum_{l=1}^{m} \rho_{\sh,l}[\theta](x_{\sh, l})$, or
    \item $-\eta[\theta](x_{\sh, m}) + \sum_{l=1}^{m-1} \rho_{\sh,l}[\theta](x_{\sh, l})$,
\end{enumerate}
where $\eta \in K[X] \setminus \set{0}$ and all $\rho$'s are in $K[X]$.
We show that both cases lead to a contradiction.

\begin{subclaim} \label{claim_remove_eq_lower_va}
    In case (i), there is an $L$-formula $\psi(\uxvec{}'_\sh\uxvec{}'_\va; \uz; \uw')$ with $|\ux'_\va| < |\ux_\va|$ such that
    $$
    \ux'_\sh \in \VV \wedge \exists \ux'_\va \in \VV : \psi_\theta(\ux'_\sh\ux'_\va; \uz; \uw')
    $$
    witnesses \ATP{} in $\ux'_\sh\uz$.
\begin{innerproof}
    Set $\uw'' := \uw_1\dots\uw_r$, and consider the tree $(\ud''_\nu : \nu \in 2^{<\omega})$ where $\ud''_\nu := \ud_{\mu_1}\dots\ud_{\mu_{r-1}}\ud_{\mu_r{}^\frown \nu}$ (see Figure \ref{figure_reindexed_tree_at_mu_r}).
    Set $\ux'_\sh := (x_{\sh, 1}, \dots, x_{\sh, m+1})$ and $\ux'_\va := (x_{\va, 1}, \dots, x_{\va, n-1})$.
    Now define a tuple of $\LRC$-terms $\ut(y)$ and $\psi'(\uxvec{}'_\sh\uxvec{}'_\va; \uz; \uw''\tiluw)$ such that $\psi'_\theta(\ux'_\sh\ux'_\va; \uz; \uw''\ut(y))$ is $\psi_\theta(\ux_\sh\ux_\va; \uz; \uw_{r})$ but with:
\begin{enumerate}[(a)]
    \item every occurrence of $\theta^i(x_{\sh, l})$ replaced by $\theta^i(\eta[\theta](x_{\sh, l}))$;
    \item every occurrence of $\theta^i(x_{\va, l})$ replaced by $\theta^i(\eta[\theta](x_{\va, l}))$ for $l < n$;
    \item every occurrence of $\theta^i(x_{\va, n})$ replaced by $\theta^i\big(\sum_{l=1}^{n-1}\rho_{r, l}[\theta](x_{\va, l}) + x_{\sh, m+1}\big) - \theta^i(\eta[\theta]^{-1}(y))$.
\end{enumerate}
For details on this see, e.g., the proof of Lemma 4.6 in \cite{Chi25}.
Let $Z \subset 2^{<\omega}$ be an arbitrary antichain.
By applying (\hyperref[lemma_remove_disj_ez]{$\star$}) from Claim \ref{lemma_remove_disj_ez} to the antichain (see Figure \ref{figure_reindexed_tree_at_mu_r})
$$
Y = (X \setminus \set{\mu_r}) \cup \mu_r{\!}^\frown Z,
$$
we obtain a realization $\uv_\sh(\uv_\nu : \nu \in Y)\ua$ of $\set{\psi_\theta(\ux_\sh\ux_\nu; \uz; \ud_\nu) : \nu \in Y}$ such that $\uv_\sh(\uv_{\nu} : \nu \in Y)$ lies in $\VV$, and for every $\nu \in Z$, there is some $u_\nu \in \varphi(\MM; \ud''_\nu)$ with $\lambda(\uv_\sh \uv_{\mu_1}\dots \uv_{\mu_{r-1}} \uv_{\mu_r{\!}^\frown \nu}) = u_\nu$ (note that the antichain $\set{\mu_1, \dots, \mu_{r-1}, \mu_r{\!}^\frown \nu}$ is strongly isomorphic to $X$ for any $\nu \in 2^{<\omega}$; see Figure \ref{figure_reindexed_tree_at_mu_r}).
Define the $(m+1)$-tuple
$$
\uv'_\sh := \eta[\theta]^{-1}\big(v_{\sh, 1}, \dots, v_{\sh, m}, \lambda(\uv_\sh\uv_{\mu_1}\dots\uv_{\mu_{r-1}}\uzero)\big)
$$
and, for every $\nu \in Z$, define the $(n-1)$-tuple $\uv'_\nu := \eta[\theta]^{-1}(v_{\mu_r{\!}^\frown\nu, 1}, \dots, v_{\mu_r{\!}^\frown\nu, n-1})$.
Now the $L_\theta(M)$-sentence $\psi'_\theta(\uv'_\sh \uv'_\nu;\ua; \ud''_\nu\ut(u_\nu))$ is, by definition, $\psi_\theta(\ux_\sh\ux_\va; \ua; \ud_{\mu_r{\!}^\frown \nu})$ with:
\begin{enumerate}[(a)]
    \item every occurrence of $\theta^i(x_{\sh, l})$ replaced by $\theta^i(\eta[\theta](\eta[\theta]^{-1}(v_{\sh, l}))) = \theta^i(v_{\sh, l})$;
    \item every occurrence of $\theta^i(x_{\va, l})$ replaced by $\theta^i( \eta[\theta](\eta[\theta]^{-1}(v_{\mu_r{\!}^\frown\nu, l}))) = \theta^i(v_{\mu_r{\!}^\frown\nu, l})$ for $l < n$;
    \item every occurrence of $\theta^i(x_{\va, n})$ replaced by
    \begin{align*}
        &\theta^i\Big(\sum\nolimits_{l=1}^{n-1}\rho_{r, l}[\theta](\eta[\theta]^{-1}(v_{\mu_r{\!}^\frown\nu, l})) + \eta[\theta]^{-1}(\lambda(\uv_\sh\uv_{\mu_1}\dots\uv_{\mu_{r-1}}\uzero))\Big) - \theta^i(\eta[\theta]^{-1}(u_{\nu})) \\
        =\ &\theta^i\Big(\eta[\theta]^{-1}\Big(\sum\nolimits_{l=1}^{n-1}\rho_{r, l}[\theta](v_{\mu_r{\!}^\frown\nu, l}) + \lambda(\uv_\sh\uv_{\mu_1}\dots\uv_{\mu_{r-1}}\uzero) - u_{\nu}\Big)\Big) \\
        =\ &\theta^i\Big(\eta[\theta]^{-1}\Big(\sum\nolimits_{l=1}^{n-1}\rho_{r, l}[\theta](v_{\mu_r{\!}^\frown\nu, l}) + \lambda(\uv_\sh\uv_{\mu_1}\dots\uv_{\mu_{r-1}}\uzero) - \lambda(\uv_\sh \uv_{\mu_1}\dots \uv_{\mu_{r-1}} \uv_{\mu_r{\!}^\frown \nu})\Big)\Big) \\
        =\ &\theta^i\Big(\eta[\theta]^{-1}\Big(\sum\nolimits_{l=1}^{n-1}\rho_{r, l}[\theta](v_{\mu_r{\!}^\frown\nu, l}) - \lambda(\uzero\dots\uzero\ \!\uv_{\mu_r{\!}^\frown \nu})\Big)\Big) \\
        =\ &\theta^i\Big(\eta[\theta]^{-1}\Big(\sum\nolimits_{l=1}^{n-1}\rho_{r, l}[\theta](v_{\mu_r{\!}^\frown\nu, l}) - \Big( -\eta[\theta](v_{\mu_r{\!}^\frown\nu, n}) + \sum\nolimits_{l=1}^{n-1}\rho_{r, l}[\theta](v_{\mu_r{\!}^\frown\nu, l}) \Big)\Big)\Big) \\
        =\ & \theta^i(\eta[\theta]^{-1}(\eta[\theta](v_{\mu_r{\!}^\frown\nu, n}))) \\
        =\ & \theta^i(v_{\mu_r{\!}^\frown\nu, n}).
    \end{align*}
\end{enumerate}
Therefore, $\psi'_\theta(\uv'_\sh \uv'_\nu;\ua; \ud''_\nu\ut(u_\nu))$ is just $\psi_\theta(\uv_\sh\uv_{\mu_r{\!}^\frown\nu};\ua; \ud_{\mu_r{\!}^\frown \nu})$.
Since this $L_\theta(M)$-sentence holds for every $\nu \in Z$, we can easily see that the formula
$$
\exists y\tiluw : \big( \varphi(y; \uw'') \wedge \tiluw = \ut(y)\big) \wedge \big(\ux'_\sh \in  \VV \wedge\exists \ux'_\va \in \VV : \psi'_\theta(\ux'_\sh\ux'_\va; \uz; \uw'' \tiluw)\big)
$$
witnesses \ATP{} in $\ux'_\sh\uz$ together with the tree $(\ud''_\nu : \nu \in 2^{<\omega})$.
Indeed, if a $2$-path were consistent, then the same $2$-path for $\ux_\sh \in \VV \wedge \exists \ux_\va \in \VV : \psi_\theta(\ux_\sh\ux_\va; \uz; \uw)$ with $(\ud_\nu : \nu \in 2^{<\omega})$ would also be consistent by the way we defined $\psi'(\uxvec{}'_\sh\uxvec{}'_\va; \uz; \uw''\tiluw)$.
Set $\uw' := \uw''\tiluw$ and use Lemma \ref{lemma_remove_alg_set} to show that the formula $\ux'_\sh \in \VV \wedge\exists \ux'_\va \in \VV : \psi'_\theta(\ux'_\sh\ux'_\va; \uz; \uw')$ must witness \ATP{}.
By construction, we have $|\ux'_\va| < |\ux_\va|$.
\end{innerproof}
\end{subclaim}

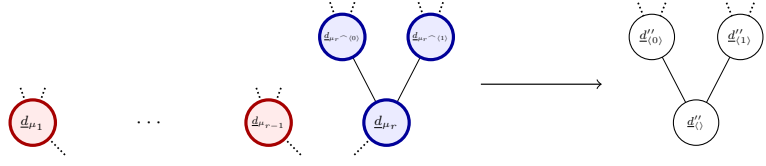
\begin{figure}[tbp]
    \centering
    \begin{tikzpicture}[
        x=0.78cm,
        y=0.78cm,
        node/.style={circle, draw, fill=white, inner sep=1pt, minimum size=6mm, font=\scriptsize},
        long label node/.style={circle, draw, fill=white, inner sep=1pt, minimum size=6mm, font=\scriptsize},
        tree edge/.style={line width=0.35pt},
        continuation edge/.style={densely dotted, line width=0.65pt},
        fixed node/.style={node, draw=red!65!black, fill=red!8, very thick},
        moving node/.style={node, draw=blue!60!black, fill=blue!8, very thick},
        map arrow/.style={->, line width=0.45pt}
    ]
        \coordinate (fixedone) at (-3,1.6);
        \coordinate (fixedlast) at (1,1.6);
        \coordinate (movingroot) at (3,1.6);
        \coordinate (movingzero) at (2.25,3.05);
        \coordinate (movingone) at (3.75,3.05);

        \draw[continuation edge] (fixedone) -- ++(0.55,-0.55);
        \draw[continuation edge] (fixedlast) -- ++(0.55,-0.55);
        \draw[continuation edge] (fixedone) -- ++(-0.22,0.62);
        \draw[continuation edge] (fixedone) -- ++(0.22,0.62);
        \draw[continuation edge] (fixedlast) -- ++(-0.22,0.62);
        \draw[continuation edge] (fixedlast) -- ++(0.22,0.62);
        \draw[continuation edge] (movingroot) -- ++(-0.55,-0.55);
        \draw[tree edge] (movingroot) -- (movingzero);
        \draw[tree edge] (movingroot) -- (movingone);
        \draw[continuation edge] (movingzero) -- ++(-0.22,0.62);
        \draw[continuation edge] (movingzero) -- ++(0.22,0.62);
        \draw[continuation edge] (movingone) -- ++(-0.22,0.62);
        \draw[continuation edge] (movingone) -- ++(0.22,0.62);

        \node[fixed node] at (fixedone) {\scalebox{0.72}{$\ud_{\mu_1}$}};
        \node[font=\scriptsize] at (-1,1.6) {$\cdots$};
        \node[fixed node] at (fixedlast) {\scalebox{0.58}{$\ud_{\mu_{r-1}}$}};
        \node[moving node] at (movingroot) {\scalebox{0.72}{$\ud_{\mu_r}$}};
        \node[long label node, draw=blue!60!black, fill=blue!8, very thick] at (movingzero) {\scalebox{0.46}{$\ud_{\mu_r{}^\frown\bigseq{0}}$}};
        \node[long label node, draw=blue!60!black, fill=blue!8, very thick] at (movingone) {\scalebox{0.46}{$\ud_{\mu_r{}^\frown\bigseq{1}}$}};

        \draw[map arrow] (4.6,2.25) -- (6.65,2.25);

        \coordinate (newroot) at (8.25,1.6);
        \coordinate (newzero) at (7.5,3.05);
        \coordinate (newone) at (9,3.05);

        \draw[tree edge] (newroot) -- (newzero);
        \draw[tree edge] (newroot) -- (newone);
        \draw[continuation edge] (newzero) -- ++(-0.25,0.62);
        \draw[continuation edge] (newzero) -- ++(0.25,0.62);
        \draw[continuation edge] (newone) -- ++(-0.25,0.62);
        \draw[continuation edge] (newone) -- ++(0.25,0.62);

        \node[long label node] at (newroot) {\scalebox{0.62}{$\ud''_{\bigemptyseq}$}};
        \node[long label node] at (newzero) {\scalebox{0.66}{$\ud''_{\bigseq{0}}$}};
        \node[long label node] at (newone) {\scalebox{0.66}{$\ud''_{\bigseq{1}}$}};
    \end{tikzpicture}
    \caption{Visualization of the tree $(\ud''_\nu : \nu \in 2^{<\omega})$. The red parameters $\ud_{\mu_1},\dots,\ud_{\mu_{r-1}}$ form the fixed initial tuple of every $\ud''_\nu$, while the blue subtree with root $\ud_{\mu_r}$ supplies the varying final subtuple. At the level of indices, for every antichain $Z \subset 2^{<\omega}$, the set $\set{\mu_1,\dots,\mu_{r-1}}\cup \mu_r{\!}^\frown Z$ is an antichain; moreover, for each blue element $\mu_r{}^\frown\nu$, the set $\set{\mu_1,\dots,\mu_{r-1},\mu_r{}^\frown\nu}$ is strongly isomorphic to $X$.}
    \label{figure_reindexed_tree_at_mu_r}
\end{figure}
\begin{subclaim} \label{subclaim_only_shared_reduce}
    If we are in case (ii), i.e., $\lambda(\ux_\sh\ux_{1}\dots \ux_{r}) = -\eta[\theta](x_{\sh, m}) + \sum_{l=1}^{m-1} \rho_{\sh,l}[\theta](x_{\sh, l})$, then there is an $L$-formula $\psi(\uxvec{}'_\sh\uxvec{}'_\va; \uz; \uw')$ with $|\ux'_\va| = |\ux_\va|$ and $|\ux'_\sh| < |\ux_\sh|$ such that
    $$
    \ux'_\sh \in \VV \wedge \exists \ux'_\va \in \VV : \psi_\theta(\ux'_\sh\ux'_\va; \uz; \uw')
    $$
    witnesses \ATP{} in $\ux'_\sh\uz$.
\begin{innerproof}
    As in case (i), set $\uw'' := \uw_1\dots\uw_r$, and consider the same tree $(\ud''_\nu : \nu \in 2^{<\omega})$ where $\ud''_\nu := \ud_{\mu_1}\dots\ud_{\mu_{r-1}}\ud_{\mu_r{}^\frown \nu}$.
    Set $\ux'_\sh := (x_{\sh, 1}, \dots, x_{\sh, m-1})$ and $\ux'_\va := \ux_\va$.
    Again, define a tuple of $\LRC$-terms $\ut(y)$ and $\psi'(\uxvec{}'_\sh\uxvec{}'_\va; \uz; \uw''\tiluw)$ such that $\psi'_\theta(\ux'_\sh\ux'_\va; \uz; \uw''\ut(y))$ is $\psi_\theta(\ux_\sh\ux_\va; \uz; \uw_{r})$ but with:
\begin{enumerate}[(a)]
    \item every occurrence of $\theta^i(x_{\sh, l})$ replaced by $\theta^i(\eta[\theta](x_{\sh, l}))$ for $l < m$;
    \item every occurrence of $\theta^i(x_{\sh, m})$ replaced by $\theta^i\big(\sum_{l=1}^{m-1}\rho_{\sh, l}[\theta](x_{\sh, l})\big) - \theta^i(\eta[\theta]^{-1}(y))$.
\end{enumerate}
By arguments similar to those in case (i), with $\uw' := \uw''\tiluw$, we easily verify that the formula $\ux'_\sh \in \VV \wedge\exists \ux'_\va \in \VV : \psi'_\theta(\ux'_\sh\ux'_\va; \uz; \uw')$ must witness \ATP{}.
\end{innerproof}
\end{subclaim}

\noindent In both cases, we obtain an $L$-formula $\psi(\uxvec{}'_\sh\uxvec{}'_\va; \uz; \uw')$ with $(|\ux'_\va|, |\ux'_\sh|) <_\Lex (|\ux_\va|, |\ux_\sh|)$ such that
$$
    \ux'_\sh \in \VV \wedge \exists \ux'_\va \in \VV : \psi_\theta(\ux'_\sh\ux'_\va; \uz; \uw')
$$
witnesses \ATP{} in $\ux'_\sh\uz$.
However, no such formula exists by the minimality of $\psi(\uxvec{}_\sh\uxvec{}_\va; \uz; \uw)$.
Hence, our assumption that $\set{\psi(\uxvec{}_\sh\uxvec{}_\nu; \uz; \ud_\nu) : \nu \in X}$ implies some finite disjunction of non-trivial linear dependencies in $\uxvec{}_\sh(\uxvec{}_\nu : \nu \in X)$ over $\VV$ must be wrong.
Finally, note that the partial type $$\set{\psi(\uxvec{}_\sh\uxvec{}_\emptyseq; \uz; \ud_\emptyseq), \psi(\uxvec{}_\sh\uxvec{}_{\seq{0}}; \uz; \ud_{\seq{0}})}$$ must imply some finite disjunction of non-trivial linear dependencies.
Otherwise, our characterization of existentially closed models of $T_\theta{\!\!}^{C_0}$, Theorem \ref{theorem_big_characterization}, would immediately yield a realization of
$$
\set{\ux_\sh \in \VV \wedge \exists \ux_\va \in \VV : \psi_\theta(\ux_\sh\ux_\va; \uz; \ud_\nu) : \nu \in \set{\bigemptyseq, \bigseq{0}}},
$$
contradicting the fact that $\ux_\sh \in \VV \wedge \exists \ux_\va \in \VV : \psi_\theta(\ux_\sh\ux_\va; \uz; \uw)$ witnesses \ATP{} with $(\ud_\nu : \nu \in 2^{<\omega})$.
We conclude that $\psi(\uxvec{}_\sh\uxvec{}_\va; \uz; \uw)$ and $(\ud_\nu : \nu \in 2^{<\omega})$ are as described in the statement of Lemma \ref{lemma_atp_step_11_ez}.
\end{proof}

\subsection{Step (II): Proving ATP}
\label{sec_step_ii_always_ez}
We now show that the existence of an $L$-formula and a tree, as described in the conclusion of Lemma \ref{lemma_atp_step_11_ez}, implies that $T$ has \ATP{}.
Together, Lemma \ref{lemma_atp_step_11_ez} and the lemma below immediately yield a proof of Theorem \ref{theorem_preserve_natp_ez}.

\begin{lemma} \label{lemma_has_li_atp_impl_atp}
Let $\psi(\ux_\sh\ux_\va; \uz; \uw)$ be an $L$-formula, and let $(\ud_\nu : \nu \in 2^{<\omega})$ be a strongly $L$-indiscernible tree in $\mm \models T$ such that the following holds:
\begin{enumerate}[(i)]
    \item The type $\set{\psi(\ux_\sh\ux_\nu; \uz; \ud_\nu) : \nu \in X}$ implies no finite disjunction of non-trivial linear dependencies in $\ux_\sh(\ux_\nu : \nu \in X)$ over $\VV$ whenever $X \subset 2^{<\omega}$ is an antichain.
    \item The type $\set{\psi(\ux_\sh\ux_\emptyseq; \uz; \ud_\emptyseq), \psi(\ux_\sh\ux_{\seq{0}}; \uz; \ud_{\seq{0}})}$ implies some finite disjunction of non-trivial linear dependencies in $\ux_\sh\ux_\emptyseq\ux_{\seq{0}}$ over $\VV$.
\end{enumerate}
Then $T$ has \ATP{}.
\end{lemma}
\begin{proof}
    We may assume without loss that $\psi(\ux_\sh\ux_\va; \uz; \uw)$ implies $\ux_\sh\ux_\va \in \VV$. We also work in a sufficiently large monster model $\MM \models T$ containing $\mm$.
    By Fact \ref{lemma_hfour_fml}, the formula $\psi(\ux_\sh\ux_\emptyseq; \uz; \ud_\emptyseq) \wedge \psi(\ux_\sh\ux_{\seq{0}}; \uz; \ud_{\seq{0}})$ must imply a disjunction of the form
    \begin{align}
        \bigvee\nolimits_{i=1}^r \varphi_{1, i}(\lambda_{1, i}(\ux_\sh\ux_\emptyseq\ux_{\seq{0}}); \ud_\emptyseq\ud_{\seq{0}})
        \vee\bigvee\nolimits_{i=1}^s \varphi_{2, i}(\lambda_{2, i}(\ux_\sh); \ud_\emptyseq\ud_{\seq{0}})
        \vee\bigvee\nolimits_{i=1}^t \varphi_{3, i}(\lambda_{3, i}(\ux_\sh\ux_\emptyseq); \ud_\emptyseq\ud_{\seq{0}}). \label{tag_original_disju}
    \end{align}
Here each $\lambda_{k, i}$ is an $\LK$-term, and each $\varphi_{k, i}(y; \uw_1\uw_2)$ is an $L$-formula such that $\varphi_{k, i}(y; \ud_1\ud_2)$ defines a finite subset of $\VV$ for every $\ud_1\ud_2 \in M$.
We can also assume that every $\lambda_{1, i}(\ux_\sh\ux_\emptyseq\ux_{\seq{0}})$ is non-trivial in $\ux_{\seq{0}}$, every $\lambda_{2, i}(\ux_\sh)$ is non-trivial in $\ux_\sh$, and every $\lambda_{3, i}(\ux_\sh\ux_\emptyseq)$ is non-trivial in $\ux_\emptyseq$.
If not, move the corresponding disjunct into the correct disjunction or remove it.
We now simplify the disjunction (\ref{tag_original_disju}) in a series of claims.

\begin{subclaim} \label{lemma_intersection_inconst}
 We may assume without loss that
    $
    \varphi_{3, i}(x; \ud_\emptyseq\ud_\mu) \wedge \varphi_{3, i}(x; \ud_\emptyseq\ud_\eta)
    $
    is inconsistent for any $i$ and any $\mu, \eta \in 2^{<\omega}$ with $\mu \perp \eta$.
\begin{innerproof}
    By Lemma \ref{lemma_make_intersection_definable}, applied to $\varphi(y; \uw_1\uw_2) := \bigvee_{i=1}^t \varphi_{3, i}(y; \uw_1\uw_2)$, we can assume (that each $\ud_\nu$ is a finite tuple and) that, for any $\nu, \mu, \eta \in 2^{<\omega}$ with $\mu \perp \eta$,
    \begin{align}
        \varphi(\MM; \ud_\nu\ud_{\nu^\frown\mu}) \cap \varphi(\MM; \ud_\nu\ud_{\nu^\frown\eta}) \subseteq \ud_\nu \label{tag_intersection_is_subset}
    \end{align}
    holds.
    We can now replace $\psi(\ux_\sh\ux_\va; \uz; \uw)$ with
    $
    \psi(\ux_\sh\ux_\va; \uz; \uw) \wedge \bigwedge\nolimits_{i=1}^t \lambda_{3, i}(\ux_\sh\ux_\va) \not\in \uw
    $.
    By (ii) of Lemma \ref{lemma_make_intersection_definable}, the partial type $\set{\psi(\ux_\sh\ux_\nu; \uz; \ud_\nu) : \nu \in X}$ still implies no finite disjunction of non-trivial linear dependencies in $\ux_\sh(\ux_\nu : \nu \in X)$ over $\VV$ for any antichain $X \subset 2^{<\omega}$.
    Similarly, the formula $\psi(\ux_\sh\ux_\emptyseq; \uz; \ud_\emptyseq) \wedge \psi(\ux_\sh\ux_{\seq{0}}; \uz; \ud_{\seq{0}})$ now implies the disjunction
    $$
    \bigvee\nolimits_{i=1}^r \varphi_{1, i}(\lambda_{1, i}(\ux_\sh\ux_\emptyseq\ux_{\seq{0}}); \ud_\emptyseq\ud_{\seq{0}})
    \vee\bigvee\nolimits_{i=1}^s \varphi_{2, i}(\lambda_{2, i}(\ux_\sh); \ud_\emptyseq\ud_{\seq{0}})
    \vee\bigvee\nolimits_{i=1}^t \varphi'_{3, i}(\lambda_{3, i}(\ux_\sh\ux_\emptyseq); \ud_\emptyseq\ud_{\seq{0}})
    $$
    where $\varphi'_{3, i}(y; \uw_1\uw_2) := \varphi_{3, i}(y; \uw_1\uw_2) \wedge y \not\in \uw_1$.
    Now a realization of $\varphi'_{3, i}(y; \ud_\emptyseq\ud_\mu) \wedge \varphi'_{3, i}(y; \ud_\emptyseq\ud_\eta)$ lies in
    $
    \big(\varphi(\MM; \ud_\emptyseq\ud_{\mu}) \cap \varphi(\MM; \ud_\emptyseq\ud_{\eta})\big) \setminus \ud_\emptyseq,
    $
    which is empty by (\ref{tag_intersection_is_subset}).
\end{innerproof}
\end{subclaim}

\begin{notation}
    Given an element $\nu \colon m \to 2$ of $2^{<\omega}$ and $n > 0$, we write $\nu \uparrow n$ for the element of $2^{<\omega}$ with domain $m \cdot n$ defined by $(\nu \uparrow n)(i) := \nu(\lfloor i/n \rfloor)$.
\end{notation}

\begin{subclaim} \label{lemma_remove_x_lower}
    We may assume that $\set{\psi(\ux_\sh\ux_\emptyseq; \uz; \ud_\emptyseq), \psi(\ux_\sh\ux_{\seq{0}}; \uz; \ud_{\seq{0}})}$ implies a disjunction of the form
    $$
    \bigvee\nolimits_{i=1}^q \varphi_i(\lambda_i(\ux_\sh\ux_{\seq{0}}); \ud_\emptyseq\ud_{\seq{0}})
    $$
    where each $\lambda_i(\ux_\sh\ux_\va)$ is a non-trivial $\LK$-term and each $L$-formula $\varphi_i(y; \uw_0\uw_1)$ defines a finite subset of $\VV$ for any parameters plugged into $\uw_0\uw_1$.
\begin{innerproof}
    Choose $n$ such that $2^n > r + t$ for the $r$ and $t$ from (\ref{tag_original_disju}).
    Let $\uv_\sh\uv_\emptyseq(\uv_\eta : \eta \in 2^n)\ua$ be a realization of the formula
    $$
    \psi(\ux_\sh\ux_\emptyseq; \uz; \ud_\emptyseq) \wedge \bigwedge\nolimits_{\eta\in 2^n} \psi(\ux_\sh\ux_\eta; \uz; \ud_\eta).
    $$
    Notice that $\uv_\sh\uv_\emptyseq(\uv_\eta : \eta \in 2^n) \in \VV$, since $\psi(\ux_\sh\ux_\va; \uz; \uw)$ implies $\ux_\sh\ux_\va \in \VV$.
    By the strong indiscernibility of $(\ud_\nu : \nu \in 2^{< \omega})$, we obtain
    \begin{align}
        \bigvee\nolimits_{i=1}^r \varphi_{1, i}(\lambda_{1, i}(\uv_\sh\uv_\emptyseq\uv_\eta); \ud_\emptyseq\ud_\eta)
        \vee\bigvee\nolimits_{i=1}^s \varphi_{2, i}(\lambda_{2, i}(\uv_\sh); \ud_\emptyseq\ud_\eta)
        \vee\bigvee\nolimits_{i=1}^t \varphi_{3, i}(\lambda_{3, i}(\uv_\sh\uv_\emptyseq); \ud_\emptyseq\ud_\eta) \label{tag_not_original_disju}
    \end{align}
    for every $\eta \in 2^n$.
    Color every $\eta \in 2^n$ using the colors $(1, 1), \dots, (1, r)$, $(2, 1), \dots, (2, s)$, $(3, 1), \dots, (3, t)$, according to the first disjunct in (\ref{tag_not_original_disju}) that holds.
    By Claim \ref{lemma_intersection_inconst}, there can be at most $t$ elements with a color of the form $(3, i)$.
    Since $2^n > r + t$, there must either be two distinct elements $\eta_1, \eta_2 \in 2^n$ with the same color of the form $(1, i)$ or some element $\eta \in 2^n$ with a color of the form $(2, i)$.
    Since $\lambda_{1, i}$ is linear, in the first case we obtain
    $$
    \lambda_{1, i}(\uzero\;\uzero\;(\uv_{\eta_1}-\uv_{\eta_2})) \in \varphi_{1, i}(\MM; \ud_\emptyseq\ud_{\eta_1}) - \varphi_{1, i}(\MM; \ud_\emptyseq\ud_{\eta_2}),
    $$
    and in the second case we obtain $\lambda_{2, i}(\uv_\sh) \in \varphi_{2, i}(\MM; \ud_\emptyseq\ud_\eta)$.
    Since the $\LK$-terms $\lambda_{k, i}$ are non-trivial in the relevant entries and each $\varphi_{k, i}(\MM; \ud_\emptyseq\ud_\eta)$ defines a finite subset of $\VV$, we see that
    $$
    \psi(\ux_\sh\ux_\emptyseq; \uz; \ud_\emptyseq) \wedge \bigwedge\nolimits_{\eta\in 2^n} \psi(\ux_\sh\ux_\eta; \uz; \ud_\eta)
    $$
    indeed implies some finite disjunction $\bigvee\nolimits_{i=1}^q \varphi_{i}(\lambda_{i}(\ux_\sh(\ux_\eta : \eta \in 2^n)); \ud_\emptyseq(\ud_\eta : \eta \in 2^n))$ of non-trivial linear dependencies in $\ux_\sh(\ux_\eta : \eta \in 2^n)$ over $\VV$.
    The exact same argument also yields that
    $$
    \psi(\ux_\sh\ux_\nu; \uz; \ud_\nu) \wedge \bigwedge\nolimits_{\eta \in 2^n} \psi(\ux_\sh\ux_{\mu^\frown\eta}; \uz; \ud_{\mu^\frown\eta})
    $$
    implies $\bigvee\nolimits_{i=1}^q \varphi_{i}(\lambda_{i}(\ux_\sh(\ux_{\mu^\frown\eta} : \eta \in 2^n)); \ud_\nu(\ud_{\mu^\frown\eta} : \eta \in 2^n))$ for any $\nu, \mu \in 2^{<\omega}$ with $\nu \trianglelefteq \mu$. In the case where no realization $\uv_\sh\uv_\emptyseq(\uv_\eta : \eta \in 2^n)\ua$ as in the beginning of the proof exists, we obtain the same conclusion with $q = 0$.

    Now set $\ux'_\sh := \ux_\sh$, set $\ux'_\va := (\ux_{\va, \eta} : \eta \in 2^n)$ with $|\ux_{\va, \eta}| = |\ux_\va|$ for all $\eta \in 2^n$, define the tuple $\uw' := (\uw_\eta : \eta \in 2^n)$, and define
    $$
    \psi'(\ux'_\sh\ux'_\va; \uz; \uw') := \bigwedge\nolimits_{\eta \in 2^n} \psi(\ux_\sh\ux_{\va, \eta}; \uz; \uw_\eta).
    $$
    Also set $\ud'_\nu := (\ud_{(\nu \uparrow n)^\frown\eta} : \eta \in 2^n)$; see the notation above this claim and Figure \ref{figure_stretched_packed_tree_n_two} for the case $n=2$.
    One can easily check that the partial type $\set{\psi'(\ux'_\sh\ux'_\nu; \uz; \ud'_\nu) : \nu \in X}$ implies no finite disjunction of non-trivial linear dependencies in $\ux'_\sh(\ux'_\nu : \nu \in X)$ over $\VV$ for any antichain $X \subset 2^{<\omega}$.
    Indeed, by the definition of $\psi'(\ux'_\sh\ux'_\va; \uz; \uw')$ and of each $\ud'_\nu$, this type is equivalent to a type of the form $\set{\psi(\ux_\sh\ux_\nu; \uz; \ud_\nu) : \nu \in Y}$ for some antichain $Y \subset 2^{<\omega}$ (see Figure \ref{figure_stretched_packed_tree_n_two}).
    One can also check that for any $\nu \triangleleft \mu$, the formula
    $
    \psi'(\ux'_\sh\ux'_\nu; \uz; \ud'_\nu) \wedge \psi'(\ux'_\sh\ux'_\mu; \uz; \ud'_\mu)
    $
    implies the following finite disjunction of non-trivial linear dependencies in $\ux'_\sh\ux'_\mu$ over $\VV$:
    $$
    \bigvee\nolimits_{i=1}^q \varphi_{i}(\lambda_{i}(\ux'_\sh\ux'_\mu); \ud_{(\nu \uparrow n)^\frown \seq{0}^n}\ud'_{\mu}) \vee \bigvee\nolimits_{i=1}^q \varphi_{i}(\lambda_{i}(\ux'_\sh\ux'_\mu); \ud_{(\nu \uparrow n)^\frown \seq{1}^n}\ud'_{\mu}).
    $$
    Notice that $\ud_{(\nu \uparrow n)^\frown \seq{0}^n}$ and $\ud_{(\nu \uparrow n)^\frown \seq{1}^n}$ are subtuples of $\ud'_{\nu}$.
    By Lemma \ref{lemma_weak_modeling}, we can assume that the tree $(\ud'_{\nu} : \nu \in 2^{<\omega})$ is strongly indiscernible.
    Now replace $\psi(\ux_\sh\ux_\va; \uz; \uw)$ with $\psi'(\ux'_\sh\ux'_\va; \uz; \uw')$ and $(\ud_\nu : \nu \in 2^{<\omega})$ with $(\ud'_\nu : \nu \in 2^{<\omega})$ to conclude.
\end{innerproof}
\end{subclaim}
\begin{figure}[tbp]
    \centering
    \begin{tikzpicture}[
        x=0.78cm,
        y=0.78cm,
        node/.style={circle, draw, fill=white, inner sep=1pt, minimum size=6mm, font=\scriptsize},
        small node/.style={circle, draw, fill=white, inner sep=0.5pt, minimum size=4mm},
        tree edge/.style={line width=0.35pt},
        continuation edge/.style={densely dotted, line width=0.65pt},
        package outline/.style={draw=black!65, densely dashed, line width=0.45pt, line cap=round, line join=round},
        blue node/.style={node, draw=blue!60!black, fill=blue!8, very thick},
        red node/.style={node, draw=red!65!black, fill=red!8, very thick},
        map arrow/.style={->, line width=0.45pt}
    ]
        \def\drawblock#1#2#3#4{\path let \p1=(#1), \p2=(#2) in
                \pgfextra{
                    \pgfmathsetmacro{\ax}{\x1/1pt}
                    \pgfmathsetmacro{\ay}{\y1/1pt}
                    \pgfmathsetmacro{\bx}{\x2/1pt}
                    \pgfmathsetmacro{\by}{\y2/1pt}
                    \pgfmathsetmacro{\rad}{#3/1pt}
                    \pgfmathsetmacro{\len}{sqrt((\bx-\ax)^2+(\by-\ay)^2)}
                    \pgfmathsetmacro{\nx}{-(\by-\ay)/\len}
                    \pgfmathsetmacro{\ny}{(\bx-\ax)/\len}
                    \pgfmathsetmacro{\angup}{atan2(\ny,\nx)}
                    \pgfmathsetmacro{\angdown}{\angup-180}
                    \pgfmathsetmacro{\angfinish}{\angup-360}
                    \xdef\blockaup{(\ax+\rad*\nx pt,\ay+\rad*\ny pt)}
                    \xdef\blockbup{(\bx+\rad*\nx pt,\by+\rad*\ny pt)}
                    \xdef\blockbdown{(\bx-\rad*\nx pt,\by-\rad*\ny pt)}
                    \xdef\blockadown{(\ax-\rad*\nx pt,\ay-\rad*\ny pt)}
                    \xdef\blockangup{\angup}
                    \xdef\blockangdown{\angdown}
                    \xdef\blockangfinish{\angfinish}
                };
            \draw[package outline,#4]
                \blockaup
                -- \blockbup
                arc[start angle=\blockangup, end angle=\blockangdown, radius=#3]
                -- \blockadown
                arc[start angle=\blockangdown, end angle=\blockangfinish, radius=#3]
                -- cycle;
        }

        \coordinate (o) at (-1.8,-1.6);
        \coordinate (o0) at (-3,-0.8);
        \coordinate (o1) at (-0.6,-0.8);
        \coordinate (o00) at (-3.6,0);
        \coordinate (o01) at (-2.4,0);
        \coordinate (o10) at (-1.2,0);
        \coordinate (o11) at (0,0);

        \coordinate (o000) at (-4.8,1.25);
        \coordinate (o001) at (-3,1.25);
        \coordinate (o0000) at (-5.25,2.5);
        \coordinate (o0001) at (-4.35,2.5);
        \coordinate (o0010) at (-3.45,2.5);
        \coordinate (o0011) at (-2.55,2.5);

        \coordinate (o110) at (-0.6,1.25);
        \coordinate (o111) at (1.2,1.25);
        \coordinate (o1100) at (-1.05,2.5);
        \coordinate (o1101) at (-0.15,2.5);
        \coordinate (o1110) at (0.75,2.5);
        \coordinate (o1111) at (1.65,2.5);

        \draw[tree edge] (o) -- (o0) -- (o00);
        \draw[tree edge] (o0) -- (o01);
        \draw[tree edge] (o) -- (o1) -- (o10);
        \draw[tree edge] (o1) -- (o11);
        \draw[tree edge] (o00) -- (o000) -- (o0000);
        \draw[tree edge] (o000) -- (o0001);
        \draw[tree edge] (o00) -- (o001) -- (o0010);
        \draw[tree edge] (o001) -- (o0011);
        \draw[tree edge] (o11) -- (o110) -- (o1100);
        \draw[tree edge] (o110) -- (o1101);
        \draw[tree edge] (o11) -- (o111) -- (o1110);
        \draw[tree edge] (o111) -- (o1111);
        \draw[continuation edge] (o01) -- ++(-0.28,0.76);
        \draw[continuation edge] (o01) -- ++(0.28,0.76);
        \draw[continuation edge] (o10) -- ++(-0.28,0.76);
        \draw[continuation edge] (o10) -- ++(0.28,0.76);
        \foreach \leaf in {o0000,o0001,o0010,o0011,o1100,o1101,o1110,o1111} {
            \draw[continuation edge] (\leaf) -- ++(-0.2,0.72);
            \draw[continuation edge] (\leaf) -- ++(0.2,0.72);
        }

        \drawblock{o00}{o11}{3.8mm}{}
        \drawblock{o0000}{o0011}{3.8mm}{draw=red!65!black}
        \drawblock{o1100}{o1111}{3.8mm}{draw=red!65!black}

        \node[small node] at (o) {};
        \node[small node] at (o0) {};
        \node[small node] at (o1) {};
        \node[node] at (o00) {\scalebox{0.62}{$\ud_{\bigseq{00}}$}};
        \node[node] at (o01) {\scalebox{0.62}{$\ud_{\bigseq{01}}$}};
        \node[node] at (o10) {\scalebox{0.62}{$\ud_{\bigseq{10}}$}};
        \node[node] at (o11) {\scalebox{0.62}{$\ud_{\bigseq{11}}$}};
        \node[small node] at (o000) {};
        \node[small node] at (o001) {};
        \node[blue node] at (o0000) {\scalebox{0.56}{$\ud_{\bigseq{0000}}$}};
        \node[blue node] at (o0001) {\scalebox{0.56}{$\ud_{\bigseq{0001}}$}};
        \node[blue node] at (o0010) {\scalebox{0.56}{$\ud_{\bigseq{0010}}$}};
        \node[blue node] at (o0011) {\scalebox{0.56}{$\ud_{\bigseq{0011}}$}};
        \node[small node] at (o110) {};
        \node[small node] at (o111) {};
        \node[blue node] at (o1100) {\scalebox{0.56}{$\ud_{\bigseq{1100}}$}};
        \node[blue node] at (o1101) {\scalebox{0.56}{$\ud_{\bigseq{1101}}$}};
        \node[blue node] at (o1110) {\scalebox{0.56}{$\ud_{\bigseq{1110}}$}};
        \node[blue node] at (o1111) {\scalebox{0.56}{$\ud_{\bigseq{1111}}$}};

        \draw[map arrow] (2.55,1.25) -- (3.65,1.25);

        \coordinate (pempty) at (5,0);
        \coordinate (p0) at (4.25,2.5);
        \coordinate (p1) at (5.75,2.5);

        \draw[tree edge] (pempty) -- (p0);
        \draw[tree edge] (pempty) -- (p1);
        \foreach \leaf in {p0,p1} {
            \draw[continuation edge] (\leaf) -- ++(-0.28,0.76);
            \draw[continuation edge] (\leaf) -- ++(0.28,0.76);
        }

        \node[node] at (pempty) {\scalebox{0.72}{$\ud'_{\bigemptyseq}$}};
        \node[red node] at (p0) {\scalebox{0.72}{$\ud'_{\bigseq{0}}$}};
        \node[red node] at (p1) {\scalebox{0.72}{$\ud'_{\bigseq{1}}$}};
    \end{tikzpicture}
    \caption{Visualization of the tree $(\ud'_\nu : \nu \in 2^{<\omega})$ for $n=2$. The bottom dashed block is $\ud'_\emptyseq=\ud_{\seq{00}}\ud_{\seq{01}}\ud_{\seq{10}}\ud_{\seq{11}}$. Since $\bigseq{0}\uparrow 2=\bigseq{00}$ and $\bigseq{1}\uparrow 2=\bigseq{11}$, the two upper dashed blocks are $\ud'_{\seq{0}}$ and $\ud'_{\seq{1}}$, respectively. Any antichain (red) in the packed tree $(\ud'_\nu : \nu \in 2^{<\omega})$ corresponds to an antichain (blue) in the original tree $(\ud_\nu : \nu \in 2^{<\omega})$.}
    \label{figure_stretched_packed_tree_n_two}
\end{figure}
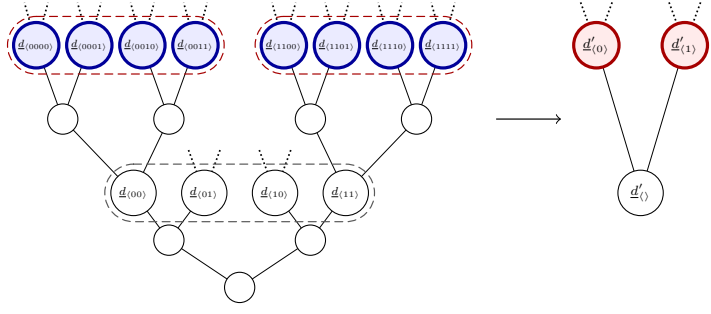

\noindent From here on, the argument is essentially the same as in Proposition 4.11 of \cite{AKLL25}, but with linear (in)dependence instead of algebraicity.

\begin{subclaim} \label{lemma_remove_param_lower}
    We may assume that there is some natural number $N > 0$ such that the partial type $\set{\psi(\ux_\sh\ux_{\seq{0}^k}; \uz; \ud_{\seq{0}^k}) : 0 \leq k \leq N}$ implies a disjunction of the form
    $$
    \bigvee\nolimits_{i=1}^q \varphi_i(\lambda_i(\ux_\sh\ux_{\seq{0}^N}); \ud_{\seq{0}^N})
    $$
    where each $\lambda_i(\ux_\sh\ux_\va)$ is a non-trivial $\LK$-term and each $L$-formula $\varphi_i(y; \uw)$ defines a finite subset of $\VV$ for any parameters plugged into $\uw$.
\begin{innerproof}
    Set $\uw' := \uw\uw_1\uw_2$, set $\psi'(\ux_\sh\ux_\va; \uz; \uw') := \psi(\ux_\sh\ux_\va; \uz; \uw)$, and set $\ud'_\nu := \ud_{\seq{00}{}^\frown \nu}\ud_{\emptyseq}\ud_{\seq{0}}$ (see Figure \ref{figure_shifted_tree_with_fixed_parameters}).
    Clearly, the tree $(\ud'_\nu : \nu \in 2^{<\omega})$ is strongly indiscernible, and, for any antichain $X \subset 2^{<\omega}$, the partial type $\set{\psi'(\ux_\sh\ux_\nu; \uz; \ud'_\nu) : \nu \in X}$ implies no finite disjunction of non-trivial linear dependencies in $\ux_\sh(\ux_\nu : \nu \in X)$ over $\VV$.

    Let the disjunction $\bigvee\nolimits_{i=1}^q \varphi_i(\lambda_i(\ux_\sh\ux_{\seq{0}}); \uw_0\uw_1)$ be as in Claim \ref{lemma_remove_x_lower}, and set $N := q+1$.
    Using the strong indiscernibility of the tree $(\ud_\nu : \nu \in 2^{<\omega})$ and the pigeonhole principle, we see that the partial type $\set{\psi'(\ux_\sh\ux_{\seq{0}^k}; \uz; \ud'_{\seq{0}^k}) : 0 \leq k \leq N}$ implies the finite disjunction
        \begin{align*}
            \bigvee\nolimits_{0 \leq j<k< N} \bigvee\nolimits_{i=1}^q \lambda_i(\ux_\sh\ux_{\seq{0}^{N}}) \in \big(\varphi_i(\MM; \ud_{\seq{0}^{2+j}}\ud_{\seq{0}^{2+N}}) \cap \varphi_i(\MM; \ud_{\seq{0}^{2+k}}\ud_{\seq{0}^{2+N}})\big)
        \end{align*}
        of non-trivial linear dependencies in $\ux_\sh\ux_{\seq{0}^N}$ over $\VV$.
    By Lemma \ref{lemma_fint_set_eq_or}, for $j < k < N$ we have
        $$
        \varphi_i(\MM; \ud_{\seq{0}^{2+j}}\ud_{\seq{0}^{2+N}}) \cap \varphi_i(\MM; \ud_{\seq{0}^{2+k}}\ud_{\seq{0}^{2+N}}) = \varphi_i(\MM; \ud_{\emptyseq}\ud_{\seq{0}^{2+N}}) \cap \varphi_i(\MM; \ud_{\seq{0}}\ud_{\seq{0}^{2+N}}),
        $$
        which is a finite $\ud'_{\seq{0}^N}$-definable set, as $\ud'_{\seq{0}^N} := \ud_{\seq{0}^{2+N}}\ud_{\emptyseq}\ud_{\seq{0}}$.
    Replace $\psi(\ux_\sh\ux_\va; \uz; \uw)$ with $\psi'(\ux_\sh\ux_\va; \uz; \uw')$ and $(\ud_\nu : \nu \in 2^{<\omega})$ with $(\ud'_\nu : \nu \in 2^{<\omega})$ to conclude.
\end{innerproof}
\end{subclaim}
\begin{figure}[tbp]
    \centering
    \begin{tikzpicture}[
        x=0.78cm,
        y=0.78cm,
        node/.style={circle, draw, fill=white, inner sep=1pt, minimum size=6mm, font=\scriptsize},
        long label node/.style={circle, draw, fill=white, inner sep=1pt, minimum size=6mm, font=\scriptsize},
        tree edge/.style={line width=0.35pt},
        continuation edge/.style={densely dotted, line width=0.65pt},
        fixed node/.style={node, draw=red!65!black, fill=red!8, very thick},
        moving node/.style={node, draw=blue!60!black, fill=blue!8, very thick},
        map arrow/.style={->, line width=0.45pt}
    ]
        \coordinate (oldroot) at (-2,0);
        \coordinate (oldzero) at ($(oldroot)+(-0.75,1.35)$);
        \coordinate (oldzerozero) at ($(oldzero)+(-0.75,1.35)$);
        \coordinate (oldzerozerozero) at ($(oldzerozero)+(-0.75,1.35)$);
        \coordinate (oldzerozeroone) at ($(oldzerozero)+(0.75,1.35)$);

        \draw[tree edge] (oldroot) -- (oldzero) -- (oldzerozero);
        \draw[continuation edge] (oldroot) -- ++(0.45,0.81);
        \draw[continuation edge] (oldzero) -- ++(0.45,0.81);
        \draw[tree edge] (oldzerozero) -- (oldzerozerozero);
        \draw[tree edge] (oldzerozero) -- (oldzerozeroone);
        \foreach \leaf in {oldzerozerozero,oldzerozeroone} {
            \draw[continuation edge] (\leaf) -- ++(-0.22,0.72);
            \draw[continuation edge] (\leaf) -- ++(0.22,0.72);
        }

        \node[fixed node] at (oldroot) {\scalebox{0.66}{$\ud_{\bigemptyseq}$}};
        \node[fixed node] at (oldzero) {\scalebox{0.72}{$\ud_{\bigseq{0}}$}};
        \node[moving node] at (oldzerozero) {\scalebox{0.62}{$\ud_{\bigseq{00}}$}};
        \node[moving node] at (oldzerozerozero) {\scalebox{0.56}{$\ud_{\bigseq{000}}$}};
        \node[moving node] at (oldzerozeroone) {\scalebox{0.56}{$\ud_{\bigseq{001}}$}};

        \draw[map arrow] (-1.55,2.7) -- (0.15,2.7);

        \coordinate (newroot) at (2.1,2.7);
        \coordinate (newzerozero) at (1.35,4.05);
        \coordinate (newzeroone) at (2.85,4.05);

        \draw[tree edge] (newroot) -- (newzerozero);
        \draw[tree edge] (newroot) -- (newzeroone);
        \foreach \leaf in {newzerozero,newzeroone} {
            \draw[continuation edge] (\leaf) -- ++(-0.28,0.76);
            \draw[continuation edge] (\leaf) -- ++(0.28,0.76);
        }

        \node[long label node] at (newroot) {\scalebox{0.66}{$\ud'_{\bigemptyseq}$}};
        \node[long label node] at (newzerozero) {\scalebox{0.72}{$\ud'_{\bigseq{0}}$}};
        \node[long label node] at (newzeroone) {\scalebox{0.72}{$\ud'_{\bigseq{1}}$}};
    \end{tikzpicture}
    \caption{Visualization of the tree $(\ud'_\nu : \nu \in 2^{<\omega})$ with $\ud'_\nu=\ud_{\seq{00}{}^\frown\nu}\ud_\emptyseq\ud_{\seq{0}}$. The blue subtree with root $\ud_{\seq{00}}$ supplies the moving first component, while the red parameters $\ud_\emptyseq$ and $\ud_{\seq{0}}$ are appended to every tuple. Every path in the tree $(\ud'_\nu : \nu \in 2^{<\omega})$ corresponds to a path in the original tree $(\ud_\nu : \nu \in 2^{<\omega})$ that starts with $\ud_{\emptyseq{}}$ and $\ud_{\seq{0}}$.}
    \label{figure_shifted_tree_with_fixed_parameters}
\end{figure}
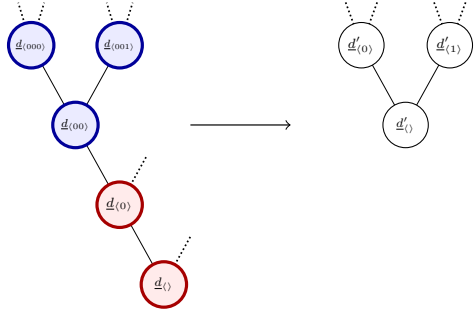

\begin{subclaim}
    The theory $T$ has $(N+1)$-\ATP{} for the $N > 0$ from Claim \ref{lemma_remove_param_lower}.
\begin{innerproof}
    Define the formula
    $
    \psi'(\ux_\sh; \uz; \uw) := \exists \ux_\va : \psi(\ux_\sh\ux_\va; \uz; \uw) \wedge \bigwedge\nolimits_{i=1}^q \neg \varphi_i(\lambda_i(\ux_\sh\ux_\va); \uw),
    $
    where all the $\varphi_i$'s and $\lambda_i$'s are as in Claim \ref{lemma_remove_param_lower}.
    Let $X$ be a finite antichain.
    Since $\set{\psi(\ux_\sh\ux_\nu; \uz; \ud_\nu) : \nu \in X}$ implies no finite disjunction of non-trivial linear dependencies in $\ux_\sh(\ux_\nu : \nu \in X)$ over $\VV$, and $\bigwedge_{\nu \in X}\bigwedge_{i=1}^q \neg \varphi_i(\lambda_i(\ux_\sh\ux_\nu); \ud_\nu)$ only excludes finitely many linear dependencies over $\VV$, the partial type $\set{\psi'(\ux_\sh; \uz; \ud_\nu) : \nu \in X}$ is consistent.

    Now let $\uv_\sh\ua$ be a realization of $\set{\psi'(\ux_\sh; \uz; \ud_{\seq{0}^k}) : 0 \leq k \leq N}$.
    By the definition of $\psi'(\ux_\sh; \uz; \uw)$, there are tuples $(\uv_{\seq{0}^k} : 0 \leq k \leq N)$ such that:
    \begin{enumerate}[(i)]
        \item $\uv_\sh(\uv_{\seq{0}^k} : 0 \leq k \leq N)\ua$ realizes $\set{\psi(\ux_\sh\ux_{\seq{0}^k}; \uz; \ud_{\seq{0}^k}) : 0 \leq k \leq N}$; and
        \item $\MM \models \bigwedge_{i=1}^q \neg \varphi_i(\lambda_i(\uv_\sh\uv_{\seq{0}^N}); \ud_{\seq{0}^N})$.
    \end{enumerate}
    However, by Claim \ref{lemma_remove_param_lower}, point (i) implies $\MM \models \bigvee_{i=1}^q \varphi_i(\lambda_i(\uv_\sh\uv_{\seq{0}^N}); \ud_{\seq{0}^N})$, contradicting (ii).
    Therefore, $\set{\psi'(\ux_\sh; \uz; \ud_{\seq{0}^k}) : 0 \leq k \leq N}$ must be inconsistent.
    By strong indiscernibility, this translates to all paths of length $N+1$.
\end{innerproof}
\end{subclaim}

\noindent Since $(N+1)$-\ATP{} implies $2$-\ATP{}, which is \ATP{}, we obtain that $T$ has \ATP{}.
\end{proof}

\section{Preservation of NATP for general kernel configurations} \label{sec_proof_general}
The goal of this section is to prove the same result as in Section \ref{sec_natp_ez}, but for arbitrary kernel configurations $C$. In other words we prove the following theorem:

\begin{theorem}\label{theorem_preserve_natp}
    Suppose that $T$ satisfies \Hfour{}. If $T$ has \NATP{}, then $T\theta^C$ also has \NATP{}.
\end{theorem}

\noindent We can use the same proof idea discussed at the beginning of Section \ref{sec_natp_ez}. Since Step (II) (i.e., Lemma \ref{lemma_has_li_atp_impl_atp}) does not depend in any way on the kernel configuration, it remains to prove Step (I) for arbitrary kernel configurations in order to prove Theorem \ref{theorem_preserve_natp}.
Working with $C = C_0$ in Section \ref{sec_step_red_ez} allowed us to completely neglect the parametrized $C$-sequence-systems that come with our characterization of definable sets (see Fact \ref{theorem_big_fml_preceise}).
This made it possible to obtain a formula of the form $\exists \ux \in \VV : \psi_\theta(\ux; \uz; \uw)$ that witnesses \ATP{} together with some strongly indiscernible tree $(\ud_\nu : \nu \in 2^{<\omega})$.
We then partitioned $\ux$ into two subtuples $\ux_\sh$ and $\ux_\va$ such that $\exists \ux_\va \in \VV : \psi_\theta(\ux_\sh\ux_\va; \uz; \uw)$ still witnessed \ATP{}.
Thus, one can think of $\ux_\sh$ as a subtuple $\ux' \subseteq \ux$ for which the witness of $\exists \ux'$ in $\set{\exists \ux \in \VV : \psi_\theta(\ux; \uz; \ud_\nu) : \nu \in X}$ can be chosen uniformly in $\nu$ for any antichain $X \subset 2^{<\omega}$.
With this splitting of $\ux$, we were able to show that whenever $\set{\psi(\uxvec{}_\sh\uxvec{}_\nu; \uz; \ud_\nu) : \nu \in X}$ implies a finite disjunction of non-trivial linear dependencies in $\uxvec{}_\sh(\uxvec{}_\nu : \nu \in X)$ over $\VV$ for some antichain $X \subset 2^{<\omega}$, there is another formula $\psi'(\uxvec{}_\sh'\uxvec{}'_\va; \uz; \uw')$ with $(|\ux'_\va|, |\ux'_\sh|) <_\Lex (|\ux_\va|, |\ux_\sh|)$ such that $\exists \ux'_\va \in \VV : \psi'_\theta(\ux'_\sh\ux'_\va; \uz; \uw')$ witnesses \ATP{}.
Under the assumption that $(|\ux_\va|, |\ux_\sh|)$ is minimal with respect to lexicographic order, this allowed us to show that $\set{\psi(\uxvec{}_\sh\uxvec{}_\nu; \uz; \ud_\nu) : \nu \in X}$ implies no finite disjunction of non-trivial linear dependencies in $\uxvec{}_\sh(\uxvec{}_\nu : \nu \in X)$ over $\VV$ for any antichain $X \subset 2^{<\omega}$.

We will generalize this proof idea to arbitrary kernel configurations $C$, but this comes with a few technical challenges.
Note that our characterization of definable sets, Fact \ref{theorem_big_fml_preceise}, will, in combination with some additional arguments, yield that a formula of the form
$$
\exists \ux \in \VV : \psi_\theta(\ux; \uz; \uw) \wedge S(\ux; \uy),
$$
where $\psi(\uxvec{}; \uz; \uw)$ is an $L$-formula and $S(\ux; \uy)$ is a parametrized $C$-sequence-system,
witnesses \ATP{} together with a strongly indiscernible tree $(\uu_\nu\ud_\nu : \nu \in 2^{<\omega})$ where each $\uu_\nu$ is compatible with $S$.
This raises the question of how we can partition $\ux$ into $\ux_\sh$ and $\ux_\va$ as above while being able
\begin{itemize}
    \item to measure the complexity of $\ux_\sh$ and $\ux_\va$ (as $(|\ux_\va|, |\ux_\sh|)$ did above together with the lexicographic order), and
    \item to reduce that complexity if $\set{\psi(\uxvec{}_\sh\uxvec{}_\nu; \uz; \ud_\nu) : \nu \in X}$ implies some finite disjunction of non-trivial linear dependencies in $\uxvec{}_\sh(\uxvec{}_\nu : \nu \in X)$ over $\VV$ for some antichain $X$.
\end{itemize}
Our answer to this question is not just to split the tuple $\ux$ into two, but to split the parametrized $C$-sequence-system $S(\ux; \uy)$ into a kind of ``pair of parametrized $C$-sequence-systems''.
In \cite{Chi25}, we have already defined the rank $\rk(S)$ and the degree $\deg(S)$ of a single parametrized $C$-sequence-system $S(\ux; \uy)$, which measure its complexity in a certain sense.
We have also shown that the conjunction of $S(\ux; \uy)$ and some $\LKThe$-equations in $\ux$ can, under some assumptions, essentially be reduced, or rather transformed, into another parametrized $C$-sequence-system $S'(\ux'; \uy')$ that satisfies $(\rk(S'), \deg(S')) <_\Lex (\rk(S), \deg(S))$; see our \TrafoLemma{} (Lemma~\ref{lemma_trafoo}).

Note that in the case $C = C_0$, a parametrized $C$-sequence-system $S(\ux; \uy)$ is just the formula $\top$; its rank is $|\ux|$ and its degree is $0$.
Thus, in a sense, we were already working with ``pairs of parametrized $C$-sequence-systems'' in the case $C = C_0$.

\subsection{A recap on transformability} \label{sec_trafo_recap}

\noindent Before we introduce the previously mentioned ``pairs of parametrized $C$-sequence-systems'', we recall the notion of transformability from \cite{Chi25}:

\begin{definition} \label{def_transfo}
    Let $E(\ux; \uy)$ and $E'(\ux'; \uy)$ be two conjunctions of $\LRC$-equations.
    We say that $E(\ux; \uy)$ is \textbf{transformable} into $E'(\ux'; \uy)$ if both
    \begin{enumerate}[(i)]
        \item $\TKvsThe \cup \set{\text{``$\theta$ is $C$-image-complete''}} \models \forall \ux\uy: E(\ux; \uy) \rightarrow E'(\underline{\nu}(\ux); \uy) \wedge \ux = \utau(\underline{\nu}(\ux); \uy)$,
        \item $\TKvsThe \cup \set{\text{``$\theta$ is $C$-image-complete''}} \models \forall \ux'\uy: E'(\ux'; \uy) \rightarrow E(\utau(\ux'; \uy); \uy)$
    \end{enumerate}
    hold for some tuple $\underline{\nu}(\ux)$ of $\LRC$-terms and some tuple $\utau(\ux'; \uy)$ with each entry $\tau_k(\ux'; \uy)$ being the sum of an $\LKThe$-term in $\ux'$ and an $\LRC$-term in $\uy$.
    We say that the tuples $\underline{\nu}(\ux)$ and $\utau(\ux'; \uy)$ \textbf{witness} the transformability.
\end{definition}

\noindent When $\utau(\ux'; \uy)$ does not depend on $\uy$, we may simply write $\utau(\ux')$. Transformability can be thought of as a stronger form of ``equivalence up to existence''.
This is because $E(\ux; \uy)$ being transformable into $E'(\ux'; \uy)$ implies 
$$
\exists \ux \in \VV : \phi(\ux; \uy) \wedge E(\ux; \uy) \quad \equiv\quad \exists \ux' \in \VV : \phi(\utau(\ux'; \uy); \uy) \wedge E'(\ux'; \uy)
$$ for every formula $\phi(\ux; \uy)$. 
This allows us to apply the results of this section to $L_\theta$-formulas. We used this notion first in \cite{Chi25} to prove our characterization of existentially closed models of $T_\theta^C$ (Theorem \ref{theorem_big_characterization}) and then in \cite{Chi25b} to prove our characterization of definable sets (Fact \ref{theorem_big_fml_preceise}). In those applications, $E'(\ux'; \uy)$ was always essentially a parametrized $C$-sequence-system.

\begin{observation} \label{obser_trafo_nice_2}
    Looking at the definition of transformability, one can easily see the following:
    \begin{enumerate}[(i)]
        \item Let $E_1(\ux_1; \uy)$ and $E_2(\ux_2; \tiluy)$ be two conjunctions of $\LRC$-equations, let $\ut(\ux_1; \uy)$ be a tuple of $\LRC$-terms, and let $\ux'_2$ be another tuple with $|\ux'_2| = |\ux_2|$.
        Assume that $E_1(\ux_1; \uy)$ is transformable into $E'_1(\ux'_1; \uy)$, witnessed by $\underline{\nu}_1(\ux_1)$ and $\utau{}_1(\ux'_1; \uy)$.

        Then $E_1(\ux_1; \uy) \wedge E_2(\ux_2; \ut(\ux_1; \uy))$ is transformable into $E'_1(\ux'_1; \uy) \wedge E_2(\ux'_2; \ut(\utau{}_1(\ux'_1; \uy); \uy))$, witnessed by the tuples $\underline{\nu}_1(\ux_1)\ux_2$ and $\utau{}_1(\ux'_1; \uy)\ux'_2$.
        \item Let $E_1(\ux_1; \uy)$, $E_*(\ux_1; \uy)$, and $E_2(\ux_2; \ux_1\uy)$ be three conjunctions of $\LRC$-equations.
        Assume that $E_1(\ux_1; \uy) \wedge E_*(\ux_1; \uy)$ is transformable into $E'_1(\ux'_1; \uy)$, witnessed by $\underline{\nu}_1(\ux_1)$ and $\utau{}_1(\ux'_1; \uy)$, and assume that $E_2(\ux_2; \ux_1\uy)$ is transformable into $E'_2(\ux'_2; \ux_1\uy) \wedge E_*(\ux_1; \uy)$, witnessed by $\underline{\nu}_2(\ux_2)$ and $\utau{}_2(\ux'_2; \ux_1\uy)$.

        Then the conjunction $E(\ux_1\ux_2; \uy) := E_1(\ux_1; \uy) \wedge E_2(\ux_2; \ux_1\uy)$ is transformable into the conjunction $E'(\ux'_1\ux'_2; \uy) := E'_1(\ux'_1; \uy) \wedge E'_2(\ux'_2; \utau{}_1(\ux'_1; \uy)\uy)$, witnessed by the tuples $\underline{\nu}_1(\ux_1)\underline{\nu}_2(\ux_2)$ and $\utau{}_1(\ux'_1; \uy)\utau{}_2(\ux'_2; \utau{}_1(\ux'_1; \uy)\uy)$.
    \end{enumerate}
\end{observation}

\begin{remark}
    Transformability depends on which variables are treated as parameters.
    For example, if $E(\ux; \uy)$ is transformable into $E'(\ux'; \uy)$, then \hbox{$E_*(x_*\ux; \uy) := E(\ux; \uy)$} (which is logically equivalent to $E(\ux; \uy)$) is not transformable into $E'(\ux'; \uy)$ (but it is transformable into $E'_*(x_*\ux'; \uy) := E'(\ux'; \uy)$).
    Also, transformability is not symmetric in general, as we require each entry of $\utau(\ux';\uy)$ to be the sum of an $\LKThe$-term in $\ux'$ and an $\LRC$-term in $\uy$, whereas $\underline{\nu}(\ux)$ can be any tuple of $\LRC$-terms in $\ux$.
\end{remark}

\noindent In \cite{Chi25} and \cite{Chi25b}, we mostly used transformability in the setting where $E(\ux; \uy)$ was essentially a conjunction of some parametrized $C$-sequence-system and some $\LKThe$-equations. In \cite{Chi25}, we showed that such a conjunction was (under some assumptions) transformable into a ``smaller'' parametrized $C$-sequence-system. We will now make this precise:

\begin{definition} \label{def_rk_deg_def}
    Let $S(\ux; \uy) = \bigwedge_{k=1}^n f_k^{q_k}[\theta](x_{\ldd, k}) = y_k$ be a parametrized $C$-sequence-system, as in Definition \ref{def_param_c_sequence_system} (i.e., $\ux = \ux_\li\ux_\ld$ and so on).
    We define:
    \begin{enumerate}[(i)]
        \item The \textbf{rank} of $S(\ux; \uy)$ as $\rk(S) := m = |\ux_\li|$.
        \item The \textbf{degree} of $S(\ux; \uy)$ as $\deg(S) := \sum\nolimits_{k=1}^n \deg(f_k^{q_k})$.
    \end{enumerate}
    We define the rank and degree of a $C$-sequence-system (i.e., a parametrized $C$-sequence-system with compatible parameters plugged in) to be the rank and degree of the underlying parametrized $C$-sequence-system.
\end{definition}

\begin{definition}
    \label{def_compatible_pair} Let $S(\ux; \tiluy) = \bigwedge_{k=1}^n f_k^{q_k}[\theta](x_{\ldd, k}) = \Tilde{y}{}_k$ be a parametrized $C$-sequence-system as in Definition \ref{def_param_c_sequence_system}.
    Let $\varphi(\uy)$ be a conjunction of $\LRC$-equations, and let $\umu{}(\uy)$ be a tuple of $\LRC$-terms, where $\uy$ is another tuple of variables.
    If
    $$
    \TKvsThe \cup \set{\text{``$\theta$ is $C$-image-complete''}} \models \forall \uy \in \VV : \varphi(\uy) \rightarrow \text{``$\umu{}(\uy)$ is compatible with $S$''}
    $$
    holds, we say that the pair $(\varphi, \umu{})$ is \textbf{compatible} with the parametrized $C$-sequence-system $S$.
\end{definition}

\begin{lemma}[Transformation Lemma, Lemma 4.11 in \cite{Chi25}]\label{lemma_trafoo}
    Let $S(\ux; \uy{}_1)$ be a parametrized $C$-sequence-system, let $E(\ux; \uy{}_2)$ be a conjunction of $\LKThe$-equations, and let $\umu{}_1(\uy)$ and $\umu{}_2(\uy)$ be tuples of $\LRC$-terms.
    Then there is another parametrized $C$-sequence-system $S'(\ux'; \uy')$ and a pair $(\varphi', \umu{}')$ compatible with $S'$ such that
    $$
    \text{$S(\ux; \umu{}_1(\uy)) \wedge E(\ux; \umu{}_2(\uy))\quad$ is transformable into $\quad\varphi'(\uy) \wedge S'(\ux'; \umu{}'(\uy))$}.
    $$
    Furthermore, $(\rk(S'), \deg(S')) \leq_\Lex (\rk(S), \deg(S))$ holds, and if $E(\ux; \uy{}_2)$ contains at least one equation that is non-trivial in $\ux$ and bounded by $S$, then this inequality is strict.
\end{lemma}

\noindent So, for any (parametrized) $C$-sequence-system, $(\rk(S), \deg(S))$ lies in $\NN^2$.  
Notice that $(\NN^2, <_\Lex)$ is a well-ordered set, where $<_\Lex$ is the lexicographical order, which defines the first entry to be more significant than the second entry.
This together with the \TrafoLemma{} made it possible to prove both our axiomatization of existentially closed models (Theorem \ref{theorem_big_characterization}) and the characterization of definable sets in them (Fact  \ref{theorem_big_fml_preceise}) via some induction on a (parametrized) $C$-sequence-system. The \TrafoLemma{} above also has the following consequence:

\begin{corollary} \label{corollary_trafo_any_lrc}
    Any conjunction $E(\ux; \uy)$ of $\LRC$-equations is transformable into a formula of the form
    $$
    \varphi(\uy) \wedge S(\ux'; \umu{}(\uy)),
    $$
    where $S(\ux'; \uy')$ is a parametrized $C$-sequence-system, and $(\varphi, \umu{})$ is a pair compatible with $S$.

    Furthermore, if $\LRC$-terms $r_1(\ux; \uy), \dots, r_q(\ux; \uy)$ are given, then we can choose $S(\ux'; \uy')$ and $(\varphi, \umu{})$ such that transformability is witnessed by tuples $\underline{\nu}{}(\ux)$ and $\utau(\ux'; \uy)$ with the following property: For each $k \in \set{1, \dots, q}$, there is an $\LKThe$-term $s_k(\ux')$ and an $\LRC$-term $r'_k(\uy)$ such that $\varphi(\uy) \wedge S(\ux'; \umu{}(\uy))$ implies
    $$
    r_k(\utau(\ux'; \uy); \uy) = s_k(\ux') + r'_k(\uy).
    $$
\begin{proof}
    Assume first that $C$ is transcendental, and let $\ux = (x_1, \dots, x_n)$.
    By the \commonBaseTheorem{} (Fact \ref{theorem_r_c_element}), we can assume that $E(\ux; \uy)$ is of the form
    $$
    \bigwedge\nolimits_{k=1}^m \Big(\sum\nolimits_{l=1}^n \Big( \rho_{k, l}[\theta] (\eta[\theta]^{-1} \circ \pi_{\Image(F^C)}(x_l)) + \sum\nolimits_{f\in F} \rho_{f, k, l}[\theta] (\pi_{\Ker(f^C)}(x_l))\Big) + t_k(\uy) = 0\Big),
    $$
    where each $t_k(\uy)$ is an $\LRC$-term in $\uy$, $F \subseteq \Kp{0<C<\infty}$ is finite, and $\Fac(\eta) \subseteq F \cup \Kp{C=0}$.
    Let $\ux_F := (x_{F, 1}, \dots, x_{F, n})$ and $\ux_f := (x_{f, 1}, \dots, x_{f, n})$ for all $f \in F$ be new tuples of variables with the same length as $\ux$.
    With
    \begin{align*}
        \underline{\nu}(\ux) &:= (\eta \cdot F^C)[\theta]^{-1} \circ \pi_{\Image(F^C)}(\ux) (\pi_{\Ker(f^C)}(\ux) : f \in F), \\
        \utau(\ux_F(\ux_f : f \in F); \uy) &:= (\eta \cdot F^C)[\theta](\ux_F) + \sum\nolimits_{f\in F} \ux_f,
    \end{align*}
    and $F^C[\theta] \circ F^C[\theta]^{-1} = \pi_{\Image(F^C)}$, one can show that $E(\ux; \uy)$ is transformable into
    \begin{align}
        &\bigwedge\nolimits_{k=1}^m \Big( \sum\nolimits_{l=1}^n \Big( (\rho_{k, l} \cdot F^C)[\theta](x_{F, l}) + \sum\nolimits_{f\in F} \rho_{f, k, l}[\theta](x_{f, l}) \Big) + t_k(\uy) = 0 \Big) \notag \\
        &\hspace{190pt}\wedge \Big( \bigwedge\nolimits_{f\in F}\bigwedge\nolimits_{l=1}^n f^C[\theta](x_{f, l}) = 0 \Big). \label{tag_big_conjuctionn}
    \end{align}
    The conjunction on the second line is a parametrized $C$-sequence-system with $\ux_\li := \ux_F$ and $\ux_\ld := (\ux_f : f \in F)$ and the trivial tuple $\uzero$, treated as a tuple of $\LKThe$-terms in $\uy$, plugged in.
    Now apply the \TrafoLemma{} and the transitivity of transformability.
    The case where $C$ is algebraic follows similarly by choosing $F = \Kp{0<C<\infty}$, which gives $\Image(F^C) = \set{0}$.

    We now show the furthermore part.
    We only show it in the transcendental case, since the algebraic case is analogous.
    Notice that we can apply the \commonBaseTheorem{} so that, for all $k \in \set{1, \dots, q}$, we also have
    $$
    r_k(\ux; \uy) = \sum\nolimits_{l=1}^n \Big( \xi_{k, l}[\theta] \circ \eta[\theta]^{-1} \circ \pi_{\Image(F^C)}(x_l) + \sum\nolimits_{f\in F} \xi_{f, k, l}[\theta] \circ \pi_{\Ker(f^C)}(x_l)\Big) + r^*_{k}(\uy),
    $$
    where each $\xi$ is a polynomial, and each $r^*_{k}(\uy)$ is an $\LRC$-term.
    Clearly, $\bigwedge\nolimits_{f\in F}\bigwedge\nolimits_{l=1}^n f^C[\theta](x_{f, l}) = 0$ implies $x_{f, l} \in \Ker(f^C)$ for all $f \in F$ and $l \in \set{1, \dots, n}$.
    Hence, this conjunction implies
    \begin{align}
        r_k(\utau(\ux_F(\ux_f : f \in F)); \uy) = \sum\nolimits_{l=1}^n \Big( (\xi_{k, l} \cdot F^C)[\theta](x_{F, l}) + \sum\nolimits_{f\in F} \xi_{f, k, l}[\theta](x_{f, l})\Big) + r^*_k(\uy). \label{tag_eq_for_tern}
    \end{align}
    Now let $\underline{\nu}{}'(\ux_F(\ux_f : f \in F))$ and $\utau'(\ux'; \uy)$ witness the transformability of (\ref{tag_big_conjuctionn}) into the conjunction $\varphi(\uy) \wedge S(\ux'; \umu{}(\uy))$, where (\ref{tag_big_conjuctionn}) refers to both lines.
    By (ii) of Definition \ref{def_transfo}, $\varphi(\uy) \wedge S(\ux'; \umu{}(\uy))$ implies $\bigwedge\nolimits_{f\in F}\bigwedge\nolimits_{l=1}^n f^C[\theta](\tau'_{f, l}(\ux'; \uy)) = 0$, where each $\tau'_{f, l}(\ux'; \uy)$ is the entry of $\utau'(\ux'; \uy)$ that corresponds to $x_{f, l}$.
    Therefore, $\varphi(\uy) \wedge S(\ux'; \umu{}(\uy))$ implies that (\ref{tag_eq_for_tern}) holds with $\utau'(\ux'; \uy)$ plugged in for $\ux_F(\ux_f : f \in F)$.
    Since any entry of $\utau'(\ux'; \uy)$ can, by the definition of transformability, be written as the sum of an $\LKThe$-term in $\ux'$ and an $\LRC$-term in $\uy$, we now see that $\varphi(\uy) \wedge S(\ux'; \umu{}(\uy))$, together with (\ref{tag_eq_for_tern}), implies
    $$
    r_k(\utau(\utau'(\ux'; \uy)); \uy) = s_k(\ux') + r'_k(\uy)
    $$
    for some $\LKThe$-term $s_k(\ux')$ and some $\LRC$-term $r'_k(\uy)$.
    Finally, since $\underline{\nu}{}'(\underline{\nu}(\ux))$ and $\utau(\utau'(\ux'; \uy))$ witness that $E(\ux; \uy)$ is transformable into $\varphi(\uy) \wedge S(\ux'; \umu{}(\uy))$, this concludes the furthermore part.
\end{proof}
\end{corollary}

\subsection{Staged $C$-sequence-system} \label{sec_staged_c_ss}

We will now give a precise definition of what we previously referred to as ``pairs of parametrized $C$-sequence-systems'':

\begin{definition} \label{def_new_haram_staged_c_ss}
    A \textbf{parametrized staged $\mathbf{C}$-sequence-system} is a formula of the form
    $$
    S(\ux_\sh\ux_\va; \uy) = S_\sh(\ux_\sh; \uy_\sh) \wedge S_\va(\ux_\va; \ut(\ux_\sh) + \uy_\va)
    $$
    where $S_\sh(\ux_\sh; \uy_\sh)$ and $S_\va(\ux_\va; \uy_\va)$ are both parametrized $C$-sequence-systems, as in Definition \ref{def_param_c_sequence_system}, $\uy = \uy_\sh\uy_\va$, and $\ut(\ux_\sh)$ is a tuple of $\LKThe$-terms, such that
    $$
    \TKvsThe \models \forall \ux_\sh : S_\sh(\ux_\sh; \uzero) \rightarrow \text{``$\ut(\ux_\sh)$ is compatible with $S_\va$''}.
    $$
\end{definition}

\noindent As with (parametrized) $C$-sequence-systems, we will use the letter $S$ to denote parametrized staged $C$-sequence-systems.
Note that we will write $S(\ux_\sh\ux_\va; \uy)$ for parametrized staged $C$-sequence-systems and, as before, $S(\ux; \uy)$ for ``regular'' parametrized $C$-sequence-systems.

\begin{notationnumber} \label{not_expanded}
    We may write a formula $S(\ux_\sh\ux_\va; \uy) = S_\sh(\ux_\sh; \uy_\sh) \wedge S_\va(\ux_\va; \ut(\ux_\sh) + \uy_\va)$, where $S_\sh(\ux_\sh; \uy_\sh)$ and $S_\va(\ux_\va; \uy_\va)$ are both parametrized $C$-sequence-systems, and $\ut(\ux_\sh)$ is a tuple of $\LKThe$-terms, in \textbf{expanded form}, i.e.,
    $$
    S(\ux_\sh\ux_\va; \uy) = \underbrace{\bigwedge\nolimits_{k=1}^{n_\sh} f_{\sh, k}^{q_{\sh, k}}[\theta](x_{\sh, \ld, k}) = y_{\sh, k}}_{S_\sh(\ux_\sh; \uy_\sh)} \wedge \underbrace{\bigwedge\nolimits_{k=1}^{n_\va} f_{\va, k}^{q_{\va, k}}[\theta](x_{\va, \ld, k}) = t_k(\ux_\sh) + y_{\va, k}}_{S_\va(\ux_\va; \ut(\ux_\sh) + \uy_\va)}
    $$
    where everything is as for usual parametrized $C$-sequence-systems (see Definition \ref{def_param_c_sequence_system}), i.e., the tuples are $\ux_\sh = \ux_{\sh, \li}\ux_{\sh, \ld} = (x_{\sh, \li, 1}, \dots, x_{\sh, \li, m_{\sh}})(x_{\sh, \ld, 1}, \dots, x_{\sh, \ld, n_{\sh}})$, $\uy_\sh = (y_{\sh, 1}, \dots, y_{\sh, n_\sh})$, $\ux_\va = \ux_{\va, \li}\ux_{\va, \ld} = (x_{\va, \li, 1}, \dots, x_{\va, \li, m_{\va}})(x_{\va, \ld, 1}, \dots, x_{\va, \ld, n_{\va}})$, $\uy_\va = (y_{\va, 1}, \dots, y_{\va, n_\va})$, and:
    \begin{enumerate}[(i)]
        \item $m_\sh = m_\va = 0$ holds if $C$ is algebraic,
        \item for all $* \in \set{\sh, \va}$ and $k \in \set{1, \dots, n_*}$ the irreducible polynomial $f_{*, k}$ lies in $\Kp{0<C}$ and $q_{*, k} \in \set{q \in \NN : 0 < q \leq C(f_{*, k})}$.
    \end{enumerate}
    Furthermore, $\uy = \uy_\sh\uy_\va$ and $\ut(\ux_\sh) = (t_1(\ux_\sh), \dots, t_{n_\va}(\ux_\sh))$.
\end{notationnumber}

\noindent If we say that $S(\ux_\sh\ux_\va; \uy) = \bigwedge\nolimits_{k=1}^{n_\sh} f_{\sh, k}^{q_{\sh, k}}[\theta](x_{\sh, \ld, k}) = y_{\sh, k} \wedge \bigwedge\nolimits_{k=1}^{n_\va} f_{\va, k}^{q_{\va, k}}[\theta](x_{\va, \ld, k}) = t_k(\ux_\sh) + y_{\va, k}$ is in expanded form, we implicitly assume that all tuples of variables and terms are as outlined above and that both (i) and (ii) hold.

Our first goal will be to extend all the usual definitions, such as compatibility and boundedness, from ``regular'' $C$-sequence-systems to our new parametrized staged $C$-sequence-systems.
The $\TKvsThe \models \forall \ux_\sh : S_\sh(\ux_\sh; \uzero) \rightarrow \text{``$\ut(\ux_\sh)$ is compatible with $S_\va$''}$ condition in Definition \ref{def_new_haram_staged_c_ss} is sometimes a bit hard to work with, so we will give some alternatives for it in the following lemma:

\begin{lemma} \label{lemma_staged_c_ss_iff}
    Given a formula $S(\ux_\sh\ux_\va; \uy) \hspace{-0.26pt}=\hspace{-0.26pt} S_\sh(\ux_\sh; \uy_\sh) \wedge S_\va(\ux_\va; \ut(\ux_\sh) + \uy_\va)$, where $S_\sh(\ux_\sh; \uy_\sh)$ and $S_\va(\ux_\va; \uy_\va)$ are both parametrized $C$-sequence-systems, and $\ut(\ux_\sh)$ is a tuple of $\LKThe$-terms, the following are equivalent:
    \begin{enumerate}[(i)]
        \item $
        \TKvsThe \models \forall \ux_\sh : S_\sh(\ux_\sh; \uzero) \rightarrow \text{``$\ut(\ux_\sh)$ is compatible with $S_\va$''}
        $, i.e., the formula $S(\ux_\sh\ux_\va; \uy)$ is a parametrized staged $C$-sequence-system.
        \item $
        \TKvsThe \cup \set{\text{``$\theta$ is $C$-image-complete''}} \models \forall \ux_\sh : S_\sh(\ux_\sh; \uzero) \rightarrow \text{``$\ut(\ux_\sh)$ is compatible with $S_\va$''}.
        $
        \item Writing $S(\ux_\sh\ux_\va; \uy) = \bigwedge\nolimits_{k=1}^{n_\sh} f_{\sh, k}^{q_{\sh, k}}[\theta](x_{\sh, \ld, k}) = y_{\sh, k} \wedge \bigwedge\nolimits_{k=1}^{n_\va} f_{\va, k}^{q_{\va, k}}[\theta](x_{\va, \ld, k}) = t_k(\ux_\sh) + y_{\va, k}$ in expanded form, for each $k \in \set{1, \dots, n_\va}$ with $C(f_{\va, k}) < \infty$, there is an $\LKThe$-term $s_k(\uy)$ such that
        $$
        \TKvsThe \models \forall \ux_\sh\uy : S_\sh(\ux_\sh; \uy_\sh) \rightarrow f_{\va, k}^{C-q_{\va, k}}[\theta](t_k(\ux_\sh) + y_{\va, k}) = s_k(\uy).
        $$
    \end{enumerate}
\begin{proof}
    The implications ``(i) $\Rightarrow$ (ii)'' and ``(iii) $\Rightarrow$ (i)'' are clear.
    For ``(ii) $\Rightarrow$ (iii)'', fix some $j \in \set{1, \dots, n_\va}$ with $C(f_{\va, j}) < \infty$.
    Obviously, there is an $\LKThe$-term
    $$
    s(\ux_\sh; \uy) = \sum\nolimits_{l=1}^{m_\sh} \rho_{\li, l}[\theta](x_{\sh, \li, l}) + \sum\nolimits_{l=1}^{n_\sh} \rho_{\ld, l}[\theta](x_{\sh, \ld, l}) + r(\uy)
    $$
    where all the $\rho$'s are polynomials and $r(\uy)$ is an $\LKThe$-term, such that
    \begin{align}
        \TKvsThe \models \forall \ux_\sh\uy : \Big(\bigwedge\nolimits_{k=1}^{n_\sh} f_{\sh, k}^{q_{\sh, k}}[\theta](x_{\sh, \ld, k}) = y_{\sh, k}\Big) \rightarrow f^{C-q_{\va, j}}_{\va, j}[\theta](t_j(\ux_\sh) + y_{\va, j}) = s(\ux_\sh; \uy). \label{tag_implication_with_s}
    \end{align}
    With some Euclidean divisions, we can furthermore ensure that the inequality $\deg(\rho_{\ld, l}) < \deg(f^{q_{\sh, l}}_{\sh, l})$ holds for all $l \in \set{1, \dots, n_\sh}$.
    We need to show that all the $\rho$'s are actually $0$ in order to prove (iii).
    Let $\Sigma(\ux_\sh)$ denote the partial type that states that the sequence
    $$
    (\theta^i(x_{\sh, \li, k}) : k \in \set{1, \dots, m_\sh}, i \in \omega)^\frown (\theta^i(x_{\sh, \ld, k}) : k \in \set{1, \dots, n_\sh}, 0 \leq i < \deg(f_{\sh, k}^{q_{\sh, k}}))
    $$
    is linearly independent.
    Since for every formula in $\Sigma(\ux_\sh)$, the corresponding $\LK$-formula in $\uxvec{}_\sh$ is bounded by $S_\sh$ and implies no finite disjunction of non-trivial linear dependencies in $\uxvec{}_\sh$, we can use our characterization of existentially closed models and compactness to see that the partial type $\Sigma(\ux_\sh) \cup \set{S_\sh(\ux_\sh; \uzero)}$ has a realization $\uv_\sh$ in some $(\VV, \theta) \models \TKvs\theta^C$.
    Since existentially closed models are $C$-image-complete, we obtain
    $$
    0 \underset{\text{(ii)}}{=} f^{C-q_{\va, j}}_{\va, j}[\theta](t_j(\uv_\sh)) \underset{\text{(\ref{tag_implication_with_s})}}{=} s(\uv_\sh; \uzero) = \sum\nolimits_{l=1}^{m_\sh} \rho_{\li, l}[\theta](v_{\sh, \li, l}) + \sum\nolimits_{l=1}^{n_\sh} \rho_{\ld, l}[\theta](v_{\sh, \ld, l}).
    $$
    However, since $\uv_\sh$ realizes $\Sigma(\ux_\sh)$, and $\deg(\rho_{\ld, l}) < \deg(f_{\sh, l}^{q_{\sh, l}})$ holds for all $l \in \set{1, \dots, n_\sh}$, all $\rho$'s must indeed be $0$.
\end{proof}
\end{lemma}

\noindent We now introduce the definition of compatibility for parametrized staged $C$-sequence-systems:

\begin{definition} \label{def_new_haram_comp}
     Fix $(\VV, \theta) \models \TKvsThe$.
    \begin{enumerate}[(i)]
        \item Let a parametrized staged $C$-sequence-system $S(\ux_\sh\ux_\va; \uy) = S_\sh(\ux_\sh; \uy_\sh) \wedge S_\va(\ux_\va; \ut(\ux_\sh) + \uy_\va)$, as in Definition \ref{def_new_haram_staged_c_ss}, be given.
        We call a tuple $\uu := \uu_\sh\uu_\va \in \VV$ \textbf{compatible} with $S$ if the following two conditions hold:
        \begin{enumerate}[(a)]
            \item $\uu_\sh$ is compatible with $S_\sh$.
            \item Given any extension $(\VV', \theta') \supseteq (\VV, \theta)$ with $(\VV', \theta') \models \TKvsThe$, and any $\uv_\sh \in \VV'$ with $(\VV', \theta') \models S_\sh(\uv_\sh; \uu_\sh)$, the tuple $\ut(\uv_\sh) + \uu_\va$ is compatible with $S_\va$.
        \end{enumerate}
        \item We define a \textbf{staged $\bm{C}$-sequence-system} over $(\VV, \theta)$ to be a parametrized staged $C$-sequence-system $S(\ux_\sh\ux_\va; \uy)$ with some compatible tuple $\uu$ plugged in for $\uy$.
    \end{enumerate}
\end{definition}

\noindent Recall that the definition of a tuple being compatible with a ``regular'' $C$-sequence-system is just that certain entries of it lie in some kernels.
Hence, we can easily find a first-order formula for this.
However, (i) in Definition \ref{def_new_haram_comp} refers to extensions, so it is not as easy to find such a formula for our parametrized staged $C$-sequence-systems; nevertheless, (iii) of Lemma \ref{lemma_staged_c_ss_iff} makes it possible:

\begin{lemma} \label{corollary_comP_fml_staged}
    Given a parametrized staged $C$-sequence-system
    $$
    S(\ux_\sh\ux_\va; \uy) = \bigwedge\nolimits_{k=1}^{n_\sh} f_{\sh, k}^{q_{\sh, k}}[\theta](x_{\sh, \ld, k}) = y_{\sh, k} \wedge \bigwedge\nolimits_{k=1}^{n_\va} f_{\va, k}^{q_{\va, k}}[\theta](x_{\va, \ld, k}) = t_k(\ux_\sh) + y_{\va, k}
    $$ written in expanded form (see Notation \ref{not_expanded}), a tuple $\uu = \uu_\sh\uu_\va$ is compatible with $S(\ux_\sh\ux_\va; \uy)$, as in Definition \ref{def_new_haram_comp} above, if and only if the following two conditions hold:
    \begin{enumerate}[(a)]
        \item $u_{\sh, k} \in \Ker(f^{C-q_{\sh, k}}_{\sh, k})$ for all $k \in \set{1, \dots, n_\sh}$ with $C(f_{\sh, k}) < \infty$.
        \item $s_k(\uu) = 0$ for all $k \in \set{1, \dots, n_\va}$ with $C(f_{\va, k}) < \infty$.
        Here the $s_k(\uy)$'s are the $\LKThe$-terms from (iii) of Lemma \ref{lemma_staged_c_ss_iff}.
    \end{enumerate}
    In particular, there is a conjunction of $\LKThe$-equations $\delta_S(\uy)$ that, modulo $\TKvsThe$, expresses ``$\uy$ is compatible with $S$''.
\begin{proof}
    By the definition of compatibility for parametrized $C$-sequence-systems (see Definition \ref{def_compatible}), the tuple $\uu_\sh$ is compatible with $S_\sh(\ux_\sh; \uy_\sh) := \bigwedge\nolimits_{k=1}^{n_\sh} f_{\sh, k}^{q_{\sh, k}}[\theta](x_{\sh, \ld, k}) = y_{\sh, k}$ if and only if (a) holds.
    Similarly, given some $\uv_\sh$ that realizes $S_\sh(\ux_\sh; \uu_\sh)$, the tuple $\ut(\uv_\sh) + \uu_\va$ is compatible with the parametrized $C$-sequence-system $S_\va(\ux_\va; \uy_\va) := \bigwedge\nolimits_{k=1}^{n_\va} f_{\va, k}^{q_{\va, k}}[\theta](x_{\va, \ld, k}) = y_{\va, k}$ if and only if $f^{C-q_{\va, k}}_{\va, k}[\theta](t_k(\uv_\sh) + u_{\va, k}) = 0$ holds for all $k \in \set{1, \dots, n_\va}$ with $C(f_{\va, k}) < \infty$.
    By (iii) of Lemma \ref{lemma_staged_c_ss_iff}, we have the equality $f^{C-q_{\va, k}}_{\va, k}[\theta](t_k(\uv_\sh) + u_{\va, k}) = s_k(\uu)$ for those $k$, since $\uv_\sh$ realizes $S_\sh(\ux_\sh; \uu_\sh)$.
\end{proof}
\end{lemma}

\noindent With the formula $\delta_S(\uy)$ from Lemma \ref{corollary_comP_fml_staged} above, we can now define the analogue, for parametrized staged $C$-sequence-systems, of a compatible pair for ``regular'' parametrized $C$-sequence-systems (see Definition \ref{def_compatible_pair}):

\begin{definition} \label{def_comp_pair_staged}
Let
$
S(\ux_\sh\ux_\va; \tiluy) = S_\sh(\ux_\sh; \tiluy_\sh) \wedge S_\va(\ux_\va; \ut(\ux_\sh) + \tiluy_\va)
$
be a parametrized staged $C$-sequence-system, let $\varphi(\uy)$ be a conjunction of $\LRC$-equations, and let $\umu{}(\uy) = \umu{}_\sh(\uy)\umu{}_\va(\uy)$ be a tuple of $\LRC$-terms.
We say that the pair $(\varphi, \umu{})$ is \textbf{compatible} with $S$ if
        $$
        \TKvsThe \cup \set{\text{``$\theta$ is $C$-image-complete''}} \models \forall \uy \in \VV : \varphi(\uy) \rightarrow \text{``$\umu{}(\uy)$ is compatible with $S$''}.
        $$
\end{definition}

\begin{lemma} \label{obs_staged_comp}
    Fix a formula of the form $S(\ux_\sh\ux_\va; \tiluy) = S_\sh(\ux_\sh; \tiluy_\sh) \wedge S_\va(\ux_\va; \ut(\ux_\sh) + \tiluy_\va)$, where $S_\sh(\ux_\sh; \tiluy_\sh)$ and $S_\va(\ux_\va; \tiluy_\va)$ are parametrized $C$-sequence-systems, and $\ut(\ux_\sh)$ is a tuple of $\LKThe$-terms.
    Given a conjunction $\varphi(\uy)$ of $\LRC$-equations and a tuple $\umu{}(\uy) = \umu{}_\sh(\uy)\umu{}_\va(\uy)$ of $\LRC$-terms, the following are equivalent:
    \begin{enumerate}[(i)]
        \item The following two sentences hold in $\TKvsThe \cup \set{\text{``$\theta$ is $C$-image-complete''}}$:
        \begin{enumerate}[(a)]
            \item $\forall \uy \in \VV : \varphi(\uy) \rightarrow \text{``$\umu{}_\sh(\uy)$ is compatible with $S_\sh$''}$,
            \item $\forall \ux_\sh\uy \in \VV : \varphi(\uy) \wedge S_\sh(\ux_\sh; \umu{}_\sh(\uy)) \rightarrow \text{``$\ut(\ux_\sh) + \umu{}_\va(\uy)$ is compatible with $S_\va$''}$.
        \end{enumerate}
        \item $S(\ux_\sh\ux_\va; \tiluy)$ is a parametrized staged $C$-sequence-system, and the pair $(\varphi, \umu{})$ is compatible with $S$.
    \end{enumerate}
\begin{proof}
    The implication ``(ii) $\Rightarrow$ (i)'' is clear by the definition of a compatible tuple $\uu$ (Definition \ref{def_new_haram_comp}) and the definition of a compatible pair $(\varphi, \umu{})$ (Definition \ref{def_comp_pair_staged}).
    For the other implication, notice that, since any conjunction of $\LRC$-equations with $\uzero$ plugged in holds, (b) implies
    $$
    \TKvsThe \cup \set{\text{``$\theta$ is $C$-image-complete''}} \models \forall \ux_\sh \in \VV : S_\sh(\ux_\sh; \uzero) \rightarrow \text{``$\ut(\ux_\sh)$ is compatible with $S_\va$''}.
    $$
    With Lemma \ref{lemma_staged_c_ss_iff}, we immediately see that $S(\ux_\sh\ux_\va; \tiluy)$ is a parametrized staged $C$-sequence-system.
    Fix $(\VV, \theta) \models \TKvsThe \cup \set{\text{``$\theta$ is $C$-image-complete''}}$ and some $\uu \in \VV$ with $(\VV, \theta) \models \varphi(\uu)$.
    By the definition of compatibility for parametrized $C$-sequence-systems, it is clear that (a) implies that $\umu{}_\sh(\uu)$ satisfies point (a) in Lemma \ref{corollary_comP_fml_staged}.
    With (iii) of Lemma \ref{lemma_staged_c_ss_iff}, one can also easily verify that (b) implies that $\umu{}(\uu)$ satisfies point (b) in Lemma \ref{corollary_comP_fml_staged}.
    We conclude that $\umu{}(\uu)$ is compatible with $S$.
    This shows that $(\varphi, \umu{})$ is compatible with $S$.
\end{proof}
\end{lemma}

\noindent Unlike compatibility, the generalization of boundedness for parametrized staged $C$-sequence-systems is straightforward:

\begin{definition}
    \label{def_bound_staged_c_ss_new}
    Let $S(\ux_\sh\ux_\va; \uy) = S_\sh(\ux_\sh; \uy_\sh) \wedge S_\va(\ux_\va; \ut(\ux_\sh) + \uy_\va)$ be a parametrized staged $C$-sequence-system as in Definition \ref{def_new_haram_staged_c_ss}.
    \begin{enumerate}[(i)]
        \item If $\psi(\uxvec{}_\sh\uxvec{}_\va; \uw)$ is a formula, then we say $\psi(\uxvec{}_\sh\uxvec{}_\va; \uw)$ is \textbf{bounded} by $S$ if both the formulas $\psi_\sh(\uxvec{}_\sh;\uxvec{}_\va \uw) := \psi(\uxvec{}_\sh\uxvec{}_\va; \uw)$ and $\psi_\va(\uxvec{}_\va;\uxvec{}_\sh \uw) := \psi(\uxvec{}_\sh\uxvec{}_\va; \uw)$ are bounded by the param\-etrized $C$-sequence-systems $S_\sh$ and $S_\va$, respectively, as in Definition \ref{def_formual_bounded}.

        Writing $S(\ux_\sh\ux_\va; \uy) = \bigwedge\nolimits_{k=1}^{n_\sh} f_{\sh, k}^{q_{\sh, k}}[\theta](x_{\sh, \ld, k}) = y_{\sh, k} \wedge \bigwedge\nolimits_{k=1}^{n_\va} f_{\va, k}^{q_{\va, k}}[\theta](x_{\va, \ld, k}) = t_k(\ux_\sh) + y_{\va, k}$ in expanded form (see Notation \ref{not_expanded}), this is equivalent to the condition that the placeholder variables $x^i_{*, \ld, k}$ for $\theta^i(x_{*, \ld, k})$ may only appear in the formula $\psi(\uxvec{}_\sh\uxvec{}_\va; \uw)$ for $* \in \set{\sh, \va}$, $k \in \set{1, \dots, n_*}$, and $i \in \omega$ with $i < \deg(f^{q_{*, k}}_{*, k})$.
        This is also equivalent to the condition that $\theta^i(x_{*, \ld, k})$ may only appear in $\psi_\theta(\ux_\sh\ux_\va; \uw)$ for the respective $*$, $k$, and $i$.
        \item Given $\psi(\uxvec{}_\sh\uxvec{}_\va; \uw)$ as in (i), we say that $\psi_\theta(\ux_\sh\ux_\va; \uw)$ is \textbf{bounded} by $S$ if $\psi(\uxvec{}_\sh\uxvec{}_\va; \uw)$ is bounded by $S$.
        \item If $S'(\ux_\sh\ux_\va) := S(\ux_\sh\ux_\va; \uu)$ is a staged $C$-sequence-system over some $(\VV, \theta)$, then we say that a formula is \textbf{bounded} by $S'$ if it is bounded by the underlying parametrized staged $C$-sequence-system $S$.
    \end{enumerate}
\end{definition}

\noindent Recall our characterization of existentially closed models: A  model $(\mm, \theta) \models T^C_\theta$ is existentially closed if and only if it is $C$-image-complete and $\exists \ux \in \VV : \psi_\theta(\ux) \wedge S(\ux)$ holds for $C$-sequence-systems $S(\ux)$ over $(\VV, \theta)$ and $L(M)$-formulas $\psi(\uxvec{})$ that are bounded by $S$ and imply no finite disjunction of non-trivial linear dependencies in $\uxvec{}$ over $\VV$. When we proved the implication from left to right in Section 4.3 in \cite{Chi25}, we actually proved a stronger version that also works for so-called generalized $C$-sequence-systems, which are defined as follows:

\begin{definition} \label{def_param_gen_c_ss}
    Let $\ux := \ux_\li \ux_\ld := (x_{\lii, k} : 1 \leq k \leq m)(x_{\ldd, k} : 1 \leq k \leq n)$ and $(\VV, \theta) \models \TKvsThe$ be given.
    A \textbf{generalized $\mathbf{C}$-sequence-system} over $(\VV, \theta)$ is an $\LKThe(\VV)$-formula of the form
    \begin{align*}
        S(\ux) = & \bigwedge\nolimits_{k=1}^n \xi_k[\theta](x_{\ldd, k}) = \sum\nolimits_{l=1}^m P_{k, l}[\theta](x_{\lii, l}) + \sum\nolimits_{l=1}^{k-1} Q_{k, l}[\theta](x_{\ldd, l}) + u_{k}
    \end{align*}
    that satisfies the following conditions:
    \begin{enumerate}[(i)]
        \item If $C$ is algebraic, we require $m = 0$.
        \item Each $\xi_k$ is a monic polynomial that either satisfies $\Fac(\xi_k) \subseteq \Kp{C=\infty}$ or is of the form $f^q$ for some $f \in \Kp{0<C<\infty}$ and some $q$ with $0 < q \leq C(f)$.
        \item Given any $k$ with $\xi_k = f^q$ for some $f \in \Kp{0<C<\infty}$ and $q \in \NN$ with $0 < q \leq C(f)$, the following holds:
        $$
        \TKvsThe \cup \Diag(\VV, \theta) \models \forall \ux : S(\ux) \rightarrow f^C[\theta](x_{\ldd, k}) = 0.
        $$
        Here $\Diag(\VV, \theta)$ is the \textbf{diagram} of $(\VV, \theta)$, which is the set of all quantifier-free $\LKThe(\VV)$-sentences that hold in $(\VV, \theta)$.
    \end{enumerate}
\end{definition}

\noindent We will now use this stronger version of our characterization of existentially closed models to prove the following lemma, simply by showing that the staged $C$-sequence-systems introduced in this section are essentially a special form of generalized ones:

\begin{lemma} \label{theorem_stronger_left_to_right_new}
    Let $(\mm, \theta) \models T^C_\theta$ be existentially closed.
    Then
    $$
    (\mm, \theta) \models \exists \ux_\sh\ux_\va \in \VV : \psi_\theta(\ux_\sh\ux_\va) \wedge S(\ux_\sh\ux_\va)
    $$
    holds for every staged $C$-sequence-system $S(\ux_\sh\ux_\va)$ and every $L(M)$-formula $\psi(\uxvec{}_\sh\uxvec{}_\va)$ that is bounded by $S$ and implies no finite disjunction of non-trivial linear dependencies in $\uxvec{}_\sh\uxvec{}_\va$ over $\VV$.
\begin{proof}
    Write $S(\ux_\sh\ux_\va) = \bigwedge\nolimits_{k=1}^{n_\sh} f_{\sh, k}^{q_{\sh, k}}[\theta](x_{\sh, \ld, k}) = u_{\sh, k} \wedge \bigwedge\nolimits_{k=1}^{n_\va} f_{\va, k}^{q_{\va, k}}[\theta](x_{\va, \ld, k}) = t_k(\ux_\sh) + u_{\va, k}$ as a parametrized staged $C$-sequence-system in expanded form with some compatible parameters $\uu \in \VV$ plugged in.
    Rearranging $\ux_\sh\ux_\va$ into $\ux = \ux_\li\ux_\ld$, where $\ux_\li = \ux_{\sh,\li}\ux_{\va, \li}$ and $\ux_\ld = \ux_{\sh,\ld}\ux_{\va, \ld}$, one sees that $S(\ux_\sh\ux_\va)$ is a generalized $C$-sequence-system over $(\VV, \theta)$, as defined above in Definition \ref{def_param_gen_c_ss}.
    Points (i) and (ii) of Definition \ref{def_param_gen_c_ss} are easily verified.
    Point (iii) of Definition \ref{def_param_gen_c_ss} follows since we have
    $$
    \TKvsThe \cup \Diag(\VV, \theta) \models \forall \ux_\sh\ux_\va : S(\ux_\sh\ux_\va) \rightarrow f^C_{\sh, k}[\theta](x_{\sh, \ld, k}) = f^{C-q_{\sh, k}}_{\sh, k}[\theta](u_{\sh, k}) = 0,
    $$
    for all $k \in \set{1, \dots, n_\sh}$ with $C(f_{\sh, k}) < \infty$ (see (a) of Lemma \ref{corollary_comP_fml_staged}), and
    $$
    \TKvsThe \cup \Diag(\VV, \theta) \models \forall \ux_\sh\ux_\va : S(\ux_\sh\ux_\va) \!\rightarrow\! f^C_{\va, k}[\theta](x_{\va, \ld, k}) \!=\! f^{C-q_{\va, k}}_{\va, k}[\theta](t_k(\ux_\sh) + u_{\va, k}) \!=\! s_k(\uu) \!=\! 0,
    $$
    for all $k \in \set{1, \dots, n_\va}$ with $C(f_{\va, k}) < \infty$ (see (iii) of Lemma \ref{lemma_staged_c_ss_iff} and (b) of Lemma \ref{corollary_comP_fml_staged}).
    One can also easily verify that the definition of $\psi(\uxvec{}_\sh\uxvec{}_\va)$ being bounded by $S(\ux_\sh\ux_\va)$, as in Definition \ref{def_bound_staged_c_ss_new} above, coincides with the definition of $\psi(\uxvec{}_\sh\uxvec{}_\va)$ being bounded by $S(\ux_\sh\ux_\va)$ viewed as a generalized $C$-sequence-system, i.e., Definition 4.16 in \cite{Chi25}.
    Now, as $\psi(\uxvec{}_\sh\uxvec{}_\va)$ implies no finite disjunction of non-trivial linear dependencies in $\uxvec{}_\sh\uxvec{}_\va$ over $\VV$, we can apply Theorem 4.18 in \cite{Chi25} to conclude.
\end{proof}
\end{lemma}

\noindent The goal for the rest of this section is to state and prove a version of the \TrafoLemma{} for parametrized staged $C$-sequence-systems. Recall that our \TrafoLemma{} (Lemma \ref{lemma_trafoo}) states that a conjunction of the form $S(\ux; \umu{}_1(\uy)) \wedge E(\ux; \umu{}_2(\uy))$ is transformable into a formula of the form $\varphi'(\uy) \wedge S'(\ux'; \umu{}'(\uy))$, where $(\varphi', \umu{}')$ is compatible with $S'$, i.e., $\varphi'(\uy)$ implies that $\umu{}'(\uy)$ is compatible with $S'$.
A crucial part of the \TrafoLemma{} is that, if $E(\ux; \uy_2)$ is non-trivial in $\ux$ and bounded by $S$, then $(\rk(S'), \deg(S')) <_\Lex (\rk(S), \deg(S))$ holds.
In the case of parametrized staged $C$-sequence-systems, we use the following instead of the rank and degree:

\begin{definition} \label{def_ranks_and_degs}
    Let $S(\ux_\sh\ux_\va; \uy) = S_\sh(\ux_\sh; \uy_\sh) \wedge S_\va(\ux_\va; \ut(\ux_\sh) + \uy_\va)$ be a parametrized staged $C$-sequence-system with everything as in Definition \ref{def_new_haram_staged_c_ss}.
    We define
    \begin{multicols}{2}
    \begin{enumerate}[(i)]
        \item $\rk_\va(S) := \rk(S_\va)$,
        \item $\deg_\va(S) := \deg(S_\va)$,
        \item $\rk_\sh(S) := \rk(S_\sh)$,
        \item $\deg_\sh(S) := \deg(S_\sh)$.
    \end{enumerate}
    \end{multicols}
\end{definition}

\noindent Similarly to ``regular'' parametrized $C$-sequence-systems, we can now ``order'' parametrized staged $C$-sequence-systems using $(\rk_\va(S), \deg_\va(S), \rk_\sh(S), \deg_\sh(S))$ with the lexicographic order, where the first entry is the most significant, i.e., $(0, 0, 0, 1) <_\Lex (1, 0, 0, 0)$.
As with ``regular'' parametrized $C$-sequence-systems, we will also make use of the fact that $(\NN^4, <_\Lex)$ is a well-ordered set.

\begin{lemma}[Transformation Lemma for staged $C$-sequence-systems] \label{lemma_trafo_for_staged_c_ss}
    Let $S(\ux_\sh\ux_\va; \uy_1)$ be a parame\-trized staged $C$-sequence-system, $E(\ux_\sh\ux_\va; \uy_2)$ a conjunction of $\LKThe$-equations that is non-trivial in $\ux_\sh\ux_\va$ and bounded by $S$, $\umu{}_1(\uy)$ and $\umu{}_2(\uy)$ two tuples of $\LRC$-terms, and $\varphi(\uy)$ a conjunction of $\LRC$-equations such that $(\varphi, \umu{}_1)$ is compatible with $S$.
    Then the conjunction $\varphi(\uy) \wedge S(\ux_\sh\ux_\va; \umu{}_1(\uy)) \wedge E(\ux_\sh\ux_\va; \umu{}_2(\uy))$ is transformable into a formula of the form
    $$
    \varphi'(\uy) \wedge S'(\ux'_\sh\ux'_\va; \umu{}'(\uy))
    $$
    where $S'(\ux'_\sh\ux'_\va; \uy')$ is another parametrized staged $C$-sequence-system, and $(\varphi', \umu{}')$ is a pair compatible with $S'$.
    Furthermore:
    \begin{enumerate}[(i)]
        \item There are witnesses $\underline{\nu}{}(\ux_\sh\ux_\va)$ and $\utau(\ux'_\sh\ux'_\va; \uy)$ of this transformability such that the subtuple $\underline{\nu}{}_\sh(\ux_\sh)$ of $\underline{\nu}{}(\ux_\sh\ux_\va)$ that corresponds to the subtuple $\ux'_\sh \subseteq \ux'$ depends only on $\ux_\sh$, and the subtuple $\utau{}_\sh(\ux'_\sh\ux'_\va; \uy)$ of $\utau(\ux'_\sh\ux'_\va; \uy)$ that corresponds to $\ux_\sh \subseteq \ux$ depends only on $\ux'_\sh\uy$.
        \item If $E(\ux_\sh\ux_\va; \uy_2)$ is non-trivial in $\ux_\va$, then $(\rk_\va(S'), \deg_\va(S')) <_\Lex (\rk_\va(S), \deg_\va(S))$ holds.
        If the conjunction $E(\ux_\sh\ux_\va; \uy_2)$ is trivial in $\ux_\va$, then both $(\rk_\va(S'), \deg_\va(S')) = (\rk_\va(S), \deg_\va(S))$ and $(\rk_\sh(S'), \deg_\sh(S')) <_\Lex (\rk_\sh(S), \deg_\sh(S))$ hold.
    \end{enumerate}
\begin{proof}
    Throughout the proof, we assume that $S(\ux_\sh\ux_\va; \uy_1) = S_\sh(\ux_\sh; \uy_{1, \sh}) \wedge S_\va(\ux_\va; \ut(\ux_\sh) + \uy_{1, \va})$ and that $\umu{}_1(\uy) = \umu{}_{1, \sh}(\uy)\umu{}_{1, \va}(\uy)$.
    We may write $\ux := \ux_\sh\ux_\va$ and $\ux' := \ux'_\sh\ux'_\va$ when convenient.
    We also work in $\TKvsThe \cup \set{\text{``$\theta$ is $C$-image-complete''}}$.
    Thus, when we say that a formula implies another one, we mean that this implication holds modulo $\TKvsThe \cup \set{\text{``$\theta$ is $C$-image-complete''}}$.
    We start by showing that $\varphi(\uy) \wedge S(\ux_\sh\ux_\va; \umu{}_1(\uy)) \wedge E(\ux_\sh\ux_\va; \umu{}_2(\uy))$ is transformable into a formula $\varphi'(\uy) \wedge S'(\ux'_\sh\ux'_\va; \umu{}'(\uy))$, as described in the statement.
    For this, we distinguish between the case where $E(\ux_\sh\ux_\va; \uy_2)$ is trivial in $\ux_\va$ and the case where it is non-trivial in $\ux_\va$.

    \begin{subclaim}
        The lemma holds if $E(\ux_\sh\ux_\va; \uy_2) =: E'(\ux_\sh; \uy_2)$ is trivial in $\ux_\va$.
    \begin{innerproof}
        In this case, we use the \TrafoLemma{} to see that
        \begin{align}
            S_\sh(\ux_\sh; \umu{}_{1,\sh}(\uy)) \wedge (\varphi(\uy) \wedge E'(\ux_\sh; \umu{}_2(\uy))) \quad \text{is transformable into}\quad \varphi'(\uy) \wedge S'_\sh(\ux'_\sh; \umu{}'_\sh(\uy)) \label{tag_trafo_for_imply}
        \end{align}
        where $S'_\sh(\ux'_\sh; \uy'_\sh)$ is another parametrized $C$-sequence-system, and $(\varphi', \umu{}'_\sh)$ is compatible with $S'_\sh$.
        Let $\underline{\nu}{}_\sh(\ux_\sh)$ and $\utau{}_\sh(\ux'_\sh; \uy)$ witness this transformability.
        Recall that each entry of $\utau{}_\sh(\ux'_\sh; \uy)$ is, by the definition of transformability, the sum of an $\LKThe$-term in $\ux'_\sh$ and an $\LRC$-term in $\uy$.
        Hence, we can find a tuple of $\LKThe$-terms $\ut'(\ux'_\sh)$ and a tuple of $\LRC$-terms $\umu{}'_\va(\uy)$ such that
        $$
        \ut'(\ux'_\sh) + \umu{}'_\va(\uy) = \ut(\utau{}_\sh(\ux'_\sh; \uy)) + \umu{}_{1, \va}(\uy).
        $$
        Let $\ux'_\va$ and $\uy'_\va$ be tuples of the same lengths as $\ux_\va$ and $\uy_{1, \va}$.
        Setting $\underline{\nu}{}_\va(\ux_\va) := \ux_\va$ and $\utau{}_\va(\ux'_\va; \uy) := \ux'_\va$, (i) of Observation \ref{obser_trafo_nice_2} yields that $\underline{\nu}{}_\sh(\ux_\sh)\underline{\nu}_\va(\ux_\va)$ and $\utau{}_\sh(\ux'_\sh; \uy)\utau{}_\va(\ux'_\va; \uy)$ witness that the conjunction $\varphi(\uy) \wedge S(\ux_\sh\ux_\va; \umu{}_1(\uy)) \wedge E(\ux_\sh\ux_\va; \umu{}_2(\uy))$ is transformable into
        $$
        \varphi'(\uy) \wedge S'_\sh(\ux'_\sh; \umu{}'_\sh(\uy)) \wedge S_\va(\ux'_\va; \ut'(\ux'_\sh) + \umu{}'_\va(\uy)).
        $$
        By the definition of compatibility, $\varphi'(\uy)$ implies that $\umu{}'_\sh(\uy)$ is compatible with $S'_\sh$.
        By (ii) of the definition of transformability (see Definition \ref{def_transfo}) and (\ref{tag_trafo_for_imply}), the conjunction $\varphi'(\uy) \wedge S'_\sh(\ux'_\sh; \umu{}'_\sh(\uy))$ implies $\varphi(\uy) \wedge S_\sh(\utau{}_\sh(\ux'_\sh; \uy); \umu{}_{1,\sh}(\uy))$.
        Since $(\varphi, \umu{}_1)$ is compatible with $S$, Lemma \ref{obs_staged_comp} yields that the conjunction $\varphi'(\uy) \wedge S'_\sh(\ux'_\sh; \umu{}'_\sh(\uy))$ also implies that $\ut(\utau{}_\sh(\ux'_\sh; \uy)) + \umu{}_{1, \va}(\uy) = \ut'(\ux'_\sh) + \umu{}'_{\va}(\uy)$ is compatible with $S_\va$.
        Set $\umu{}'(\uy) := \umu{}'_\sh(\uy)\umu{}'_\va(\uy)$.
        Now, by Lemma \ref{obs_staged_comp}, $S'(\ux'; \uy') := S'_\sh(\ux'_\sh; \uy'_\sh) \wedge S_\va(\ux'_\va; \ut'(\ux'_\sh) + \uy'_{\va})$ is a parametrized staged $C$-sequence-system with which the pair $(\varphi', \umu{}')$ is compatible.

        It is clear that our tuples $\underline{\nu}(\ux_\sh\ux_\va) := \underline{\nu}{}_\sh(\ux_\sh)\underline{\nu}_\va(\ux_\va)$ and $\utau{}(\ux'_\sh\ux'_\va; \uy) := \utau{}_\sh(\ux'_\sh; \uy)\utau{}_\va(\ux'_\va; \uy)$, which witness the transformability, are as described in (i).
        Since $E(\ux; \uy_2)$ is non-trivial and bounded by $S$, $E'(\ux_\sh; \uy_2)$ is non-trivial and bounded by $S_\sh$, so our \TrafoLemma{} yields the lexicographical inequality $(\rk(S'_\sh), \deg(S'_\sh)) <_\Lex (\rk(S_\sh), \deg(S_\sh))$.
        Since $(\rk_\va(S'), \deg_\va(S')) = (\rk_\va(S), \deg_\va(S))$ also holds, (ii) follows immediately.
    \end{innerproof}
    \end{subclaim}
    \begin{subclaim}
        The lemma also holds if $E(\ux; \uy_2) =: E'(\ux_\va; \ux_\sh\uy_2)$ is non-trivial in $\ux_\va$.
    \begin{innerproof}
        In this case, our \TrafoLemma{} (with $\ux_\sh\uy$ taking the role of $\uy$) shows that the conjunction $S_\va(\ux_\va; \ut(\ux_\sh) + \umu{}_{1, \va}(\uy)) \wedge (\varphi(\uy) \wedge E'(\ux_\va; \ux_\sh\umu{}_2(\uy)))$ is transformable into a formula of the form
    $$
    \varphi_*(\ux_\sh\uy) \wedge S'_\va(\ux'_\va; \umu{}_*(\ux_\sh\uy))
    $$
    where $S'_\va(\ux'_\va; \uy'_\va)$ is a parametrized $C$-sequence-system and $(\varphi_*, \umu{}_*)$ is a pair compatible with $S'_\va$.
    Let $\underline{\nu}{}_*(\ux_\va)$ and $\utau{}_*(\ux'_\va; \ux_\sh\uy)$ witness this transformability.
    Now $\varphi_*(\ux_\sh\uy) \wedge S_\sh(\ux_\sh; \umu{}_{1,\sh}(\uy))$ is a conjunction of $\LRC$-equations and therefore, by Corollary \ref{corollary_trafo_any_lrc}, transformable into a formula of the form
    $$
    \varphi'(\uy) \wedge S'_\sh(\ux'_\sh; \umu{}'_\sh(\uy))
    $$
    where $S'_\sh(\ux'_\sh; \uy'_\sh)$ is a parametrized $C$-sequence-system and $(\varphi', \umu{}'_\sh)$ is a pair compatible with $S'_\sh$.
    Let $\underline{\nu}_{\sh}(\ux_\sh)$ and $\utau{}_\sh(\ux'_\sh; \uy)$ witness this transformability.
    By the final assertion of Corollary \ref{corollary_trafo_any_lrc}, we can choose $S'_\sh$ and $(\varphi', \umu{}'_\sh)$ in such a way that, modulo $\varphi'(\uy) \wedge S'_\sh(\ux'_\sh; \umu{}'_\sh(\uy))$, $\umu{}_*(\utau{}_\sh(\ux'_\sh; \uy)\uy) =: \ut'(\ux'_\sh) + \umu{}'_\va(\uy)$ and $\utau{}_*(\ux'_\va;\utau{}_\sh(\ux'_\sh; \uy)\uy) =: \utau{}_\va(\ux'_\sh\ux'_\va; \uy)$ can both be written as the sum of a tuple of $\LKThe$-terms in the respective $\ux'$-variables and a tuple of $\LRC$-terms in $\uy$. We furthermore define $\underline{\nu}{}_\va(\ux_\va) := \underline{\nu}{}_*(\ux_\va)$.
    With (ii) of Observation \ref{obser_trafo_nice_2}, we see that $\varphi(\uy) \wedge S_\sh(\ux_\sh; \umu{}_{1, \sh}(\uy)) \wedge S_\va(\ux_\va; \ut(\ux_\sh) + \umu{}_{1, \va}(\uy)) \wedge E(\ux; \umu{}_2(\uy))$ is transformable into
    $$
    \varphi'(\uy) \wedge S'_\sh(\ux'_\sh; \umu{}'_\sh(\uy)) \wedge S'_\va(\ux'_\va; \ut'(\ux'_\sh) + \umu{}'_\va(\uy))
    $$
    witnessed by $\underline{\nu}{}_\sh(\ux_\sh)\underline{\nu}{}_\va(\ux_\va)$ and $\utau{}_\sh(\ux'_\sh; \uy)\utau{}_\va(\ux'_\sh\ux'_\va; \uy)$.
    Since $(\varphi', \umu{}'_\sh)$ is compatible with $S'_\sh$, $\varphi'(\uy)$ implies that $\umu{}'_\sh(\uy)$ is compatible with $S'_\sh$.
    Since $\varphi'(\uy) \wedge S'_\sh(\ux'_\sh; \umu{}'_\sh(\uy))$ implies the formula $\varphi_*(\utau{}_\sh(\ux'_\sh; \uy)\uy)$ by the second implication in the definition of transformability, and $(\varphi_*, \umu{}_*)$ is compatible with $S'_\va$, we see that $\varphi'(\uy) \wedge S'_\sh(\ux'_\sh; \umu{}'_\sh(\uy))$ implies that the tuple of $\LRC$-terms $\umu{}_*(\utau{}_\sh(\ux'_\sh; \uy)\uy) = \ut'(\ux'_\sh) + \umu{}'_\va(\uy)$ is compatible with $S'_\va$.
    Set $\umu{}'(\uy) := \umu{}'_\sh(\uy)\umu{}'_\va(\uy)$.
    By Lemma \ref{obs_staged_comp}, $S'(\ux'; \uy') := S'_\sh(\ux'_\sh; \uy'_\sh) \wedge S'_\va(\ux'_\va; \ut'(\ux'_\sh) + \uy'_\va)$ is a parametrized staged $C$-sequence-system with which $(\varphi', \umu{}')$ is compatible.

    It is again clear that (i) holds.
    Since $E(\ux; \uy_2) = E'(\ux_\va; \ux_\sh\uy_2)$ is non-trivial in $\ux_\va$ and bounded by $S_\va$, the \TrafoLemma{} yields $(\rk(S'_\va), \deg(S'_\va)) <_\Lex (\rk(S_\va), \deg(S_\va))$, which means that (ii) holds as well.
    \end{innerproof}
    \end{subclaim}

    \noindent This completes the proof.
\end{proof}
\end{lemma}

\noindent As mentioned earlier in this section, if $E(\ux; \uy)$ is transformable into $E'(\ux'; \uy)$, witnessed by some $\underline{\nu}(\ux)$ and $\utau{}(\ux'; \uy)$, then
$$
\exists \ux \in \VV : \phi(\ux; \uw) \wedge E(\ux; \uy) \equiv \exists \ux' \in \VV : \phi(\utau(\ux'; \uy); \uw) \wedge E'(\ux'; \uy)
$$
holds for any formula $\phi(\ux; \uw)$.
If $E'(\ux'; \uy)$ is a parametrized staged $C$-sequence-system, we can use the following lemma to replace $\phi(\utau(\ux'; \uy); \uw)$ with a formula that is bounded by this staged system:

\begin{lemma} \label{lemma_main_trafo_for_staged}
    Let $S(\ux_\sh\ux_\va; \tiluy{})$ be a parametrized staged $C$-sequence-system, let $\umu{}(\uy)$ be a tuple of $\LRC$-terms, and let $\utau{}(\ux_\sh\ux_\va; \uy)$ be an $|\ux|$-tuple of $\LRC$-terms.
    Assume that $\utau{}(\ux_\sh\ux_\va; \uy)$ can be written entrywise as the sum of an $\LKThe$-term in $\ux_\sh\ux_\va$ and an $\LRC$-term in $\uy$.
    Given any $L$-formula $\psi(\uxvec{}; \uw)$, there are an $L$-formula $\psi'(\uxvec{}_\sh\uxvec{}_\va; \uw\tiluw)$ bounded by $S$ and a finite tuple of $\LRC$-terms $\ut(\uy)$ such that
    $$
    \psi_\theta(\utau{}(\ux_\sh\ux_\va; \uy); \uw) \quad \equiv \quad \psi'_\theta(\ux_\sh\ux_\va; \uw\ut(\uy))
    $$
    holds modulo $T_\theta \cup \set{\text{``$\theta$ is $C$-image-complete''}} \cup \set{S(\ux_\sh\ux_\va; \umu(\uy))}$.
\begin{proof}
    Write $\ux = (x_1, \dots, x_n)$.
    Using the equations in $S(\ux_\sh\ux_\va; \uy)$ and Fact \ref{lemma_bound_term_light}, we obtain, modulo $T_\theta \cup \set{\text{``$\theta$ is $C$-image-complete''}} \cup \set{S(\ux_\sh\ux_\va; \umu(\uy))}$,
    \begin{align}
        \theta^i(\tau_k(\ux_\sh\ux_\va; \uy)) &= \sum\nolimits_{l=1}^{m_\sh} \rho_{\sh, \li, i, k, l}[\theta](x_{\sh, \li, l}) + \sum\nolimits_{l=1}^{n_\sh} \rho_{\sh, \ld, i, k, l}[\theta](x_{\sh, \ld, l}) \notag \\&\quad+ \sum\nolimits_{l=1}^{m_\va} \rho_{\va, \li, i, k, l}[\theta](x_{\va, \li, l}) + \sum\nolimits_{l=1}^{n_\va} \rho_{\va, \ld, i, k, l}[\theta](x_{\va, \ld, l}) + t_{i, k}(\uy) \label{tag_stupid_eqqqq}
    \end{align}
    for some polynomials with $\deg(\rho_{\sh, \ld, i, k, l}) < \deg(f_{\sh, l}^{q_{\sh, l}})$ and $\deg(\rho_{\va, \ld, i, k, l}) < \deg(f_{\va, l}^{q_{\va, l}})$, and for some $\LRC$-term $t_{i, k}(\uy)$.
    Denote the right-hand side of (\ref{tag_stupid_eqqqq}) by $s_{k}^i(\ux_\sh\ux_\va; \uy)$.
    We define $\psi'(\uxvec{}_\sh\uxvec{}_\va; \uw\tiluw)$ and $\ut(\uy)$ such that the formula $\psi'_\theta(\ux_\sh\ux_\va; \uw\ut(\uy))$ is $\psi_\theta(\ux; \uw)$ with every occurrence of $\theta^i(x_k)$ replaced by $s_k^i(\ux_\sh\ux_\va; \uy)$.
    With this and the definition of $\psi_\theta(\ux; \uw)$, it is clear that the desired equivalence holds, and it is also easy to check that $\psi'(\uxvec{}_\sh\uxvec{}_\va; \uw\tiluw)$ is bounded by $S$. For details, see the proof of Lemma 4.6 in \cite{Chi25}.
\end{proof}
\end{lemma}

\noindent We mention one last construction that is used in the next section:

\begin{observation} \label{obs_add_to_shared}
    Let $S(\ux_\sh\ux_\va; \uy) = S_\sh(\ux_\sh; \uy_\sh) \wedge S_\va(\ux_\va; \ut(\ux_\sh) + \uy_\va)$ be a parametrized staged $C$-sequence-system, and let $S_0(\ux_0; \uy_0)$ be a parametrized $C$-sequence-system.
    Then, after rearranging the conjunction $S_0(\ux_0; \uy_0) \wedge S_\sh(\ux_\sh; \uy_\sh)$ into a ``regular'' parametrized $C$-sequence-system (see Observation 4.9 in \cite{Chi25}), the formula
    $$
    S'(\ux'_\sh\ux_\va; \uy') := \underbrace{(S_0(\ux_0; \uy_0) \wedge S_\sh(\ux_\sh; \uy_\sh))}_{=: S'_\sh(\ux'_\sh; \uy'_\sh)} \wedge S_\va(\ux_\va; \ut(\ux_\sh) + \uy_\va)
    $$
    is another parametrized staged $C$-sequence-system with $(\rk_\va(S'), \deg_\va(S')) = (\rk_\va(S), \deg_\va(S))$.
    Moreover, if $\uu$ is compatible with $S$ and $\uu_0$ is compatible with $S_0$, then $\uu_0\uu$ (or rather the corresponding rearranged tuple) is compatible with $S'$.
\end{observation}

\subsection{Step (I) for general kernel configurations} \label{sec_reduc_natp_to_T}

\noindent The goal of this section is to generalize Lemma \ref{lemma_atp_step_11_ez} to arbitrary kernel configurations $C$:

\begin{lemma} \label{lemma_atp_step_11}
    If the theory $T\theta^C$ has \ATP{}, then there is an $L$-formula $\psi(\ux_\sh\ux_\va; \uz; \uw)$ and a strongly $L$-indiscernible tree $(\ud_\nu : \nu \in 2^{<\omega})$ in some $\mm \models T$ such that
    \begin{enumerate}[(i)]
        \item For every antichain $X \subset 2^{<\omega}$, the type $\set{\psi(\ux_\sh\ux_\nu; \uz; \ud_\nu) : \nu \in X}$ implies no finite disjunction of non-trivial linear dependencies in $\ux_\sh(\ux_\nu : \nu \in X)$ over $\VV$.
        \item The type $\set{\psi(\ux_\sh\ux_\emptyseq; \uz; \ud_\emptyseq), \psi(\ux_\sh\ux_{\seq{0}}; \uz; \ud_{\seq{0}})}$ implies some finite disjunction of non-trivial linear dependencies in $\ux_\sh\ux_\emptyseq\ux_{\seq{0}}$ over $\VV$.
    \end{enumerate}
\end{lemma}

\noindent Together, Lemma \ref{lemma_atp_step_11} and Lemma \ref{lemma_has_li_atp_impl_atp} immediately yield a proof of Theorem \ref{theorem_preserve_natp}, the main theorem of this paper.

\begin{proof}[Proof of Lemma \ref{lemma_atp_step_11}]
    Throughout this proof, we work in a sufficiently large monster model $(\MM, \theta) \models T\theta^C$ (or more precisely, a monster model of a completion of $T\theta^C$ that has \ATP{}). We start with a definition:

\begin{subdefinition} \label{def_special_witness_3}
    We say that a pair $(\psi, S)$, where $S(\ux_\sh\ux_\va;\!\uy)$ is a parametrized staged $C$-sequence-system and $\psi(\uxvec{}_\sh\uxvec{}_\va;\! \uz;\! \uw)$ is an $L$-formula bounded by $S$, is a \textbf{special witness} if
    $$
    \ux_\sh \in \VV \wedge \exists \ux_\va \in \VV : \delta_S(\uy) \wedge \psi_\theta(\ux_\sh\ux_\va; \uz; \uw) \wedge S(\ux_\sh\ux_\va; \uy)
    $$
    witnesses \ATP{} in $\ux_\sh\uz$, where $\delta_S(\uy) := \text{``$\uy$ is compatible with $S$''}$ is the formula from Lemma \ref{corollary_comP_fml_staged}.
    We say that a pair $(\psi, S)$ is a special witness \textbf{together} with a tree $(\uu_\nu\ud_\nu : \nu \in \kappa^{<\omega})$ if the formula above witnesses \ATP{} with that tree.
\end{subdefinition}

\begin{subobservation} \label{remark_str_ic_witness}
    By Lemma \ref{lemma_weak_modeling}, every special witness $(\psi, S)$ witnesses \ATP{} together with a strongly $L_\theta$-indiscernible tree $(\uu_\nu\ud_\nu : \nu \in \omega^{<\omega})$.
\end{subobservation}

\begin{subclaim} \label{lemma_ATP_if_special_witness_ATP}
    $T\theta^C$ has \ATP{} if and only if it has a special witness.
\begin{innerproof}
    If $T\theta^C$ has a special witness, then it has \ATP{} by definition.
    Conversely, assume that $T\theta^C$ has \ATP{} witnessed by $\phi(\uz; \uw)$.
    By Fact \ref{theorem_big_fml_preceise} and Fact \ref{fact_atp_disj}, we may assume that
    $$
    \phi(\uz; \uw) = \exists \uy \in Y_{\uw} : \exists \ux \in \VV : \psi_{\theta}(\ux; \uz; \uy\uw) \wedge S(\ux; \pi(\uy))
    $$
    where $S(\ux; \uy)$ is a parametrized $C$-sequence-system, $\pi$ is a coordinate projection, $Y$ is an algebraic pattern such that $(Y, \pi)$ is compatible with $S$ (see Definition \ref{def_alg_pat_comp}), and $\psi(\uxvec; \uz; \uy\uw)$ is an $L$-formula bounded by $S$.
    Let $(\ud_\nu : \nu \in 2^{<\omega})$ be a tree with which $\phi(\uz; \uw)$ witnesses \ATP{}.
    Note that $Y_{\ud_{\nu}}$ is a finite set whose cardinality is bounded by some $q \in \NN$.
    By Lemma \ref{lemma_remove_alg_set}, the formula
    $$
    \exists \ux \in \VV : \psi_{\theta}(\ux; \uz; \uy\uw) \wedge S(\ux; \pi(\uy))
    $$
    must witness \ATP{} together with a strongly $L_\theta$-indiscernible tree $(\uu'_{\nu}\ud'_{\nu} : \nu \in 2^{<\omega})$ such that $\uu'_{\nu} \in Y_{\ud'_{\nu}}$ holds for all $\nu \in 2^{<\omega}$.
    Since $(Y, \pi)$ is compatible with $S$, this implies that $\pi(\uu'_{\nu})$ is compatible with $S$ for every $\nu \in 2^{<\omega}$.

    Now set $\uw_* := \uy\uw$, $\uy_* := \pi(\uy)$, $\psi_*(\uxvec; \uz; \uw_*) := \psi(\uxvec; \uz; \uy\uw)$, $\ud_{*,\nu} := \uu'_{\nu}\ud'_{\nu}$, and $\uu_{*, \nu} := \pi(\uu'_{\nu})$.
    Clearly, $(\uu_{*, \nu}\ud_{*, \nu} : \nu \in 2^{<\omega})$ is strongly $L_\theta$-indiscernible, every $\uu_{*, \nu}$ is compatible with $S$, and $\psi_*(\uxvec; \uz; \uw_*)$ is bounded by $S$.
    Setting $\ux_\va := \ux$ and defining $\ux_\sh$ as the empty tuple, we see that
    $$
    \ux_\sh \in \VV \wedge \exists \ux_\va \in \VV : \delta_{S}(\uy_*) \wedge \psi_{*,\theta}(\ux_\sh\ux_\va; \uz; \uw_*) \wedge S(\ux_\sh\ux_\va; \uy_*)
    $$
    witnesses \ATP{} together with $(\uu_{*, \nu}\ud_{*, \nu} : \nu \in 2^{<\omega})$.
    Since $\ux_\sh$ is empty and $S(\ux; \uy)$ is a parametrized $C$-sequence-system that bounds $\psi(\uxvec; \uz; \uy\uw)$, it is easy to check that $S(\ux_\sh\ux_\va; \uy)$ is a parametrized staged $C$-sequence-system (Definition \ref{def_new_haram_staged_c_ss}) that bounds $\psi_*(\uxvec{}_\sh\uxvec{}_\va; \uz; \uw_*)$ (Definition \ref{def_bound_staged_c_ss_new}) and that every $\uu_{*, \nu}$ is still compatible with it (Definition \ref{def_new_haram_comp}).
    We conclude that $(\psi_*, S)$ is a special witness.
\end{innerproof}
\end{subclaim}

\begin{subobservation} \label{observ_delta}
    Let $(\psi, S)$ be a special witness together with a tree $(\uu_\nu\ud_\nu : \nu \in \kappa^{<\omega})$.
    Then $(\MM, \theta) \models \delta_S(\uu_\nu)$ holds for every $\nu \in \kappa^{<\omega}$.
\end{subobservation}

\begin{subclaim} \label{lemma_implies_finite_disj}
    Assume that $(\psi, S)$ is a special witness together with a strongly $L_\theta$-indis\-cernible tree $(\uu_\nu\ud_\nu : \nu \in 2^{<\omega})$.
    Then the type $\set{\psi(\uxvec{}_\sh\uxvec{}_\emptyseq; \uz; \ud_\emptyseq), \psi(\uxvec{}_\sh\uxvec{}_{\seq{0}}; \uz; \ud_{\seq{0}})}$ implies some finite disjunction of non-trivial linear dependencies in $\uxvec{}_\sh\uxvec{}_\emptyseq\uxvec{}_{\seq{0}}$ over $\VV$.
\begin{innerproof}
    Define
    $
    S'(\ux_\sh\ux_\emptyseq\ux_{\seq{0}}) := S_\sh(\ux_\sh; \uu_{\sh, \emptyseq}) \wedge S_\va(\ux_\emptyseq; \ut(\ux_\sh) + \uu_{\va, \emptyseq}) \wedge S_\va(\ux_{\seq{0}}; \ut(\ux_\sh) + \uu_{\va, \seq{0}})
    $.
    Given an antichain $X \subset 2^{<\omega}$, the type $\set{S_\sh(\ux_{\sh}; \uu_{\sh, \nu}) : \nu \in X}$ is consistent.
    Hence $\uu_{\sh, \nu} = \uu_{\sh, \mu}$ for any $\nu, \mu \in 2^{<\omega}$ with $\nu \perp \mu$.
    Now Observation \ref{observation_terms_all_equal} tells us that $\uu_{\sh, \nu} = \uu_{\sh, \mu}$ for any $\nu, \mu \in 2^{<\omega}$.
    With this, we can easily verify that
    \begin{align}
        S'(\ux_\sh\ux_\emptyseq\ux_{\seq{0}}) \equiv S(\ux_\sh\ux_\emptyseq; \uu_\emptyseq) \wedge S(\ux_\sh\ux_{\seq{0}}; \uu_{\seq{0}}). \label{tag_equiv_of_stupid_c_sss}
    \end{align}
    It follows that $S'(\ux_\sh\ux_\emptyseq\ux_{\seq{0}})$ is a staged $C$-sequence-system with $\ux_\va := \ux_\emptyseq\ux_{\seq{0}}$.
    If the conjunction $\psi(\uxvec{}_\sh\uxvec{}_\emptyseq; \uz; \ud_\emptyseq) \wedge \psi(\uxvec{}_\sh\uxvec{}_{\seq{0}}; \uz; \ud_{\seq{0}})$ implies no finite disjunction of non-trivial linear dependencies in $\uxvec{}_\sh\uxvec{}_\emptyseq\uxvec{}_{\seq{0}}$ over $\VV$, we can use Lemma \ref{theorem_stronger_left_to_right_new} to find $\uv_\sh\uv_\emptyseq\uv_{\seq{0}}\ua$ such that
    $$
    (\MM, \theta) \models \psi_\theta(\uv_\sh\uv_\emptyseq; \ua; \ud_\emptyseq) \wedge \psi_\theta(\uv_\sh\uv_{\seq{0}}; \ua; \ud_{\seq{0}}) \wedge S'(\uv_\sh\uv_\emptyseq\uv_{\seq{0}})
    $$
    holds.
    With Observation \ref{observ_delta} above and the equivalence in (\ref{tag_equiv_of_stupid_c_sss}), this contradicts the assumption that the formula $\ux_\sh \in \VV \wedge \exists \ux_\va \in \VV : \delta_S(\uy) \wedge \psi_\theta(\ux_\sh\ux_\va; \uz; \uw) \wedge S(\ux_\sh\ux_\va; \uy)$ witnesses \ATP{} together with the tree $(\uu_\nu\ud_\nu : \nu \in 2^{<\omega})$.
\end{innerproof}
\end{subclaim}

\noindent For the rest of this section, suppose that $(\psi_0, S_0)$ is a $(\rk{}_\va, \deg{}_\va, \rk{}_\sh, \deg{}_\sh)$-minimal special witness (see Definition \ref{def_ranks_and_degs}).
Since $(\NN^4, <_\Lex)$ is well-ordered, such a minimal special witness exists because, by Claim \ref{lemma_ATP_if_special_witness_ATP}, there is at least one special witness.
We claim that $\psi_0(\uxvec{}_{0,\sh}\uxvec{}_{0,\va}; \uz; \uw_0)$ is as described in Lemma \ref{lemma_atp_step_11}.
Let $(\uu_{0,\nu}\ud_{0,\nu} : \nu \in \omega^{<\omega})$ be an $L_\theta$-strongly indiscernible tree such that
$$
\ux_{0,\sh} \in \VV \wedge \exists \ux_{0,\va} \in \VV : \delta_{S_0}(\uy_0) \wedge \psi_{0,\theta}(\ux_{0,\sh}\ux_{0,\va}; \uz; \uw_0) \wedge S_0(\ux_{0,\sh}\ux_{0,\va}; \uy_0)
$$
witnesses \ATP{} together with that tree (such a tree exists by Observation \ref{remark_str_ic_witness}).
Since $2^{<\omega} \subset \omega^{<\omega}$, the tree $(\uu_{0,\nu}\ud_{0,\nu} : \nu \in 2^{<\omega})$ is still strongly indiscernible.
Suppose, toward a contradiction, that there is an antichain $X \subset 2^{<\omega}$ such that
$$
\set{ \psi_0(\uxvec{}_{0,\sh}\uxvec{}_{0,\nu}; \uz; \ud_{0,\nu}) : \nu \in X }
$$
implies some finite disjunction of non-trivial linear dependencies in $\uxvec{}_{0,\sh}(\uxvec{}_{0,\nu} : \nu \in X)$ over $\VV$.
Using compactness we can assume that $X$ is finite.
Our goal is to construct a smaller special witness $(\psi', S')$, contradicting the minimality of $(\psi_0, S_0)$ with respect to the well-order induced by $(\rk{}_\va, \deg{}_\va, \rk{}_\sh, \deg{}_\sh)$.

\begin{subclaim} \label{lemma_remove_disj_1}
    Let $(\psi, S)$ be a special witness together with a strongly indiscernible tree $(\uu_\nu\ud_\nu : \nu \in \omega^{<\omega})$.
    Assume that there is a finite antichain $X \subset 2^{<\omega}$ such that $\set{\psi(\uxvec{}_\sh\uxvec{}_\nu; \uz; \ud_\nu) : \nu \!\in\! X}$ implies a finite disjunction of non-trivial linear dependencies in $\uxvec{}_\sh(\uxvec{}_\nu : \nu \in X)$ over $\VV$.
    Then there are:
    \begin{enumerate}[(i)]
        \item a non-trivial $\LK$-term $\lambda(\uxvec{}_\sh(\uxvec{}_\nu : \nu \in X))$ in which each placeholder $x_{*, k}^i$ (with $* \in \set{\sh} \cup X$ and $1 \leq k \leq |\ux_*|$) can appear only if it appears in $\psi(\uxvec{}_\sh\uxvec{}_\nu; \uz; \uw)$ for some $\nu \in X$,
        \item a formula $\varphi(y; (\uw_\nu : \nu \in X))$ algebraic in $y$,
    \end{enumerate}
    such that, for any antichain $Y \subset 2^{<\omega}$, there is a realization
    $$
    \uv_\sh(\uv_\nu : \nu \in Y)\ua \models \set{\delta_S(\uu_\nu) \wedge \psi_\theta(\ux_\sh\ux_\nu; \uz; \ud_\nu) \wedge S(\ux_\sh\ux_\nu; \uu_\nu) : \nu \in Y}
    $$
    with $\uv_\sh(\uv_\nu : \nu \in Y) \in \VV$ that satisfies $(\MM, \theta) \models \varphi(\lambda_\theta(\uv_\sh(\uv_\nu : \nu \in X')); (\ud_\nu : \nu \in X'))$ for any $X' \subseteq Y$ strongly isomorphic to the antichain $X$.
    Recall our \placeholderNotation{} (Definition \ref{def_placeholder_notation}): $\lambda_{\theta}(\ux_{\sh}(\ux_{\nu} : \nu \in X'))$ is $\lambda(\uxvec{}_{\sh}(\uxvec{}_{\nu} : \nu \in X'))$ with every placeholder variable $x^i_{*,k}$ replaced by $\theta^i(x_{*, k})$.
\begin{innerproof}
    We proceed as in Claim \ref{lemma_remove_disj_ez}, with two exceptions.
    In the second paragraph, set $\lambda(\uxvec{}_\sh(\uxvec{}_\nu : \nu \in X)) := \lambda_{i_0}(\uxvec{}_\sh(\uxvec{}_\nu : \nu \in X))$. Note that by restricting to the variables that actually appear, we can apply Fact \ref{lemma_hfour_fml} so that our placeholders appear in each $\lambda_k(\uxvec{}_{\sh}(\uxvec{}_{\nu} : \nu \in X))$ as described in (i).
    At the beginning of the last paragraph, we choose $\uv'_\sh(\uv'_{\nu} : \nu \in Z)\ua$ with $\uv'_\sh(\uv'_{\nu} : \nu \in Z) \in \VV$ to be a realization of
    $$
    \set{ \psi_{\theta}(\ux_{\sh}\ux_{\nu}; \uz; \ud_{\nu}) \wedge S(\ux_\sh\ux_\nu; \uu_\nu) : \nu \in Z}
    $$
    instead of just $\set{\psi_{\theta}(\ux_{\sh}\ux_{\nu}; \uz; \ud_{\nu}) : \nu \in Z}$, and we do the same later for $\uv_\sh(\uv_{\nu} : \nu \in Y)\ua$.
\end{innerproof}
\end{subclaim}

\noindent Note that, as in Claim \ref{lemma_remove_disj_ez}, we need $(\uu_\nu\ud_\nu : \nu \in \omega^{<\omega})$ to be indexed by $\omega^{<\omega}$ in the Claim above in order to apply some Ramsey theory.

In the following Claim, the special witness $(\psi_0, S_0)$, the tree $(\uu_{0,\nu}\ud_{0,\nu} : \nu \in \omega^{<\omega})$, and the antichain $X \subset 2^{<\omega}$ are the ones fixed above, before Claim \ref{lemma_remove_disj_1}.

\begin{subclaim} \label{lemma_new_remove_disj}
    We can find an $L$-formula $\psi(\uxvec{}_\sh\uxvec{}_\va; \uz; \uw)$, a parametrized staged $C$-sequence-system $S(\ux_\sh\ux_\va; \uy)$, an $\LKThe$-equation $E(\ux_\sh\ux_\va; y^*)$, and another strongly $L_\theta$-indiscernible tree $(\uu_\nu u^*_{\nu}\ud_\nu : \nu \in 2^{<\omega})$ such that:
    \begin{enumerate}[(i)]
        \item $E(\ux_\sh\ux_\va; y^*)$ is non-trivial and bounded by the parametrized staged $C$-sequence-system $S$.
        \item The formula $\ux_\sh \in \VV \wedge \exists \ux_\va \in \VV : \delta_{S}(\uy) \wedge \psi_\theta(\ux_\sh\ux_\va; \uz; \uw) \wedge S(\ux_\sh\ux_\va; \uy) \wedge E(\ux_\sh\ux_\va; y^*)$ witnesses \ATP{} together with the tree.
        \item We have $(\rk_\va(S), \deg_\va(S)) = (\rk_\va(S_0), \deg_\va(S_0))$.
        Furthermore, if $E(\ux_\sh\ux_\va; y^*)$ is trivial in $\ux_\va$ (and therefore non-trivial in $\ux_\sh$), then also $(\rk_\sh(S), \deg_\sh(S)) = (\rk_\sh(S_0), \deg_\sh(S_0))$.
    \end{enumerate}
\begin{innerproof}
    Apply Claim \ref{lemma_remove_disj_1} to the special witness $(\psi_0, S_0)$ and the tree $(\uu_{0, \nu} \ud_{0, \nu} : \nu \in \omega^{<\omega})$ to obtain a non-trivial $\LK$-term $\lambda(\uxvec{}_{0, \sh}(\uxvec{}_{0, \nu} : \nu \in X))$ and a formula $\varphi(y; (\uw_{0, \nu} : \nu \in X))$, as described in that Claim.
    Write $X = \set{\mu_1, \dots, \mu_{r}, \mu}$ with all $\mu$'s distinct.
    Without loss, we can distinguish between the following three cases:
    \begin{enumerate}[(a)]
        \item The $\LK$-term $\lambda(\uxvec{}_{0,\sh}(\uxvec{}_{0,\nu} \!: \nu \in X))$ only depends on the tuple $\uxvec{}_{0,\sh}$, i.e., we can write $\lambda_\theta(\ux_{0,\sh}(\ux_{0,\nu} : \nu \in X)) =: \lambda^*(\ux_{0,\sh})$.
        In this case, we define
        $$
        \psi(\uxvec{}_\sh\uxvec{}_\va; \uz; \uw) := \psi_0(\uxvec{}_{0,\sh}\uxvec{}_{0,\va}; \uz; \uw_0)\quad \text{and}\quad S(\ux_\sh\ux_\va; \uy) := S_0(\ux_{0,\sh}\ux_{0,\va}; \uy_0).$$
        This also means that we implicitly set $\ux_\sh := \ux_{0,\sh}$, and so on.
        Let $Y \subset 2^{<\omega}$ be any finite antichain.
        Since $\set{\mu_1, \dots, \mu_{r}} \cup \mu^\frown Y$ is an antichain and $\set{\mu_1, \dots, \mu_{r}, \mu ^\frown \nu}$ is strongly isomorphic to $X$ for any $\nu \in 2^{<\omega}$, Claim \ref{lemma_remove_disj_1} tells us that
        $$
        \Set{ \psi_\theta(\ux_\sh\ux_\nu; \uz; \ud_{0,\nu}) \wedge S(\ux_\sh\ux_\nu; \uu_{0,\nu}) \wedge \varphi(\lambda^*(\ux_\sh); \ud_{0,\mu_1} \dots \ud_{0,\mu_r} \ud_{0,\nu}) : \nu \in \mu^\frown Y}
        $$
        is consistent.
        Since $\psi_\theta(\ux_\sh\ux_\nu; \uz; \ud_{0,\nu}) \wedge S(\ux_\sh\ux_\nu; \uu_{0,\nu}) \wedge \psi_\theta(\ux_\sh\ux_\eta; \uz; \ud_{0,\eta}) \wedge S(\ux_\sh\ux_\eta; \uu_{0,\eta})$ is inconsistent for any $\nu \triangleleft \eta$, we can now use Lemma \ref{lemma_remove_alg_set} and Observation \ref{observ_delta} to see that the formula
        $$
        \ux_\sh \in \VV \wedge \exists \ux_\va \in \VV : \delta_S(\uy) \wedge \psi_\theta(\ux_\sh\ux_\va; \uz; \uw) \wedge S(\ux_\sh\ux_\va; \uy) \wedge \lambda^*(\ux_\sh) = y^*
        $$
        witnesses \ATP{} together with some strongly indiscernible tree $(\uu_\nu u^*_{\nu}\ud_\nu : \nu \in 2^{<\omega})$.
        Note that the equation $\lambda^*(\ux_\sh) = y^*$ must be bounded by $S$, since $\lambda(\uxvec{}_\sh(\uxvec{}_\nu : \nu \in X))$ can only use the variables that appear in $\bigwedge\nolimits_{\nu \in X} \psi(\uxvec{}_\sh\uxvec{}_\nu; \uz; \uw_\nu)$ (see (i) of Claim \ref{lemma_remove_disj_1}) and $\psi(\uxvec{}_\sh\uxvec{}_\va; \uz; \uw)$ is bounded by $S$.
        \item The $\LKThe$-term $\lambda_\theta(\ux_{0,\sh}(\ux_{0,\nu} : \nu \in X))$ depends on $\ux_{\mu}$ and $C$ is transcendental.
        In this case, write $\lambda'_{\theta}(\ux_{0,\sh}\ux_{0,\mu}) + \lambda''_{\theta}(\ux_{0,\mu_1} \dots \ux_{0,\mu_r}) := \lambda_\theta(\ux_{0,\sh}(\ux_{0,\nu} : \nu \in X))$.
        Let $Y \subset 2^{<\omega}$ be any antichain.
        As in (a), we see that the type
        $$
        \Set{\parbox{9.5cm}{$\psi_{0,\theta}(\ux_{0,\sh}\ux_{0,\nu}; \uz; \ud_{0,\nu}) \wedge S_0(\ux_{0,\sh}\ux_{0,\nu}; \uu_{0,\nu})\\{}\hspace{30pt} \wedge \ \varphi(\lambda'_{\theta}(\ux_{0,\sh}\ux_{0,\nu}) + \lambda''_{\theta}(\ux_{0,\mu_1} \dots \ux_{0,\mu_r}); \ud_{0,\mu_1} \dots \ud_{0,\mu_r} \ud_{0,\nu})$} :\ \nu \in \mu^\frown Y}
        $$
        has a realization $\uv_{0,\sh} \uv_{\mu_1} \dots \uv_{\mu_r}(\uv_{\nu} : \nu \in \mu^\frown Y)\ua$.
        Let $x_{\sh, \li, 0}$ be a variable that does not appear in $\ux_{0, \sh}\ux_{0, \va}$.
        We set $\ux_{\sh} := x_{\sh, \li, 0}\ux_{0,\sh}$,
        $$
        \ux_\va := \ux_{0, \va}, \quad\psi(\uxvec{}_\sh\uxvec{}_\va; \uz; \uw) := \psi_0(\uxvec{}_{0,\sh}\uxvec{}_{0,\va}; \uz; \uw_0), \quad S(\ux_\sh\ux_\va; \uy) := S_0(\ux_{0,\sh}\ux_{0,\va}; \uy_0),
        $$
        and $\lambda^*(\ux_\sh\ux_\va) = \lambda'_{\theta}(\ux_{0,\sh}\ux_{0,\va}) + x_{\sh, \li, 0}$.
        As in (a), this also means that we set $\uw := \uw_0$ and $\uy := \uy_0$.
        Note that $S(\ux_\sh\ux_\va; \uy)$ is still a parametrized staged $C$-sequence-system, by Observation \ref{obs_add_to_shared} applied with the added ``regular'' system $S_*(x_{\sh, \li, 0}) := \top$.
        Set $\uv_\sh := \lambda''_{\theta}(\uv_{\mu_1} \dots \uv_{\mu_r})\uv_{0,\sh}$.
        One can now check that $\uv_\sh(\uv_{\nu} : \nu \in \mu^\frown Y)\ua$ is a realization of
        $$
        \Set{ \psi_\theta(\ux_\sh\ux_\nu; \uz; \ud_{0,\nu}) \wedge S(\ux_\sh\ux_\nu; \uu_{0,\nu}) \wedge \varphi(\lambda^*(\ux_\sh\ux_\nu); \ud_{0,\mu_1} \dots \ud_{0,\mu_r} \ud_{0,\nu}) : \nu \in \mu^\frown Y}.
        $$
        As in (a), conclude that the formula
        $$
        \ux_\sh \in \VV \wedge \exists \ux_\va \in \VV : \delta_{S}(\uy) \wedge \psi_\theta(\ux_\sh\ux_\va; \uz; \uw) \wedge S(\ux_\sh\ux_\va; \uy) \wedge \lambda^*(\ux_\sh\ux_\va) = y^*
        $$
        witnesses \ATP{} together with some strongly indiscernible tree $(\uu_\nu u^*_{\nu}\ud_\nu : \nu \in 2^{<\omega})$.
        With similar arguments as in (a), one verifies that the equation $\lambda^*(\ux_\sh\ux_\va) = y^*$ is bounded by the parametrized staged $C$-sequence-system $S$.
        \item The $\LK$-term $\lambda(\uxvec{}_{0,\sh}(\uxvec{}_{0,\nu} : \nu \in X))$ depends on $\ux_{\mu}$ and $C$ is an algebraic kernel configuration.
        In this case, we write $\lambda'_{\theta}(\ux_{0,\sh}\ux_{0,\mu}) + \lambda''_{\theta}(\ux_{0,\mu_1} \dots \ux_{0,\mu_r}) := \lambda_\theta(\ux_{0,\sh}(\ux_{0,\nu} : \nu \in X))$ as in (b).
        Let \hbox{$F := \Kp{0<C<\infty}$} and define $\uv_{0,\sh}\uv_{\mu_1} \dots \uv_{\mu_r}(\uv_{\nu} : \nu \in \mu^\frown Y)\ua$ as in (b).
        For each $f \in F$, let $x_f$ and $y_f$ be new variables.
        Set $\ux_\sh := (x_f : f \in F)\ux_{0,\sh}$, $\ux_\va := \ux_{0, \va}$, and $\uy := (y_f : f \in F)\uy_0$.
        We also define
        $$
        \psi(\uxvec{}_\sh\uxvec{}_\va; \uz; \uw) := \psi_0(\uxvec{}_{0,\sh}\uxvec{}_{0,\va}; \uz; \uw_0),\quad\!\!\!\! S(\ux_\sh\ux_\va; \uy) := \bigwedge\nolimits_{f\in F} f^C[\theta](x_f) = y_f \wedge S_0(\ux_{0,\sh}\ux_\va; \uy_0)
        $$
        and $\lambda^*(\ux_\sh\ux_\va) := \lambda'_{\theta}(\ux_{0,\sh}\ux_{0,\va}) + \sum\nolimits_{f\in F} x_f$.
        One can again use Observation \ref{obs_add_to_shared} to verify that $S(\ux_\sh\ux_\va; \uy)$ is a parametrized staged $C$-sequence-system and that $\uzero\uu_{0,\nu}$ is compatible with it for every $\nu \in 2^{<\omega}$.
        Define $\uv_\sh := (\pi_{\Ker(f^C)}(\lambda''_{\theta}(\uv_{\mu_1} \dots \uv_{\mu_r})) : f \in F)\uv_{0,\sh}$.
        Now the tuple $\uv_\sh(\uv_{\nu} : \nu \in \mu^\frown Y)\ua$ is a realization of
        $$
        \Set{\psi_\theta(\ux_\sh\ux_\nu; \uz; \ud_{0,\nu}) \wedge S(\ux_\sh\ux_\nu; \uzero\uu_{0,\nu}) \wedge \varphi(\lambda^*(\ux_\sh\ux_\nu); \ud_{0,\mu_1} \dots \ud_{0,\mu_r} \ud_{0,\nu}) : \nu \in \mu^\frown Y}
        $$
        (recall that $\VV = \bigoplus\nolimits_{f\in F} \Ker(f^C)$ for $F = \Kp{0<C<\infty}$ if $C$ is algebraic).
        Now proceed as in (b).
        For the boundedness of the equation $\lambda^*(\ux_\sh\ux_\va) = y^*$ by $S$, also note that, for each $f \in F$, the subterm $\theta^i(x_f)$ appears in $\lambda^*(\ux_\sh\ux_\va)$ only with $i = 0 < \deg(f)$.
    \end{enumerate}
    With $E(\ux_\sh\ux_\va; y^*) := \lambda^*(\ux_\sh\ux_\va) = y^*$, points (i) and (ii) were proved in all cases.
    Point (iii) can also be verified because, in cases (b) and (c), we obtained $S$ by applying Observation \ref{obs_add_to_shared}, which preserves $\rk_\va$ and $\deg_\va$, and because we set $S = S_0$ in case (a).
\end{innerproof}
\end{subclaim}

\noindent Since we no longer need the antichain $X$ fixed before Claim \ref{lemma_remove_disj_1}, we may use the letter $X$ to denote other antichains from now on.
Let $\psi(\uxvec{}_\sh\uxvec{}_\va; \uz; \uw)$, $S(\ux_\sh\ux_\va; \uy)$, $E(\ux_\sh\ux_\va; y^*)$, and $(\uu_\nu u^*_{\nu}\ud_\nu : \nu \in 2^{<\omega})$ be as in the statement of Claim \ref{lemma_new_remove_disj}.
This means that for any finite antichain $X \subset 2^{<\omega}$, we have a tuple $\uv_\sh(\uv_\nu : \nu \in X)\ua$ with $\uv_\sh(\uv_\nu : \nu \in X) \in \VV$ such that
    \begin{align}
        (\MM, \theta) \models \bigwedge\nolimits_{\nu \in X} \delta_S(\uu_\nu) \wedge \psi_\theta(\uv_\sh \uv_\nu; \ua; \ud_\nu) \wedge S(\uv_\sh\uv_\nu; \uu_\nu) \wedge E(\uv_\sh\uv_\nu; u^*_{\nu}). \label{tag_realizatiion_of_X}
    \end{align}
Using our Transformation Lemma for staged $C$-sequence-systems (Lemma \ref{lemma_trafo_for_staged_c_ss}), we see that the conjunction $\delta_{S}(\uy) \wedge S(\ux_\sh\ux_\va; \uy) \wedge E(\ux_\sh\ux_\va; y^*)$ is transformable into a formula of the form
$$
    \varphi'(\uy y^*) \wedge S'(\ux'_\sh\ux'_\va; \umu{}'(\uy y^*))
$$
where $S'(\ux'_\sh\ux'_\va; \uy')$ is another parametrized staged $C$-sequence-system, and $(\varphi', \umu{}')$ is a pair compatible with $S'$.
This means that $\umu{}'(\uy y^*)$ is a tuple of $\LRC$-terms, and $\varphi'(\uy y^*)$ is a conjunction of $\LRC$-equations that implies that $\umu{}'(\uy y^*)$ is compatible with $S'$.
By (i) of Lemma \ref{lemma_trafo_for_staged_c_ss}, this transformability is witnessed by tuples of the form
\begin{align}
    \underline{\nu}(\ux_\sh\ux_\va) = \underline{\nu}{}_\sh(\ux_\sh)\underline{\nu}{}_\va(\ux_\sh\ux_\va) \quad \text{and} \quad \utau{}(\ux'_\sh\ux'_\va; \uy y^*) := \utau_\sh(\ux'_\sh; \uy y^*)\utau_\va(\ux'_\sh\ux'_\va; \uy y^*). \label{tag_special_witnesses_of_trafo}
\end{align}
Here the partition $\underline{\nu}(\ux_\sh\ux_\va) = \underline{\nu}{}_\sh(\ux_\sh)\underline{\nu}{}_\va(\ux_\sh\ux_\va)$ corresponds to the partition $\ux' = \ux'_\sh\ux'_\va$, and so on.
To avoid confusion, we note that non-underlined $\nu$'s will still be used to denote elements of trees.

\begin{subobservation} \label{obsv_term_coonnst}
    The tuple of $\LKThe(\VV)$-terms $\utau_\sh(\ux'_\sh; \uu_\nu u^*_\nu)$ does not depend on $\nu \in 2^{<\omega}$.
\begin{innerproof}
    By the definition of transformability, we can write $\utau_\sh(\ux'_\sh; \uy y^*) = \ulambda(\ux'_\sh) + \ur(\uy y^*)$, where $\ulambda(\ux'_\sh)$ is a tuple of $\LKThe$-terms and $\ur(\uy y^*)$ is a tuple of $\LRC$-terms.
    Take some two-element antichain $X = \set{\nu_1, \nu_2}$ and let $\uv_\sh(\uv_\nu : \nu \in X)\ua$ be given as in (\ref{tag_realizatiion_of_X}).
    Recall that the conjunction $\delta_S(\uy) \wedge S(\ux_\sh\ux_\va; \uy) \wedge E(\ux_\sh\ux_\va; y^*)$ is transformable into $\varphi'(\uy y^*) \wedge S'(\ux'_\sh\ux'_\va; \umu{}'(\uy y^*))$, witnessed by the tuples in (\ref{tag_special_witnesses_of_trafo}).
    By the first implication in the definition of transformability (see (i) of Definition \ref{def_transfo}) and the special form of these witnesses, we obtain
    $$
    \ulambda(\underline{\nu}{}_\sh(\uv_\sh)) + \ur(\uu_{\nu_1} u^*_{\nu_1}) = \utau_\sh(\underline{\nu}{}_\sh(\uv_\sh); \uu_{\nu_1} u^*_{\nu_1}) = \uv_\sh = \utau_\sh(\underline{\nu}{}_\sh(\uv_\sh); \uu_{\nu_2}u^*_{\nu_2}) = \ulambda(\underline{\nu}{}_\sh(\uv_\sh)) + \ur(\uu_{\nu_2} u^*_{\nu_2}),
    $$
    and hence $\ur(\uu_{\nu_1} u^*_{\nu_1}) = \ur(\uu_{\nu_2} u^*_{\nu_2})$.
    By Observation \ref{observation_terms_all_equal}, it follows that $\ur(\uu_{\nu} u^*_{\nu})$, and therefore the tuple $\utau_\sh(\ux'_\sh; \uu_\nu u^*_\nu)$, does not depend on $\nu \in 2^{<\omega}$.
\end{innerproof}
\end{subobservation}

\noindent By Lemma \ref{lemma_main_trafo_for_staged}, we can find an $L$-formula $\psi'(\uxvec{}'_\sh\uxvec{}'_\va; \uz; \uw\tiluw)$ bounded by $S'$ and a tuple of $\LRC$-terms $\ut(\uy y^*)$ such that
    \begin{align}
        \psi'_\theta(\ux'_\sh\ux'_\va; \uz; \uw\ut(\uy y^*)) \equiv \psi_\theta(\utau(\ux'_\sh\ux'_\va; \uy y^*); \uz; \uw) \label{tag_equality_with_better_fml}
    \end{align}
holds modulo $T_\theta \cup \set{\text{``$\theta$ is $C$-image-complete''}} \cup \set{S'(\ux'_\sh\ux'_\va; \umu{}'(\uy y^*))}$.
For all $\nu \in 2^{<\omega}$, define the tuples $\ud'_\nu := \ud_\nu\ut(\uu_\nu u^*_{\nu})$ and $\uu'_\nu := \umu{}'(\uu_\nu u^*_{\nu})$.
The tree $(\uu'_\nu\ud'_\nu : \nu \in 2^{<\omega})$ is clearly strongly indiscernible.

\begin{subclaim} \label{lemma_special_witness_smaller}
    The pair $(\psi', S')$ is a special witness together with the tree $(\uu'_\nu\ud'_\nu : \nu \in 2^{<\omega})$.
\begin{innerproof}
    By the definition of a special witness (Definition \ref{def_special_witness_3}), we need to show that the formula $\ux'_\sh \in \VV \wedge \exists \ux'_\va \in \VV : \delta_{S'}(\uy') \wedge \psi'_\theta(\ux'_\sh\ux'_\va; \uz; \uw') \wedge S'(\ux'_\sh\ux'_\va; \uy')$, with $\uw' := \uw\tiluw$, witnesses \ATP{} in $\ux'_\sh\uz$ together with the tree $(\uu'_\nu\ud'_\nu : \nu \in 2^{<\omega})$.
    First, we show that
    $$
    \set{\ux'_\sh \in \VV \wedge \exists \ux'_\nu \in \VV : \delta_{S'}(\uu'_\nu) \wedge \psi'_\theta(\ux'_\sh\ux'_\nu; \uz; \ud'_\nu) \wedge S'(\ux'_\sh\ux'_\nu; \uu'_\nu) : \nu \in X}
    $$
    is consistent for every finite antichain $X \subset 2^{<\omega}$.
    Let $\uv_\sh(\uv_\nu : \nu \in X)\ua$ be as in (\ref{tag_realizatiion_of_X}) for this finite antichain $X$.
    Since $\delta_S(\uy) \wedge S(\ux_\sh\ux_\va; \uy) \wedge E(\ux_\sh\ux_\va; y^*)$ is transformable into the conjunction $\varphi'(\uy y^*) \wedge S'(\ux'_\sh\ux'_\va; \umu{}'(\uy y^*))$, we obtain
    $$
    (\MM, \theta) \models \bigwedge\nolimits_{\nu \in X} \psi_\theta(\utau(\underline{\nu}{}(\uv_\sh\uv_\nu); \uu_\nu u^*_{\nu}); \ua; \ud_\nu) \wedge \varphi'(\uu_\nu u^*_{\nu}) \wedge S'(\underline{\nu}{}(\uv_\sh\uv_\nu); \uu'_\nu)
    $$
    by (i) of Definition \ref{def_transfo}.
    We now use the special form of $\underline{\nu}(\ux_\sh\ux_\va)$ in (\ref{tag_special_witnesses_of_trafo}), together with the equivalence in (\ref{tag_equality_with_better_fml}), the definition of our tree $(\uu'_\nu\ud'_\nu : \nu \in 2^{<\omega})$, and the fact that $\varphi'(\uy y^*)$ implies $\delta_{S'}(\umu{}'(\uy y^*))$ to see that
    \begin{align}
        (\MM, \theta) \models \bigwedge\nolimits_{\nu \in X} \delta_{S'}(\uu'_{\nu}) \wedge \psi'_\theta(\uv'_\sh\uv'_\nu; \ua; \ud'_\nu) \wedge S'(\uv'_\sh\uv'_\nu; \uu'_\nu). \label{tag_realizatiion_of_X_2}
    \end{align}
    holds for the tuples $\uv'_\sh := \underline{\nu}{}_\sh(\uv_\sh)$ and $\uv'_\nu := \underline{\nu}{}_\va(\uv_\sh\uv_\nu)$.
    This shows that the partial type $$\set{\ux'_\sh \in \VV \wedge \exists \ux'_\va \in \VV : \delta_{S'}(\uu'_\nu) \wedge \psi'_\theta(\ux'_\sh\ux'_\va; \uz; \ud'_\nu) \wedge S'(\ux'_\sh\ux'_\va; \uu'_\nu) : \nu \in X}$$ is consistent.

    Now, we show that any finite $X \subset 2^{<\omega}$ must be an antichain if the partial type
    $$
    \set{ \ux'_\sh \in \VV \wedge \exists \ux'_\va \in \VV : \delta_{S'}(\uu'_\nu) \wedge \psi'_\theta(\ux'_\sh\ux'_\va; \uz; \ud'_\nu) \wedge S'(\ux'_\sh\ux'_\va; \uu'_\nu) : \nu \in X}
    $$
    is consistent.
    Assuming that this partial type is consistent, there is $\uv'_\sh(\uv'_\nu : \nu \in X)\ua$ such that (\ref{tag_realizatiion_of_X_2}) holds for this $X$.
    First, note that the paragraph above, applied to the antichain $\set{\nu}$, yields $(\MM, \theta) \models \varphi'(\uu_\nu u^*_\nu)$ for any $\nu \in 2^{<\omega}$.
    Hence, with the definitions of the $\ud'_\nu$'s and the equivalence in (\ref{tag_equality_with_better_fml}), we obtain
    $$
    (\MM, \theta) \models \bigwedge\nolimits_{\nu \in X} \psi_\theta(\utau(\uv'_\sh\uv'_\nu; \uu_\nu u^*_{\nu}); \ua; \ud_\nu) \wedge \varphi'(\uu_\nu u^*_{\nu}) \wedge S'(\uv'_\sh\uv'_\nu; \uu'_\nu).
    $$
    Since $\delta_S(\uy) \wedge S(\ux_\sh\ux_\va; \uy) \wedge E(\ux_\sh\ux_\va; y^*)$ is transformable into $\varphi'(\uy y^*) \wedge S'(\ux'_\sh\ux'_\va; \umu{}'(\uy y^*))$, this implies
    \begin{align}
    (\MM, \theta) \models \bigwedge\nolimits_{\nu \in X} \delta_S(\uu_\nu) &\wedge \psi_\theta(\utau_\sh(\uv'_\sh; \uu_\nu u^*_\nu)\utau_\va(\uv'_\sh\uv'_\nu; \uu_\nu u^*_\nu); \ua; \ud_\nu) \notag  \\
    & \hspace{-20pt} \wedge S(\utau_\sh(\uv'_\sh; \uu_\nu u^*_\nu)\utau_\va(\uv'_\sh\uv'_\nu; \uu_\nu u^*_\nu); \uu_\nu) \wedge E(\utau_\sh(\uv'_\sh; \uu_\nu u^*_\nu)\utau_\va(\uv'_\sh\uv'_\nu; \uu_\nu u^*_\nu); u^*_\nu).
    \end{align}
    To see this, use (ii) of Definition \ref{def_transfo} in combination with $\uu'_\nu := \umu{}'(\uu_\nu u^*_\nu)$ and the special form of $\utau(\ux'_\sh\ux'_\va; \uy y^*)$ in (\ref{tag_special_witnesses_of_trafo}).
    Recall that the formula
    $$
    \ux_\sh \in \VV \wedge \exists \ux_\va \in \VV : \delta_{S}(\uy) \wedge \psi_\theta(\ux_\sh\ux_\va; \uz; \uw) \wedge S(\ux_\sh\ux_\va; \uy) \wedge E(\ux_\sh\ux_\va; y^*)
    $$
    witnesses \ATP{} together with the tree $(\uu_\nu u^*_{\nu}\ud_\nu : \nu \in 2^{<\omega})$ by definition (see (ii) of Claim \ref{lemma_new_remove_disj}).
    Since $\utau_\sh(\uv'_\sh; \uu_\nu u^*_\nu)$ does not depend on $\nu \in 2^{<\omega}$ (see Observation \ref{obsv_term_coonnst}), $X$ must be an antichain.
\end{innerproof}
\end{subclaim}

\begin{subclaim} \label{lemma_smaller_witness}
    The following lexicographic inequality holds:
    $$
    (\rk_\va(S'), \deg_\va(S'), \rk_\sh(S'), \deg_\sh(S')) <_\Lex (\rk_\va(S_0), \deg_\va(S_0), \rk_\sh(S_0), \deg_\sh(S_0)).
    $$
\begin{innerproof}
    By (i) of Claim \ref{lemma_new_remove_disj}, the equation $E(\ux_\sh\ux_\va; y^*)$ is non-trivial and bounded by $S$.
    By (iii) of Claim \ref{lemma_new_remove_disj}, we have $(\rk_\va(S), \deg_\va(S)) = (\rk_\va(S_0), \deg_\va(S_0))$.
    If $E(\ux_\sh\ux_\va; y^*)$ is non-trivial in $\ux_\va$, then our Transformation Lemma for staged $C$-sequence-systems immediately yields
    $$
    (\rk_\va(S'), \deg_\va(S')) <_\Lex (\rk_\va(S), \deg_\va(S)) = (\rk_\va(S_0), \deg_\va(S_0)).
    $$
    If $E(\ux_\sh\ux_\va; y^*)$ is trivial in $\ux_\va$, then (iii) of Claim \ref{lemma_new_remove_disj} yields $(\rk_\sh(S), \deg_\sh(S)) = (\rk_\sh(S_0), \deg_\sh(S_0))$.
    In this case, our Transformation Lemma for staged $C$-sequence-systems ensures
    $$
    (\rk_\va(S'), \deg_\va(S')) = (\rk_\va(S), \deg_\va(S)) \quad \text{and}\quad (\rk_\sh(S'), \deg_\sh(S')) <_\Lex (\rk_\sh(S), \deg_\sh(S)),
    $$
    so the desired inequality follows as well.
\end{innerproof}
\end{subclaim}

\noindent Combining Claim \ref{lemma_special_witness_smaller} and Claim \ref{lemma_smaller_witness}, we obtain a contradiction to the minimality of $(\psi_0, S_0)$.
Hence, our assumption that there is an antichain $X \subset 2^{<\omega}$ such that
$$
\set{ \psi_0(\uxvec{}_{0,\sh}\uxvec{}_{0,\nu}; \uz; \ud_{0,\nu}) : \nu \in X }
$$
implies some finite disjunction of non-trivial linear dependencies in $\uxvec{}_{0,\sh}(\uxvec{}_{0,\nu} : \nu \in X)$ over $\VV$ must be wrong.
This, in combination with Claim \ref{lemma_implies_finite_disj}, completes the proof of Lemma \ref{lemma_atp_step_11}.
\end{proof}

\bibliographystyle{alphaurl}
\bibliography{sample}

\Addresses

\end{document}